\documentclass{amsart}
\usepackage{amssymb, graphics, color, enumitem}
\usepackage[all]{xy}

 \usepackage{tikz}

 \usepackage[applemac]{inputenc}
\usepackage[cyr]{aeguill}

\usepackage[margin = 1in]{geometry}

\newtheorem{lemma}{Lemma}[section]
\newtheorem{proposition}[lemma]{Proposition}
\newtheorem{corollary}[lemma]{Corollary}
\newtheorem{theorem}[lemma]{Theorem}
\newtheorem*{theorem_int}{Theorem}
\newtheorem{example}[lemma]{Example}
\newtheorem{definition}[lemma]{Definition}
\newtheorem{remark}[lemma]{Remark}

\newtheorem*{Acknowledgement}{Acknowledgements}

\newtheorem{Theorem}{Theorem}
\newtheorem*{Remark}{Remark}

\newtheorem{Corollary}[Theorem]{Corollary}

\newcommand\cf{cf\@. }

\newcommand\pa{ \partial}

\newcommand\bbC{\mathbb C}

\newcommand\bbN{\mathbb N}
\newcommand\bbP{\mathbb P}
\newcommand\bbQ{\mathbb Q}
\newcommand\bbR{\mathbb R}
\newcommand\bbS{\mathbb S}

\newcommand\bbZ{\mathbb Z}

\renewcommand\Im{\operatorname{Im}}

\usepackage{color}

\newcommand\tX{\widetilde{X}}

\newcommand\hX{\widehat{X}}
\newcommand\hH{\widehat{H}}

\newcommand\hD{\widehat{D}}
\newcommand\hY{\widehat{Y}}
\newcommand\CI{\mathcal{C}^{\infty}}

\newcommand\Diff{\operatorname{Diff}}
\newcommand\Ric{\operatorname{Ric}}

\newcommand\cC{\mathcal{C}}

\newcommand\cE{\mathcal{E}}

\newcommand\db{\overline{\pa}}

\newcommand\cV{\mathcal{V}}
\newcommand\cU{\mathcal{U}}
\newcommand\cI{\mathcal{I}}

\newcommand\tV{\widetilde{V}}

\newcommand\pr{\operatorname{pr}}

\newcommand\Id{\operatorname{Id}}

\newcommand\homega{\widehat{\omega}}

\renewcommand\sc{\operatorname{sc}}

\newcommand\cH{\mathcal{H}}

\newcommand\hphi{\hat{\phi}}

\newcommand\cN{\mathcal{N}}

\newcommand\AC{\operatorname{AC}}

\newcommand\SU{\operatorname{SU}}
\newcommand\Sp{\operatorname{Sp}}

\newcommand\St{^{\mathsf{S}}\!}

\newcommand\QAC{\operatorname{QAC}}
\newcommand\QCyl{\operatorname{Qb}}
\newcommand\QFB{\operatorname{QFB}}
\newcommand\ALE{\operatorname{ALE}}
\newcommand\QALE{\operatorname{QALE}}

\newcommand\bV{\overline{V}}

\newcommand\bz{\overline{z}}
\newcommand\tomega{\widetilde{\omega}}
\newcommand\sing{\operatorname{sing}}

\newcommand\Hess{\operatorname{Hess}}

\newcommand\Fix{\operatorname{Fix}}
\newcommand\vol{\operatorname{vol}}

\newcommand\tSigma{\widetilde{\Sigma}}
\newcommand\tg{\widetilde{g}}
\newcommand\cK{\mathcal{K}}
\newcommand\GL{\operatorname{GL}}
\newcommand\cX{\mathcal{X}}
\newcommand\hvarphi{\widehat{\varphi}}

\begin{document}
\title[QAC Calabi-Yau manifolds]
{Quasi-asymptotically conical Calabi-Yau manifolds}

\author{Ronan J.~Conlon}
\address{Department of Mathematics and Statistics, Florida International University, Miami, FL 33199, USA}
\email{rconlon@fiu.edu}

\author{Anda Degeratu}
\address{University of Stuttgart}
\email{anda.degeratu@mathematik.uni-stuttgart.de
}

\author{Fr\'ed\'eric Rochon}
\address{Département de Mathématiques, Universit\'e du Qu\'ebec \`a Montr\'eal}
\email{rochon.frederic@uqam.ca}

\maketitle
\begin{center}
\textit{with an appendix by Ronan J.~Conlon, Frédéric Rochon, and Lars Sektnan}
\end{center}

\begin{abstract}
We construct new examples of quasi-asymptotically conical ($\QAC$) Calabi-Yau manifolds that are not quasi-asymptotically locally Euclidean ($\QALE$).  We do so by first providing a natural compactification of $\QAC$-spaces by manifolds with fibred corners and by giving a definition of $\QAC$-metrics in terms of an associated Lie algebra of smooth vector fields on this compactification.  Thanks to this compactification and the Fredholm theory for elliptic operators on $\QAC$-spaces developed by the second author and Mazzeo, we can in many instances obtain Kähler $\QAC$-metrics having Ricci potential decaying sufficiently fast at infinity.  This allows us to obtain  $\QAC$ Calabi-Yau metrics in the Kähler classes of these metrics by solving a corresponding complex Monge-Ampère equation.
\end{abstract}

\tableofcontents

\section*{Introduction}

A complete Kähler manifold $(X,g,J)$ of complex dimension $m$ is Calabi-Yau if it is Ricci-flat and has a nowhere vanishing parallel holomorphic volume form $\Omega_X\in H^0(X;K_X)$.  In this case, the holonomy of $(X,g)$ is contained in $\SU(m)$ and we say that $g$ is a Calabi-Yau metric.  Since the resolution of the Calabi conjecture by Yau \cite{Yau1978}, we know that a compact Kähler manifold admits a Calabi-Yau metric if and only if its canonical line bundle is trivial, in which case every Kähler class contains a unique Calabi-Yau metric obtained by solving a complex Monge-Ampère equation.  As subsequently shown by Tian and Yau in \cite{Tian-Yau1990,Tian-Yau1991}, on non-compact complete Kähler manifolds, one can also obtain many examples of complete Calabi-Yau metrics by solving a corresponding complex Monge-Ampère equation, but the triviality of the canonical line bundle is not the only required condition.  One also needs to take into account the asymptotic behavior of the metric at infinity. 

For example, if $\Gamma\subset \SU(m)$ is a finite subgroup acting freely on $\bbC^m\setminus \{0\}$ and if $X\to \bbC^m/\Gamma$ is a crepant resolution, then, as pointed out by Joyce \cite{Joyce2001}, the results of Tian-Yau or of Bando-Kobayashi \cite{BK1988,BK1990}, combined with \cite{Bando_et_al}, allow one to construct examples of Calabi-Yau metrics on $X$ that are Asymptotically Locally Euclidean ($\ALE$ for short).  In fact, Joyce, in \cite{Joyce2001,Joyce}, gave a more direct and self-contained proof of the existence of these metrics by using a Moser iteration with weights which yields better control of the solution of the complex Monge-Ampère equation at infinity.  His approach was subsequently generalized by various authors \cite{vanC, vanC2011,Goto,CH2013,CH2015} to obtain examples of Asymptotically Conical ($\AC$ for short) Calabi-Yau manifolds.

In \cite{Joyce,Joyce2001b}, Joyce generalized this approach in another direction by considering a crepant resolution $\pi:X\to \bbC^m/\Gamma$ with $\Gamma\subset \SU(m)$ not acting freely on $\bbC^m\setminus \{0\}$, that is, with $\bbC^m/\Gamma$ having non-isolated singularities going off to infinity.  For that purpose, Joyce introduced the notion of Quasi-Asymptotically Locally Euclidean metrics ($\QALE$-metrics for short).  As the name suggests, away from the singularities, these metrics resemble $\ALE$-metrics.  However, near rays of singularities going off to infinity, the crepant resolution introduces some topology and we can no longer use the Euclidean metric as a local model.   More precisely, if $p\in (\bbC^m/\Gamma)\setminus \{0\}$ is a singular point, then there is a neighborhood $V$  of $p$  of the form
$$
        V=\{  (z,\lambda)\in \bbC^{m-1}/\Gamma_p \times \bbC ; \;\; |\lambda|>\delta, \; \theta< \arg \lambda < \theta+\delta, \; |z|<\delta |\lambda|\}\subset  \bbC^{m-1}/\Gamma_p \times \bbC
$$
with $p$ corresponding to the point $(0,\zeta)\in \bbC^{m-1}/\Gamma_p\times \bbC$ for some $\zeta\in\bbC^*$, where $\Gamma_p \subset \SU(m-1)$ is the stabilizer of this point in $\Gamma$, $\delta>0$ and $\theta\in\bbR$.  Now, the type of crepant resolutions $\pi:X\to \bbC^m/\Gamma$ that Joyce considers are local product resolutions in the sense that $\pi^{-1}(V)$ corresponds to a subset of 
$$
             Y_p\times \bbC
$$    
with $Y_p$ a (local product) crepant resolution of $\bbC^{m-1}/\Gamma_p$.  Suppose now for simplicity  that $\Gamma_p$ acts freely on $\bbC^{m-1}\setminus \{0\}$, which is automatic if $m= 3$.  Then in this case, an example of a $\QALE$-metric on $\pi^{-1}(V)$ is given by the restriction to $\pi^{-1}(V)$ of the Cartesian product of an $\ALE$-metric $g_{Y_p}$ on $Y_p$ and the Euclidean metric $g_E$ on $\bbC$, that is,
\begin{equation}
   g_{\QALE}= g_{Y_p} + g_E.
\label{int.1}\end{equation}
More generally, if $\Gamma_p$ does not act freely on $\bbC^{m-1}$, we can assume inductively that $\QALE$-metrics have been defined in dimension $m-1$, so that one can still use \eqref{int.1} as a model of a $\QALE$-metric on $V$, this time however with $g_{Y_p}$ a $\QALE$-metric on $Y_p$ instead of an $\ALE$-metric.  Using local models as in  \eqref{int.1}, one can then define the quasi-isometric class of $\QALE$-metrics as the class of complete metrics which, outside a compact set, are quasi-isometric to an $\ALE$-metric away from the singularities and quasi-isometric to the model \eqref{int.1} in the neighborhood $\pi^{-1}(V)$ near a given singular point $p\in \bbC^m/\Gamma$.   However, to solve the complex Monge-Ampère equation and construct Calabi-Yau $\QALE$-metrics, it is important to have some control on the derivatives of the metric.  For this reason, in his definition of $\QALE$-metrics, Joyce also imposes some control on the asymptotic behavior of the derivatives of a $\QALE$-metric with respect to the local model \eqref{int.1}.    With these extra assumptions, Joyce \cite[Theorem~9.3.3]{Joyce} proves the following theorem.

\begin{theorem_int}[Joyce \cite{Joyce}]
Let $\Gamma$ be a finite subgroup of $\SU(m)$ and $X$ a crepant resolution of $\bbC^m/\Gamma$.  Then each Kähler class of $\QALE$-metrics on $X$ contains a unique Kähler Ricci-flat $\QALE$-metric. 
\end{theorem_int}  

Since the complex Monge-Ampère equation is used to obtain these Calabi-Yau metrics, the form of these metrics is not explicit, but Joyce in \cite[\S9.3]{Joyce} expressed the hope that $\QALE$-metrics with holonomy $\Sp(m)$ should also admit an explicit construction using hyperKähler quotients.  A first example in this direction was obtained by Carron \cite{Carron2011}, who showed that the Nakajima metric, constructed by Nakajima \cite{Nakajima} via hyperKähler quotients on the Hilbert scheme of  $n$ points on $\bbC^2$, is a $\QALE$-metric in the sense of Joyce.

Besides the Moser iteration with weights that generalizes almost immediately to $\QALE$-metrics, one of the key ingredients in the proof of \cite[Theorem~9.3.3]{Joyce} is the bijectivity of the Laplacian of a $\QALE$-metric when acting on some suitable weighted Hölder space, a result that Joyce \cite[\S~9]{Joyce} obtained using the maximum principle and barrier functions.  Joyce also made a more general conjecture \cite[Conjecture~9.5.16]{Joyce} on the mapping properties of the Laplacian of a $\QALE$-metric.  This conjecture has recently been proven by the second author and Mazzeo \cite{DM2014}  by obtaining heat kernel estimates via the methods of Grigor'yan and Saloff-Coste \cite{GSC2005}.  In fact, in \cite{DM2014}, the second author and Mazzeo introduce a much wider class of Riemannian metrics for which their results hold, namely the class of Quasi-Asymptotically Conical metrics ($\QAC$-metrics for short) which generalizes the class of $\QALE$-metrics in the same way that the class of $\AC$-metrics  generalizes the class of $\ALE$-metrics.  For example, a Cartesian product of $\ALE$-metrics is a $\QALE$-metric, and likewise, a Cartesian product of $\AC$-metrics is a $\QAC$-metric (see Example~\ref{mwfc.6} below).

The goal of the present paper is to extend Joyce's program \cite[\S~9]{Joyce} to the wider setting of $\QAC$-metrics and exhibit new examples of Calabi-Yau $\QAC$-metrics, in particular, examples of Calabi-Yau $\QAC$-metrics that are neither $\QALE$-metrics nor Cartesian products of $\AC$-metrics.  To achieve this, one of the key ingredients is to introduce a natural compactification of $\QAC$-manifolds  by  manifolds with fibred corners in the sense of \cite{ALMP2012,DLR}.  On the one hand, as in several works of Melrose and collaborators for other types of geometries \cite{MelroseAPS, MelroseGST,Mazzeo-Melrose, EMM, MazzeoEdge, Mazzeo-MelrosePhi}, this allows one to give a simple description of $\QAC$-metrics in terms of a natural Lie algebra of vector fields on the compactification.  More importantly, when it comes to solving the complex Monge-Ampère equation, it allows us to solve iteratively the equation on each face of the compactification, which in turn allows us to reduce the equation to a situation where the Ricci potential decays fast enough at infinity so that the methods of Tian-Yau \cite{Tian-Yau1990,Tian-Yau1991} can be applied.  

\begin{Remark}
For the Nakajima metric on the Hilbert scheme of $n$ points on $\bbC^2$, such a compactification has been independently obtained by Melrose \cite{Melrose_CIRM} using the hyperK\"ahler quotient construction of the metric.  
\end{Remark}
 
Postponing to \S~\ref{mwfc.0} a detailed description of the compactification, let us begin by explaining how to construct it for a $\QALE$-metric on a crepant resolution $X$ of $\bbC^m/\Gamma$ in the simple case where the singularities going off to infinity are all of complex codimension $k\ge 2$.  In this situation, we first radially compactify $\bbC^m/\Gamma$ to an orbifold with boundary $\overline{\bbC^m/\Gamma}$ by adding a boundary $\pa \overline{\bbC^m/\Gamma}\cong \bbS^{2m-1}/\Gamma$  at infinity as in \cite[\S~1.8]{MelroseGST}.   The boundary itself is then an orbifold, but by our simplifying assumption, its singularities correspond to a disjoint union $S= \bigcup S_i$ of singular edges $S_1, \ldots, S_{\ell}$.   A first step in constructing the compactification is to blow up in the sense of Melrose \cite{Melrose1992,MelroseMWC} each singular edge of the boundary within $\overline{\bbC^m/\Gamma}$, that is,
\begin{equation}
     \tX_{\sc}:=[\overline{\bbC^m/\Gamma};S]=[\overline{\bbC^m/\Gamma};S_1,\ldots,S_{\ell}],
\label{int.1b}\end{equation}
with blow-down map 
$$
\beta:\tX_{\sc} \to \overline{\bbC^m/\Gamma}.
$$
As illustrated in Figure~\ref{fig.1} below, this yields an orbifold with corners  with one boundary hypersurface $H_i:=\beta^{-1}(S_i)$ for each singular edge $S_i$ and one boundary hypersurface $H_{\ell+1}=\overline{\beta^{-1}(\pa \overline{\bbC^m/\Gamma})\setminus S}$ corresponding to the lift of the boundary of $ \overline{\bbC^m/\Gamma}$.  

\begin{figure}[h]
\begin{tikzpicture}
\draw (2.4791,5.9544) arc [radius=3, start angle=100, end angle=170];
\draw (0.0456,2.4791) arc [radius=3, start angle=190, end angle=260];
\draw  (3.5209,0.0456) arc [radius=3, start angle=280, end angle=350];
\draw (5.9544,3.5209) arc [radius=3, start angle=10, end angle=80];

\draw (2.4791,5.9544) arc [radius=0.5229, start angle=190, end angle=350];
\draw (0.0456,2.4791) arc [radius=0.5229, start angle=-80, end angle=80];
\draw  (3.5209,0.0456) arc [radius=0.5229, start angle=10, end angle=170];
\draw  (5.9544,3.5209) arc [radius=0.5229, start angle=100, end angle=260];
\draw (12,3) circle [radius=3];

\draw [dashed] (12,0)--(12,6);
\draw [dashed] (9,3)--(15,3);
\draw [dashed] (3,0.5229)--(3,5.4771);
\draw [dashed] (0.5229,3)--(5.4771,3);
\draw[fill] (15,3) circle [radius=0.1]; 
\draw[fill] (12,6) circle [radius=0.1]; 
\draw[fill] (9,3) circle [radius=0.1]; 
\draw[fill] (12,0) circle [radius=0.1]; 

\node at (6,3) {$H_1$}; 
\node at (3,6) {$H_2$};
\node at (0,3) {$H_3$};  
\node at (3,0) {$H_4$}; 
\node at (4,4) {$\tX_{\sc}$};
\node at (0.2,5.5) {$H_{\ell+1}=H_{5}$};
\draw[->] (7,3)--(8,3);
\node[above] at (7.5,3) {$\beta$};
\node[above right] at (15,3) {$S_1$}; 
\node[above right] at (12,6) {$S_2$};
\node[above right] at (9,3) {$S_3$};
\node[above right] at (12,0) {$S_4$};
\node at (13,4) {$\overline{\bbC^m/\Gamma}$};

\end{tikzpicture}
\caption{The blow-down map $\beta:\tX_{\sc} \to \overline{\bbC^m/\Gamma}$ in a case where $\ell=4$ with the dotted lines corresponding to the singularities of $\bbC^m/\Gamma$ }
\label{fig.1}\end{figure}
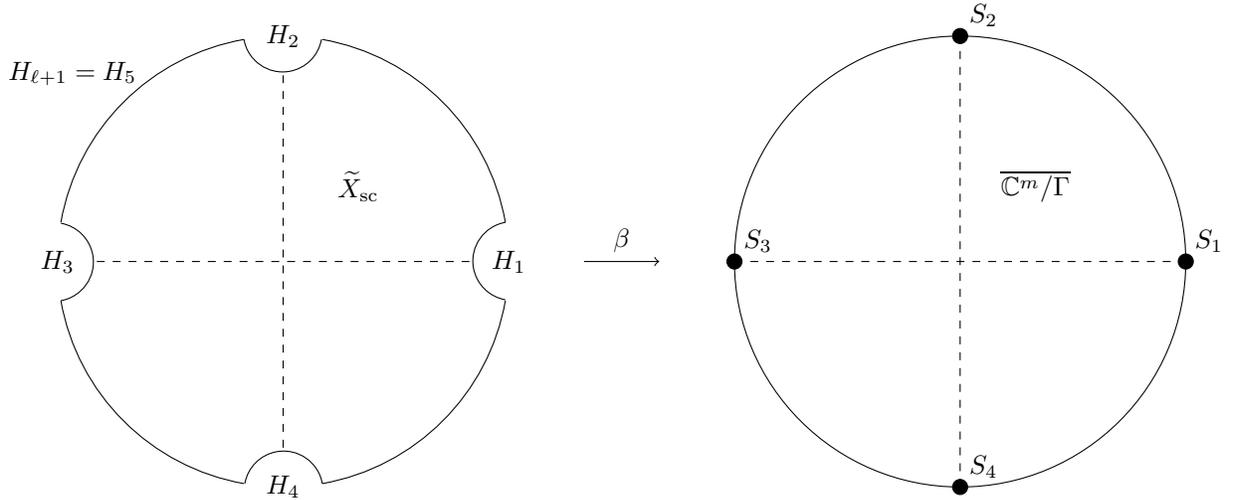

Moreover, for $i\le \ell$, the blow-down map $\beta$ restricts on $H_i$  to induce a fibre bundle structure 
\begin{equation}
\xymatrix{  \overline{\bbC^k/\Gamma_i} \ar[r] & H_i  \ar[d]^{\phi_i} \\ & S_i
}    
\label{int.2}\end{equation}
with base $S_i$ smooth, but with fibres singular at the origin, where $\Gamma_i\subset \SU(k)$ is some finite subgroup acting freely on $\bbC^k\setminus \{0\}$.  

The manifold with corners $\hX_{\QAC}$ compactifying $X$ is then obtained by observing that the crepant resolution $\pi:X\to \bbC^m/\Gamma$ naturally extends to a resolution $\pi:\hX_{\QAC}\to \overline{\bbC^m/\Gamma}$ with \eqref{int.2} replaced by    
\begin{equation}
\xymatrix{  \overline{Y}_i \ar[r] & \hH_i  \ar[d]^{\hphi_i} \\ & S_i,
}    
\label{int.3}\end{equation}
where $\overline{Y}_i$ is the radial compactification of a crepant resolution $\pi_i:Y_i\to \bbC^k/\Gamma_i$.  Given a suitable Kähler $\QAC$-metric $g$ on $X$, our strategy to construct a Calabi-Yau $\QAC$-metric is to first solve the complex Monge-Ampère equation on each fibre $\overline{Y}_i$ of $\hphi_i$, which amounts to finding a Calabi-Yau $\ALE$-metric on $Y_i$.  Using these solutions, one can then replace the metric $g$ with another Kähler $\QAC$-metric $g'$ in the same Kähler class, but with the extra property that its Ricci potential, given by 
\begin{equation}
      r'= \log \left( \frac{(\omega')^m}{c\Omega_{\hX_{\QAC}}\wedge\overline{\Omega}_{\hX_{\QAC}}} \right),
\label{rp.1}\end{equation}
decays sufficiently fast at infinity, where $\omega'$ is the Kähler class of $g'$, $\Omega_{\hX_{\QAC}}$ is some nowhere vanishing holomorphic volume form on $\hX_{\QAC}$ and $c$ is a non-zero constant.  One can then use standard techniques to solve the complex Monge-Ampère equation
$$
    \log\left(\frac{(\omega'+ \sqrt{-1}\pa\db u)^m}{(\omega')^m}  \right)= -r'
$$
and obtain a Calabi-Yau $\QAC$-metric with Kähler form $\omega' + \sqrt{-1}\pa \db u$.   

Of course, as long as we are considering $\QALE$-metrics, this is essentially the approach of Joyce rephrased in terms of the compactification $\hX_{\QAC}$.  In particular, the compactification is not really needed, since in this simpler setting the bundles \eqref{int.2} and \eqref{int.3} are in fact trivial and to construct a K\"ahler $\QALE$-metric near $\hH_i$ with Ricci potential decaying at infinity, one can simply glue directly the Cartesian product
\begin{equation}
g_{E_i}+ g_{Y_i} \quad \mbox{on} \quad (\bbR^+\times S_i)\times Y_i
\label{int.4}\end{equation}
to the Euclidean metric on $\bbC^n/\Gamma$, where
$$
g_{E_i}= dr^2 + r^2 g_{S_i}
$$   
is the Euclidean metric on $\bbR^+\times S_i$ and $g_{Y_i}$ is a Calabi-Yau $\ALE$-metric on $Y_i$. 

Now, if we replace $\bbC^m/\Gamma$ with the Euclidean metric by a more general orbifold Calabi-Yau cone $(C,g_C)$, gluing the Cartesian product \eqref{int.4} in a neighborhood of the singularities corresponding to $H_i$ can still be done directly.  However, if $g_C$ is not Euclidean, then the $\QAC$-metric introduced in this way will usually fail to have a Ricci potential that decays at the boundary face $\hH_{\ell+1}$.  This is illustrated in Figure~\ref{fig.2} below,  where the region in bold on $\hH_{\ell+1}$ corresponds to the directions in which the Ricci potential fails to decay.

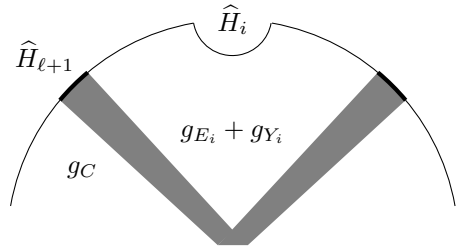
\begin{figure}[h]
\begin{tikzpicture}

\draw (2.4791,5.9544) arc [radius=3, start angle=100, end angle=170];
\draw (5.9544,3.5209) arc [radius=3, start angle=10, end angle=80];

\draw (2.4791,5.9544) arc [radius=0.5229, start angle=190, end angle=350];

\node at (3,6) {$\widehat{H}_i$};
\node at (0.5,5.5) {$\widehat{H}_{\ell+1}$};
\node at (1,4) {$g_C$};
\node at (3,4.5) {$g_{E_i}+ g_{Y_i}$};
\path [fill=gray]   (5.2981,4.9284) arc [radius=3, start angle=40, end angle=50] (5.2981,4.9284)--(3.2,3)--(2.8,3)--(4.9284,5.2981);
\path [fill=gray]   (1.0716,5.2981) arc [radius=3, start angle=130, end angle=140] (1.0716,5.2981)--(3,3.2)--(2.8,3)--(0.7019,4.9284);
\draw[ultra thick] (5.2981,4.9284) arc [radius=3, start angle=40, end angle=50];
\draw[ultra thick] (1.0716,5.2981) arc [radius=3, start angle=130, end angle=140];
 \end{tikzpicture}
\caption{Gluing of $g_{E_i}+ g_{Y_i}$ to $g_C$ with the gluing region in grey}
\label{fig.2}\end{figure}

One could instead try to glue $g_{E_i}+ g_{Y_i}$ to $g_C$ as illustrated in Figure~\ref{fig.3} below, but then the Ricci potential fails to decay on $\hH_i$, in this latter case even if $g_C$ is the Euclidean metric.  
\begin{figure}[h]
\begin{tikzpicture}
\draw (2.4791,5.9544) arc [radius=3, start angle=100, end angle=170];
\draw (5.9544,3.5209) arc [radius=3, start angle=10, end angle=80];

\draw (2.4791,5.9544) arc [radius=0.5229, start angle=190, end angle=350];

\path[fill=gray] (2.5472, 5.7386)--(1.5,3)--(1.7,3)--(2.7386,5.57);
\path[fill=gray] (3.4528, 5.76)--(4.5,3)--(4.3,3)--(3.2614,5.57);
\draw[ultra thick] (2.5472, 5.7386)--(2.7386,5.57);
\draw[ultra thick] (3.4528, 5.76)--(3.2614,5.57);
\node at (3,6) {$\widehat{H}_i$};
\node at (0.5,5.5) {$\widehat{H}_{\ell+1}$};
\node at (1,4) {$g_C$};
\node at (3,4) {$g_{E_i}+ g_{Y_i}$};
\end{tikzpicture}
\caption{A second way of gluing $g_{E_i}+ g_{Y_i}$ to $g_C$ with the gluing region in grey}
\label{fig.3}\end{figure}
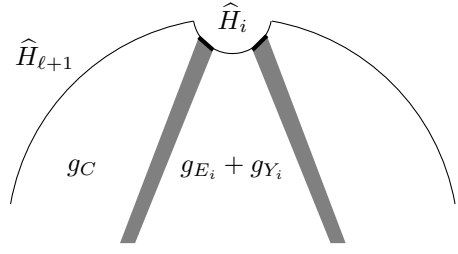

Instead, if one uses the compactification, it suffices to put the model \eqref{int.4} on $\hH_i$ and keep the model given by $g_C$ on $H_{\ell+1}$.  Provided that the two models agree on $H_i\cap H_{\ell+1}$, one can then extend the K\"ahler form in the interior as a closed $(1,1)$-form.  By continuity, this form will then be positive-definite in a neighborhood of $\hH_i$ and $\hH_{\ell+1}$.  More precisely, it will be the K\"ahler form of a $\QAC$-metric with Ricci potential decaying both at $\hH_i$ and $\hH_{\ell+1}$.

Thanks to this natural compactification of $\QAC$-manifolds by manifolds with fibred corners, we prove the following result, where we refer the reader to Definition~\ref{gid.8a} for some of the terminology and to Theorem~\ref{ma.10} and  Corollary~\ref{ma.11} for a more precise and slightly more general statement. 
\begin{Theorem}
Let $Z$ be a compact orbifold of real dimension $2n+1$ and $(C,g_C,J_C)$ a quasi-regular Calabi-Yau cone on $C=\bbR^+\times Z$ with parallel holomorphic volume form $\Omega_C$.  Let $X$ be a K\"ahler orbifold with nowhere vanishing holomorphic volume form $\Omega_X$ and with singular set of complex codimension $\mu\ge 2$. Suppose that  there is a compact set $\cK\subset X$ such that $X\setminus \cK$ is biholomorphic to $((\kappa,\infty)\times Z, J_C)$ for some $\kappa>0$.  Finally, suppose that there is a crepant resolution $\pi:\hX\to X$ with $\Omega_X$ admitting a lift $\Omega_{\hX}\in H^0(\hX;K_{\hX})$.  Then for any K\"ahler $\QAC$-metric $g$ asymptotic to $g_C$ with rate $\alpha$ such that $4\le \alpha\le 2\mu$, the complex Monge-Amp\`ere equation 
$$
\log\left(\frac{(\omega+ \sqrt{-1}\pa\db u)^{n+1}}{(\omega)^{n+1}}  \right)= -r
$$
 has a solution $u$ such that $\omega+\sqrt{-1}\pa\db u$ is the K\"ahler form of a Calabi-Yau $\QAC$-metric asymptotic to $g_C$ with rate $\alpha$, where $\omega$ is the K\"ahler form of $g$ and $r$ is its Ricci potential defined in terms of $\Omega_{\hX}$ as in \eqref{rp.1}.          
\label{int.5}\end{Theorem}

When one takes $X=\bbC^{n+1}/\Gamma$ in Theorem~\ref{int.5} for a choice of finite subgroup $\Gamma\subset \SU(n+1)$ for which $X$ admits a local product K\"ahler crepant resolution, a subtle point hidden in Definition~\ref{gid.8a} is that Theorem~\ref{int.5} is not quite the same as the result of Joyce \cite[Theorem~9.6.1]{Joyce}.  Indeed, in \cite[Theorem~9.6.1]{Joyce}, one of the hypotheses on the initial $\QALE$-metric $g$ is that its Ricci potential decays at each boundary hypersurface of $\hX_{\QAC}$, whereas in Theorem~\ref{int.5}, our result is stronger in that the Ricci potential only has to decay at the maximal boundary hypersurface ($\hH_{\ell+1}$ in the simpler setting described above), but weaker in that the rate of decay $\alpha$ at that face has to be at least $4$ (\cite[Theorem~9.6.1]{Joyce} only requires that $\alpha$ be strictly larger than $2$).    

  More interestingly, Theorem~\ref{int.5} applies to situations where the geometry at infinity is $\QAC$, but not $\QALE$.   As described in Example~\ref{exa.1} below, our main source of examples is given by applying the Calabi ansatz \cite{Calabi,LeBrun} to a K\"ahler-Einstein Fano orbifold $D$, which yields an orbifold Calabi-Yau cone metric $g_C$ on $C=K_D\setminus D$.  One can then take $X=K_D$ in Theorem~\ref{int.5} provided that $K_D$ admits a local product K\"ahler  crepant resolution in the sense of \cite{Joyce}.  This is the case for instance if $D$ itself admits a local product K\"ahler crepant resolution $\hD$, in which case $K_{\hD}$ is automatically a local product K\"ahler crepant resolution of $K_D$.  

Nevertheless, given such a $D$, one still needs to find a suitable K\"ahler $\QAC$-metric in order to apply Theorem~\ref{int.5}.  For $\AC$-metrics, it is very easy to produce examples of K\"ahler $\AC$-metrics thanks to a standard trick comprising cutting off the model metric at infinity by using a convex function.  In Lemma~\ref{kqac.1}, we do in fact use this trick to construct orbifold K\"ahler $\AC$-metrics on $X$.  Unfortunately though, this trick does not seem to generalize to $\QAC$-metrics or $\QALE$-metrics, the reason being that  the crepant resolution $\pi: \hX\to X$ introduces some topology at infinity, so that there are typically no model $\QAC$-metrics with K\"ahler form near infinity given by $\sqrt{-1}\pa\db u$ for some smooth real-valued function $u$.  In particular, just knowing that $\hX$ is K\"ahler is not sufficient to conclude that $\hX$ admits K\"ahler $\QAC$-metrics.

To overcome this difficulty, we develop a method that allows one to construct a K\"ahler $\QAC$-metric from an orbifold $\AC$-metric $g$ on $X$ by gluing suitable local models near each singularity, effectively implementing geometrically the crepant resolution.  In simple cases, this can be done directly by using cut-off functions.  However, in order to be able to tackle situations where the singularities  have arbitrary depth relatively easily, we have chosen to proceed with a systematic approach similar to what Kottke and Singer \cite{Kottke-Singer} do for gluing monopoles, that is, the gluing is implemented by using a manifold with corners with the various faces describing the metrics that need to be glued.  

Deferring to \S~\ref{kqac.0} a detailed description of this manifold with corners, let us for the moment give the main intuitive idea behind its construction by restricting to the simple setting of \eqref{int.1b}.  To simplify even further, suppose moreover that the crepant resolution of $\bbC^m/\Gamma$ is obtained by first blowing up the origin (in the sense of algebraic geometry) to obtain $K_D$ with 
$D:= \bbC\bbP^{m-1}/\Gamma'$ for some finite subgroup $\Gamma'\subset\SU(m)$, so that one can replace \eqref{int.1b} with 
\begin{equation}
      \tX_{\sc}:=[\overline{K}_D;S]=[\overline{K}_D;S_1,\ldots,S_{\ell}],
\label{int.6}\end{equation}     
where $\overline{K}_D$ is the radial compactification of $K_D$ using the Calabi-Yau cone metric.               
In this case, the singularity $S_i$ on the boundary of $\tX_{\sc}$ corresponds to the boundary  $\pa \Sigma_i$ of a singular edge $\Sigma_i$ of complex codimension $k$ in $\tX_{\sc}$, and the singular set of $\tX_{\sc}$ is given by the disjoint union
$$
        \Sigma= \bigcup_i \Sigma_i.
$$
Introducing a parameter of deformation $\varepsilon\in [0,1)$, one can then consider the orbifold with corners
$$
    \cX:=[\tX_{\sc}\times [0,1); \Sigma_1\times \{0\},\ldots, \Sigma_{\ell}\times \{0\}].
$$   
This orbifold with corners has one boundary hypersurface $\cH_i$ coming from the lift of $H_i\times [0,1)$ to $\cX$.  For each $i$, the blow-up of $\Sigma_i\times \{0\}$ also introduces a boundary hypersurface $B_i$.  Notice in particular that the blow-down map $\cX\to \tX_{\sc}\times [0,1)$ naturally induces a fibre bundle structure
\begin{equation}
\xymatrix{  \overline{\bbC^k/\Gamma_i} \ar[r] & B_i  \ar[d]^{\varphi_i} \\ & \Sigma_i\times \{0\}.
}    
\label{int.7}\end{equation}
  Finally, as illustrated in Figure~\ref{fig.4} below, the lift of the boundary hypersurface $\tX_{\sc}\times \{0\}$ induces the boundary hypersurface $B_{\ell+1}$ on $\cX$.  
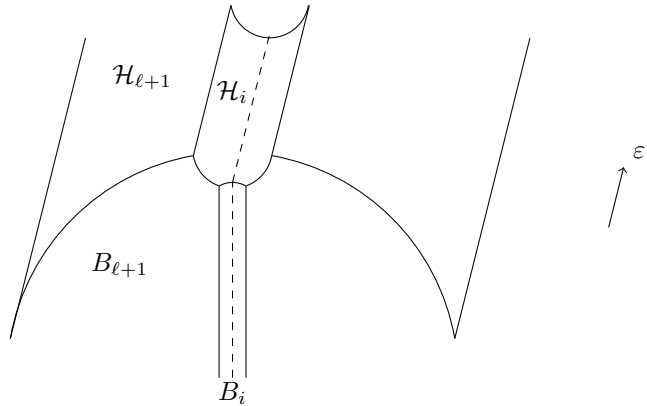
\begin{figure}[h]
\begin{tikzpicture}
\draw (2.4791,5.9544) arc [radius=3, start angle=100, end angle=170];
\draw (5.9544,3.5209) arc [radius=3, start angle=10, end angle=80];
\draw (2.4791,5.9544) arc [radius=0.5229, start angle=190, end angle=250];
\draw (3.5209,5.9544) arc [radius=0.5229, start angle=-10, end angle=-70];
\draw (3.1788,5.55) arc [radius=0.3576, start angle=60, end angle=120];
\draw (2.4791,5.9544)--(2.9791,7.9544);
\draw (3.5209, 5.9544)--(4.0209, 7.9544);
\draw (2.9791,7.9544) arc [radius=0.5229, start angle=190, end angle=350];
\draw (3.1788,5.55)--(3.1788,3);
\draw (2.8212,5.55)--(2.8212,3);
\draw (0.0456,3.5209)--(1.0456,7.5209);
\draw (5.9544,3.5209)--(6.9544,7.5209);

\draw[dashed] (3,3)--(3,5.6);
\draw[dashed] (3,5.6)--(3.5,7.6);

\node at (3,2.8) {$B_i$};
\node at (1.5,4.5) {$B_{\ell+1}$};
\node at (3,6.8) {$\cH_i$};
\node at (1.8,7) {$\cH_{\ell+1}$};
\draw[->] (8,5)--(8.2,5.8);
\node[above right] at (8.2,5.8) {$\varepsilon$};

\end{tikzpicture}
\caption{The orbifold with corners $\cX$}
\label{fig.4}\end{figure}

  The manifold with corners implementing the gluing is then obtained by observing that the resolution $\pi: \widehat{X}_{\QAC}\to \tX_{\sc}$ naturally extends to a resolution $\pi: \widehat{\cX}\to \cX$.  In particular, the resolution extends in such a way that the fibre bundle \eqref{int.7} is replaced by
\begin{equation}
\xymatrix{  \overline{Y}_i \ar[r] & \widehat{B}_i  \ar[d]^{\widehat{\varphi}_i} \\ & \Sigma_i
}    
\label{int.8}\end{equation}  
with $\overline{Y}_i$ the radial compactification of a crepant resolution $Y_i$ of $\bbC^k/\Gamma_i$.  To see how $\widehat{\cX}$ can be used to implement the gluing, notice that an orbifold K\"ahler $\AC$-metric $g$ on $K_D$ naturally induces a metric on the boundary hypersurface $B_{\ell+1}$ of $\cX$.  It also induces an orbifold K\"ahler metric $g_i$ on $B_i$ with K\"ahler form 
\begin{equation}
    \omega_i= \varphi_i^*\omega_{\Sigma_i} + \omega_{\varphi_i},
\label{int.9}\end{equation}
where $\omega_{\Sigma_i}$ is a K\"ahler form on $\Sigma_i$ and $\omega_{\varphi_i}$ is a $(1,1)$-form which restricts on each fibre of \eqref{int.7} to a Euclidean metric.  On $\widehat{\cX}$ and $\widehat{B}_i$, one can therefore replace \eqref{int.9} with
\begin{equation}
    \widehat{\omega}_i= \varphi_i^*\omega_{\Sigma_i} + \omega_{\widehat{\varphi}_i},
\label{int.10}\end{equation}  
 where $\omega_{\widehat{\varphi}_i}$ is a $(1,1)$-form which restricts on each fibre of \eqref{int.8} to the K\"ahler form of an $\ALE$-metric asymptotic to the corresponding Euclidean metric induced by $\omega_{\varphi_i}$.    In other words, the metric $\widehat{g}_i $ with K\"ahler form \eqref{int.10} is the local model that we want to glue to $g$ at the singularity $\Sigma_i$.  In terms of the manifold $\widehat{\cX}$, this gluing can be implemented by considering a closed $(1,1)$-form $\widehat{\omega}$ on (each level set of $\varepsilon$ in) $\widehat{\cX}$ which restricts to $\widehat{\omega}_i$ on $\widehat{B}_i$ and to the K\"ahler form of $g$ on $\widehat{B}_{\ell+1}=B_{\ell+1}$.  By continuity, for $\delta>0$ sufficiently small, $\left.  \widehat{\omega}\right|_{\varepsilon=\delta}$ will then be positive-definite and will induce the desired K\"ahler $\QAC$-metric on $\left. \widehat{\cX} \right|_{\varepsilon=\delta}\cong \hX_{\QAC}$.  Strictly speaking, to make this continuity argument rigorous, one must also introduce a suitable vector bundle on $\widehat{\cX}$.  We refer the reader to \S~\ref{kqac.0} for its detailed description.  
 
The biggest advantage of gluing using the manifold with corners $\widehat{\cX}$ is that it is relatively easy to generalize to singularities of arbitrary depth.  Referring to Theorem~\ref{kqac.24} below for the precise statement, let us mention some examples of the K\"ahler $\QAC$-metrics that  it produces.

\begin{Theorem}[Corollary~\ref{kqac.26} below]
Let $(D_1,g_1),\ldots, (D_q,g_q)$ be K\"ahler-Einstein orbifolds with at worst isolated singularities of complex codimension at least $2$.  Suppose that each $(D_i,g_i)$ admits a K\"ahler crepant resolution $\hD_i$.  Consider the Cartesian products
$$
      D:=D_1\times \cdots \times D_q  \quad \mbox{and} \quad \hD:= \hD_1\times \cdots \times \hD_q.
$$
Let $g_C$ be the orbifold  Calabi-Yau cone metric on $K_D\setminus D$ given by the Calabi ansatz.  Then on $K_{\hD}$, there exist K\"ahler $\QAC$-metrics that are asymptotic to $g_C$ at any rate $\alpha>0$.   
\label{int.11}\end{Theorem}    

Combining with Theorem~\ref{int.5} yields the following.

\begin{Corollary}
On $K_{\hD}$ as in Theorem~\ref{int.11}, there are Calabi-Yau $\QAC$-metrics asymptotic to $g_C$ with rate $\alpha=2\mu$, where $\mu$ is the complex codimension of the singular set of $D$.  
\end{Corollary}
As explained in Appendix~\ref{ap.0}, it is possible to push this construction further, giving even more examples of Calabi-Yau $\QAC$-metrics; see Theorem~\ref{ap.1} and the discussion that follows for details.

 The paper is organized as follows.  In \S~\ref{mwfc.0}, we give a definition of $\QAC$-metrics in terms of manifolds with fibred corners and we review some of their properties.  In \S~\ref{cr.0}, we describe how orbifold Calabi-Yau cones can be compactified by an orbifold with fibred corners and how one obtains a corresponding manifold with fibred corners when there is a local product K\"ahler crepant resolution.  In \S~\ref{gi.0}, we explain how an orbifold Calabi-Yau cone metric can be seen as a $\QAC$-metric on the corresponding orbifold with fibred corners.  In \S~\ref{kqac.0}, we introduce the manifold with corners $\widehat{\cX}$ and explain how it can be used to produce examples of K\"ahler $\QAC$-metrics.  Finally, in \S~\ref{ma.0}, we prove our main result, Theorem~\ref{int.5}, by solving the corresponding complex Monge-Amp\`ere equation.

\begin{Acknowledgement}
The authors would like to thank the Isaac Newton Institute for Mathematical Sciences, Cambridge, for support and hospitality during the program \emph{Metric and Analytic Aspects of Moduli Spaces}  where the work on this paper commenced.

Part of this work was also carried out while the first author was supported by the National Science Foundation under Grant No.~DMS-1440140 while in residence at the Mathematical Sciences Research Institute (MSRI) in Berkeley, California, during the Spring 2016 semester. He wishes to thank MSRI for their hospitality during this time.  The second author is grateful to the Beijing International Center for Mathematical Research for support and hospitality in the Summer 2015.  The third author was supported by NSERC and a Canada Research chair.  

Finally, the authors wish to thank Vestislav Apostolov, Rafe Mazzeo, Cristiano Spotti, Song Sun and Gang Tian for helpful discussions, as well as an anonymous referee for helpful suggestions.
\end{Acknowledgement}

\numberwithin{equation}{section}

\section{Manifolds with fibred corners} \label{mwfc.0}

Let $M$ be a compact manifold with corners in the sense of Melrose \cite{Melrose1992,MelroseMWC} with boundary hypersurfaces $H_1,\ldots, H_k$.  In particular, we assume that each boundary hypersurface of $M$ is embedded in $M$ and we denote by $\pa M$ the union of all the boundary hypersurfaces of $M$.  Suppose that each boundary hypersurface  $H_i$ comes endowed with a fibre bundle structure $\phi_i: H_i\to S_i$ with base $S_i$ and with each fibre a manifold with corners.  We denote by $\phi=(\phi_1,\ldots,\phi_k)$ the collection of all fibre bundle maps.

\begin{definition}[\cite{AM2011,ALMP2012, DLR}]  We say that $(M,\phi)$ is a \textbf{manifold with fibred corners} if there is a partial order on the boundary hypersurfaces such that:
\begin{itemize}
\item  Any subset $I$ of boundary hypersurfaces such that $\bigcap_{i\in I}H_i\ne \emptyset$ is totally ordered;

\item If $H_i<H_j$, then $H_i\cap H_j\ne \emptyset$, $\left. \phi_i\right|_{H_i\cap H_j}: H_i\cap H_j\to S_i$ is a surjective submersion and $S_{ji}:= \phi_j(H_i\cap H_j)$ is one of the boundary hypersurfaces of the manifold with corners $S_j$.  Moreover, there is a surjective submersion $\phi_{ji}: S_{ji}\to S_i$ such that $\phi_{ji}\circ \phi_j=\phi_i$ on $H_i\cap H_j$.

\item The boundary hypersurfaces of $S_j$ are given by the $S_{ji}$ for $H_i<H_j$.
\end{itemize}
\label{mwfc.1}\end{definition}
It follows directly from this definition that each base $S_j$ has a natural manifold with fibred corners structure induced by the maps $\phi_{ji}: S_{ji}\to S_i$ for each $i$ with $H_i< H_j$. Similarly, each fibre of $\phi_i: H_i\to S_i$ has a natural manifold with fibred corners structure.  Moreover, $S_i$ is a smooth closed manifold whenever $H_i$ is minimal with respect to the partial order, whereas the fibres of $\phi_i$ are smooth closed manifolds whenever $H_i$ is maximal.  This allows one to prove many assertions by proceeding by induction on the \textbf{depth} of $(M,\phi)$, that is, the largest codimension that a corner of $M$ may have, or by induction on the relative depth of a boundary hypersurface, where we recall that the \textbf{relative depth}  of a boundary hypersurface $H_i$ in $M$ is the largest integer $j$ such that there exist $j-1$ boundary hypersurfaces $H_{\nu_1}, \ldots , H_{\nu_{j-1}}$ with
$$
               H_i< H_{\nu_1}<\cdots < H_{\nu_{j-1}}.
$$

The notion of a manifold with fibred corners is intimately related to the notion of a stratified space, so let us recall briefly what is meant by this latter term.

\begin{definition}
A \textbf{stratified space} of dimension $n$ is a locally compact separable metrizable space $X$ together with a \textbf{stratification}, which is a locally finite partition $\mathsf{S}=\{s_i\}$ into locally closed subsets of $X$, called the \textbf{strata}, which are smooth manifolds of dimension $\dim s_i\le n$ such that at least one stratum is of dimension $n$ and
$$
       s_i\cap \overline{s}_j\ne \emptyset   \quad \Longleftrightarrow \quad s_i \subset \overline{s}_j.
$$
In this case we write that $s_i\le s_j$ and $s_i< s_j$ if $s_i\ne s_j$.  A stratification induces a filtration
$$
       \emptyset \subset X_0\subset \cdots \subset X_n =X,
$$
where $X_j$ is the union of all strata of dimension at most $j$.  The strata included in $X^{\bullet}:= X\setminus X_{n-1}$ are said to be \textbf{regular}, whereas the strata included in $X_{n-1}$ are said to be \textbf{singular}.  Given a stratified space, notice that the closure of each of its strata is also naturally a stratified space.
\label{ss.1}\end{definition}
\begin{remark}
In the present paper, it is crucial for us to impose no restrictions on the codimension of the singular strata. Notice however that in other situations, it is quite common and natural to require that the singular strata are always at least of codimension $2$; see for instance \cite{GM1980}.
\label{ss.2}\end{remark}
A good measure of the complexity of a stratified space is given by its depth, a notion which we now recall.
\begin{definition}
The \textbf{depth} of a stratified space $(X,\mathsf{S})$ is the largest $k$ such that one can find $k+1$ different strata with
$$
         s_1<s_2<\cdots<s_k<s_{k+1}.
$$
On the other hand, the \textbf{relative depth} of a stratum $s$ in $X$ is the largest $k$ such that there exists $k$ strata with
$$
     s<s_1<\cdots< s_{k}.
$$
More generally, the \textbf{relative depth} of a point $p\in X$ is the relative depth of the unique stratum $s$ containing $p$.
\label{ss.2b}\end{definition}

As observed by Melrose and described in \cite{ALMP2012,DLR},  a manifold with fibred corners $M$ always arises as a resolution of a stratified space $\St M$ given by
$\St M= M/\sim$,  where $\sim$ is the relation
$$
          p\sim q \quad \Longleftrightarrow   \quad p=q \quad \mbox{or} \quad p,q\in H_{i} \quad \mbox{with} \quad \phi_i(p)=\phi_i(q) \quad \mbox{for some hypersurface} \; H_i.
$$
In terms of the quotient map
$$
  \beta: M\to \St \! M,
$$
which we call the \textbf{blow-down map}, the regular stratum is given by $\beta(M\setminus \pa M)$.  More importantly, the blow-down map gives a one-to-one correspondence between the boundary hypersurfaces $H_i$ of $M$ and the closure of the singular strata $\overline{s}_i:=\beta(H_i)$ of $\St M$.  The base $S_i$ of the fibre bundle $\phi_i: H_i\to S_i$ is in fact itself a resolution of the stratified space $\overline{s}_i$ and we have that $s_i= \beta(\phi_i^{-1}(S_i\setminus \pa S_i))$. Moreover, the correspondence between boundary hypersurfaces and singular strata is compatible with the partial orders, i.e.,
$$
        H_i< H_j \quad \Longleftrightarrow   \quad s_i<s_j,
$$
so that the relative depth of a boundary hypersurface is equal to the relative depth of the corresponding stratum.  Notice also that the depth of $\St M$ as a stratified space is equal to the depth of $M$ as a manifold with corners.

\begin{definition}
The stratified space $\St M$ is said to be the \textbf{blow-down} of the manifold with fibred corners $(M,\phi)$.  Conversely, the manifold with fibred corners $(M, \phi)$ is said to be a \textbf{resolution} of the stratified space $\St M$.  More generally, a stratified space which admits a resolution by a manifold with fibred corners is said to be \textbf{smoothly stratified}.
\label{ss.3}\end{definition}
 \begin{remark}
 Not all stratified spaces are smoothly stratified, but the property of being smoothly stratified can be described intrinsically on the stratified space without referring to a manifold with fibred corners; see for instance \cite{ALMP2012,DLR}.
 \end{remark}
 \begin{remark}
 The notion of resolution in Definition~\ref{ss.3} should not be confused with the notion of crepant resolution discussed in the introduction.  In the first case, the singularity is resolved using a manifold with corners and can be done quite generally, while the latter case is much more specific, since the singularity must then be of a particular complex geometric nature and is resolved by a smooth complex manifold with suitable properties.  
 \end{remark}
 \begin{example}
 An orbifold is naturally a smoothly stratified space, see for instance   \cite[\S~4.4.10]{Pflaum2001} or \cite[p.210-211]{GR2013}, and the corresponding manifold with fibred corners is obtained by blowing up the singular strata in an order compatible with the partial order.
 \label{ss.4}\end{example}

Let $x_1,\ldots, x_k$ be boundary defining functions for the boundary hypersurfaces $H_1,\ldots, H_k$ of $M$, that is, $x_i$ is positive on $M\setminus H_i$, $x_i=0$ on $H_i$ and $dx_i$ is nowhere zero on $H_i$.   For each $i$, we will usually assume, unless otherwise stated, that the boundary defining function $x_i$ is identically equal to $1$ outside some tubular neighborhood of $H_i$.  This assumption is not restrictive, but turns out to be very convenient.   We also assume that the boundary defining functions $x_1,\ldots, x_k$ are compatible with $\phi$ in the following sense.
\begin{definition}
We say that the boundary defining functions $x_1,\ldots,x_k$ are \textbf{compatible} with the collection of fibre bundle maps $\phi$ if for each $i$ and $j$ with $H_i<H_j$, the restriction of $x_i$ to $H_j$ is constant along the fibres of $\phi_j: H_j\to S_j$.
\label{mwfc.1b}\end{definition}
The assumption of compatibility with $\phi$ certainly imposes some restrictions on the choice of boundary defining functions.     However, it imposes no restriction on the type of manifold with fibred corners, since by \cite[Lemma~1.4]{DLR}, we know that manifolds with fibred corners always admit compatible boundary defining functions.  One important advantage of this compatibility condition is that for each $i$, the other boundary defining functions $x_j$ naturally define boundary defining functions for $S_i$ that are compatible with the induced manifold with fibred corners structure.
Another advantage, as the next lemma shows, is that it gives a nice local product decomposition of the manifold with fibred corners structure near a boundary hypersurface.  
\begin{lemma}
If $x_1,\ldots, x_k$ are boundary defining functions compatible with $\phi$, then for each $i$, there exists a tubular neighborhood
$$
       c_i: H_i\times [0,\epsilon)\hookrightarrow M
$$
such that
\begin{enumerate}
\item[(i)] $c_i^*x_i= \pr_2$;
\item[(ii)] $c^*_i x_j= x_j\circ \pr_1$ for $j\ne i$;
\item[(iii)] $c_i^{-1}\circ \phi_j\circ c_i (h,t)= \phi_j(h)$ for $h\in H_i\cap H_j$ and $H_j<H_i$;
\item[(iv)] $c_i^{-1}\circ \phi_j\circ c_i (h,t)= (\phi_j(h),t)$ for $h\in H_i\cap H_j$ and $H_j>H_i$;
\end{enumerate}
where $\pr_1: H_i\times [0,\epsilon)\to H_i$ and $\pr_2: H_i\times [0,\epsilon) \to [0,\epsilon)$ are the projections on the first and second factors.
\label{mwfc.1c}\end{lemma}
\begin{proof}
Let $\xi\in \CI(TM)$ be a vector field such that
\begin{itemize}
\item $dx_i(\xi)>0$ everywhere on $H_i$;
\item $dx_j(\xi)=0$ in a neighborhood of $H_j$ in $M$ for $j\ne i$;
\item $\left. \xi\right|_{H_j}$ is tangent to the fibres of $\phi_j: H_j\to S_j$ for $H_j<H_i$.
\end{itemize}
Then in a sufficiently small neighborhood $\cN_i$ of $H_i$, the vector field $\eta:= \frac{\xi}{dx_i(\xi)}$ is well-defined and satisfies the same properties as $\xi$ with the extra feature that $dx_i(\eta)\equiv 1$ in $\cN_i$.  By construction then,  the flow of $\eta$ generates the desired tubular neighborhood, the last condition being satisfied thanks to the fact that the boundary defining functions $x_1,\ldots,x_k$ are compatible with $\phi$.

\end{proof}

 Recall from \cite{Melrose1992} that
$$
     \cV_b(M):= \{  \xi\in TM \; | \; \xi x_i \in x_i\CI(M) \; \forall i \}
$$
is the Lie algebra of $b$-vector fields on $M$, that is, smooth vector fields on $M$ which are tangent to all boundary hypersurfaces.  Notice that this definition does not depend on the choice of boundary defining functions $x_i$.

\begin{definition}
A \textbf{quasi fibred boundary vector field}, or $\QFB$-vector field for short, is a $b$-vector field $\xi$ such that for each $i$,
\begin{itemize}
\item $\left. \xi \right|_{H_i}$ is tangent to the fibres of $\phi_i$;
\item $\xi v_i\in v_i^2\CI(M)$, where $\displaystyle v_i= \prod_{H_j\ge H_i} x_j$.
\end{itemize}
We denote the space of $\QFB$-vector fields by $\cV_{\QFB}(M)$.
\label{mwfc.2}\end{definition}
\begin{remark}
This is closely related to the definition of an iterated fibred corners vector field given in \cite{DLR}, the difference being that in \cite{DLR}, one requires that $\xi x_i\in x_i^2\CI(M)$ for each $i$ instead of asking that  
$$
\xi v_i\in v_i^2\CI(M).
$$
\end{remark}

\begin{example}
If $M$ is a manifold with boundary, $\phi: \pa M\to S$ is a fibre bundle and $x\in \CI(M)$ is a boundary defining function, then $\cV_{\QFB}(M)$ is the Lie algebra of fibred boundary vector fields (or $\phi$-vector fields) introduced by Mazzeo and Melrose \cite{Mazzeo-MelrosePhi}. If in fact $S=\pa M$ and $\phi=\Id$, then $\cV_{\QFB}(M)$ is the Lie algebra of scattering vector fields (or asymptotically conical vector fields) introduced by Melrose \cite{MelroseGST}.
\end{example}

When there are corners of codimension 2 and higher, we can give a simple description of $\QFB$-vector fields in terms of coordinates adapted to the fibred corners structure.  If $p\in \pa M$ is contained in the corner $H_1\cap\cdots \cap H_{\ell}$, we can for simplicity  label the boundary  hypersurfaces in such a way that
$$
     H_1<\cdots <H_{\ell}.
$$
Let $x_1,\ldots, x_{\ell}$ be the corresponding boundary defining functions.  In a neighborhood of $p$ where each fibre bundle $\phi_i$ is trivial, consider tuples of functions
$y_i=(y_i^1,\ldots,y_i^{k_i})$ for $i\in\{1,\ldots, \ell\}$ and $z=(z_1,\ldots,z_q)$, such that
\begin{equation}
(x_1,y_1,\ldots, x_{\ell}, y_{\ell},z)
\label{mwfc.2b}\end{equation}
defines coordinates near $p$ with the property that on $H_i$, $(x_1,y_1,\ldots, x_{i-1}, y_{i-1}, y_i)$ induces coordinates on the base $S_i$ with $\phi_i$ corresponding to the map
$$
(x_1,y_1,\ldots,\widehat{x}_i, y_i,\ldots, x_{\ell}, y_{\ell},z)\mapsto (x_1,y_1,\ldots, x_{i-1}, y_{i-1}, y_i),
$$
where the notation ``\;\;$\widehat{}$\;\;'' above the variable $x_i$ denotes its omission.  In these coordinates, one can check that the Lie algebra $\cV_{\QFB}(M)$ is locally spanned over $\CI(M)$ by
\begin{equation}
 v_1x_1\frac{\pa}{\pa x_1}, v_1 \frac{\pa}{\pa y_1^{n_1}},  v_2x_2\frac{\pa}{\pa x_2}- v_1\frac{\pa}{\pa x_1}, v_2\frac{\pa}{\pa y_2^{n_2}}, \ldots,  v_{\ell}x_{\ell}\frac{\pa}{\pa x_{\ell}}- v_{\ell-1}\frac{\pa}{\pa x_{\ell-1}}, v_{\ell}\frac{\pa}{\pa y_{\ell}^{n_{\ell}}}, \frac{\pa}{\pa z_p}\label{mwfc.2c}\end{equation}
for $p\in\{1,\ldots,q\}$ and $n_i\in\{1,\ldots,k_i\}$, where $x=\prod_{i=1}^{\ell} x_i$ and $v_i=\prod_{m=i}^{\ell} x_m$.  In other words, in these coordinates, a $\QFB$-vector field $\xi\in \cV_{\QFB}(M)$ is of the form
\begin{equation}
a_1v_1x_1\frac{\pa}{\pa x_1} +\sum_{i=2}^{\ell}a_i\left(v_ix_i\frac{\pa}{\pa x_i}- v_{i-1}\frac{\pa}{\pa x_{i-1}}\right) + \sum_{i=1}^{\ell}\sum_{j=1}^{k_i} b_{ij}v_i\frac{\pa}{\pa y_i^{j}} +\sum_{p=1}^{q} c_p\frac{\pa}{\pa z_p}\label{mwfc.2d}\end{equation}
with $a_i, b_{ij}, c_p\in \CI(M)$, \cf \cite[equation (2.6)]{DLR}.

\begin{definition}
When the manifold with fibred corners $(M,\phi)$ is such that $S_i=H_i$ and $\phi_i=\Id$ for each maximal boundary hypersurface $H_i$, we say that a $\QFB$-vector field is a \textbf{quasi-asymptotically conical vector field} ($\QAC$-vector field for short) and $(M,\phi)$ is a $\QAC$-manifold with fibred corners.
\label{mwfc.2a}\end{definition}

The space $\cV_{\QFB}(M)$ clearly depends on the fibre bundle structure of each boundary hypersurface $H_i$.  It also depends on the choice of boundary defining functions.
\begin{definition}
If $H_1,\ldots,H_k$ is an exhaustive list of all the boundary hypersurfaces of the manifold with fibred corners $(M,\phi)$, then two different choices $x_1,\ldots, x_k$ and $x_1',\ldots,x_k'$ of boundary defining functions are said to be \textbf{$\QFB$-equivalent} if they yield  the same Lie algebra of $\QFB$-vector fields.  If $(M,\phi)$ is a $\QAC$-manifold with fibred corners, then we will also say that they are \textbf{$\QAC$-equivalent} when they are $\QFB$-equivalent.
\label{mwfc.2ee4}\end{definition}
The next lemma gives a criterion to determine when two collections of boundary defining functions are $\QFB$-equivalent.

\begin{lemma}
If $H_1,\ldots,H_k$ is an exhaustive list of all the boundary hypersurfaces of the manifold with fibred corners $(M,\phi)$, then two different choices $x_1,\ldots, x_k$ and $x_1',\ldots,x_k'$ of boundary defining functions compatible with $\phi$ are $\QFB$-equivalent if and only if for all $i$, the function
$$
f_i:= \log\left(  \frac{v_i'}{v_i}\right)= \sum_{H_j\ge H_i} \log\left(  \frac{x_j'}{x_j}\right) \in \CI(M)
$$
is such that for all $H_j\ge H_i$, $\displaystyle  \left. f_i\right|_{H_j}=\phi_j^* h_{ij}$ for some $h_{ij}\in \CI(S_j)$.

\label{mwfc.2ee}\end{lemma}
\begin{proof}
Consider the Lie algebra of vector fields
$$
  \cV_{b,\phi}(M):= \{ \xi \in \cV_b(M) \; | \; (\phi_i)_*(\left. \xi\right|_{H_i})=0 \;\forall i\}.
$$
Clearly then, by Definition~\ref{mwfc.2}, the two choices of boundary defining functions yield  the same space of $\QFB$-vector fields if and only if for all $\xi\in \cV_{b,\phi}(M)$,
\begin{equation}
  \frac{dv_i}{v_i^2}(\xi)\in \CI(M) \; \forall i  \; \Longleftrightarrow \;   \frac{dv_i'}{(v_i')^2}(\xi)\in \CI(M) \; \forall i.
\label{mwfc.2ee2}\end{equation}
Now, by definition of $f_i\in \CI(M)$, we have that $v_i'= e^{f_i}v_i$, so that
\begin{equation}
       \frac{dv_i'}{(v_i')^2}=  e^{-f_i}\left(  \frac{dv_i}{v_i^2}  + \frac{df_i}{v_i}\right).
\label{mwfc.2ee3}\end{equation}
In particular, if for all $i$ and all $H_j\ge H_i$, $\displaystyle  \left. f_i\right|_{H_j}=\phi_j^* h_{ij}$ for some $h_{ij}\in \CI(S_j)$, then we see from \eqref{mwfc.2ee3} that \eqref{mwfc.2ee2} holds.  Conversely, if for some $H_i$ and some $H_j\ge H_i$, $\displaystyle  \left. f_i\right|_{H_j}$ is not the pull-back of some element of $\CI(S_j)$, then we can find $\xi\in \cV_{b,\phi}(M)$ such that $\left. df_i(\xi)\right|_{H_j}\ne 0$ and
$$
        \frac{dv_{\ell}}{v_{\ell}^2}(\xi) \in \CI(M) \; \forall \ell,
$$
so that by \eqref{mwfc.2ee3},  $\frac{dv_i'}{(v_i')^2}(\xi)$ is not bounded near $H_j$, implying in particular that \eqref{mwfc.2ee2} does not hold.
\end{proof}

It is clear from the definition that $\cV_{\QFB}(M)$ is in fact a Lie subalgebra of $\cV_b(M)$.  In particular, we can define the space $\Diff_{\QFB}^*(M)$ of $\QFB$-differential operators as the universal enveloping algebra of $\cV_{\QFB}(M)$ over $\CI(M)$.  Thus, $\Diff^q_{\QFB}(M)$ is the space of operators generated by multiplication by elements of $\CI(M)$ and the action of up to $q$ $\QFB$-vector fields.

Now, as described in \eqref{mwfc.2c}, the Lie algebra $\cV_{\QFB}(M)$ is a locally free sheaf of rank $\dim M$ over $\CI(M)$. Hence, by the Serre-Swan theorem, there exists a natural smooth vector bundle, the \textbf{$\QFB$-tangent bundle}, which we shall denote by ${}^{\phi}\!TM\to M$, and a natural map $\iota_{\phi}: {}^{\phi}\!TM \to TM$ restricting to an isomorphism on $M\setminus \pa M$, such that
$$
      \cV_{\QFB}(M) = (\iota_{\phi})_* \CI(M;{}^{\phi}\!TM).
$$
More precisely, at a point $p\in M$, the fibre of ${}^{\phi}\!TM$ is given by ${}^{\phi}\! T_pM= \cV_{\QFB}(M)/ \cI_p \cdot \cV_{\QFB}(M)$, where $\cI_p$ is the ideal of smooth functions vanishing at $p$.   It is then natural to define the \textbf{$\QFB$-cotangent bundle} to be the dual ${}^{\phi}\!T^*M$  of the $\QFB$-tangent bundle ${}^{\phi}\!TM$.  In terms of the coordinates \eqref{mwfc.2b} near $H_1\cap\cdots \cap H_{\ell}$, a local basis of sections of the $\QFB$-cotangent bundle is given by
\begin{equation}
  \frac{d v_1}{v_1^2}, \frac{dy_1^{n_1}}{v_1}, \ldots,  \frac{d v_{\ell}}{v_{\ell}^2}, \frac{dy_{\ell}^{n_{\ell}}}{v_{\ell}}, dz_k
 \label{mwfc.2e}\end{equation}
for $k\in\{1,\ldots,q\}$ and $n_i\in\{1,\ldots,k_i\}$.

\begin{definition}
A \textbf{quasi fibred boundary metric} ($\QFB$-metric for short) is a choice of Euclidean metric $g_{\phi}$ for the vector bundle ${}^{\phi}\!TM$.  A \textbf{smooth $\QFB$-metric} on $M\setminus \pa M$ is a Riemannian metric on $M\setminus \pa M$ induced by some $\QFB$-metric $g_{\phi}$ via the map $\iota_{\phi}: {}^{\phi}\!TM \to TM$.   Hoping this will lead to no confusion, we will also denote by $g_{\phi}$ the smooth $\QFB$-metric induced by a $\QFB$-metric $g_{\phi}\in \CI(M;{}^{\phi}\!T^*M\otimes {}^{\phi}\!T^*M)$.
\label{mwfc.2ffb}
\end{definition}
\begin{remark}
Because $M$ is a compact space and all of the Euclidean metrics of a vector bundle on a compact space are quasi-isometric, notice that all smooth $\QFB$-metrics are automatically quasi-isometric among themselves.  Thus, more generally, if a Riemannian metric on $M\setminus \pa M$ is quasi-isometric to a smooth $\QFB$-metric $g_{\QFB}$ and if all of its derivative taken with respect to the covariant derivative of $g_{\QFB}$ are bounded, then we say that it is a \textbf{$\QFB$-metric}. Similarly, we say that a $\QFB$-metric $g_{\phi}$ is \textbf{polyhomogeneous} if $g_{\phi}$ is induced by a Euclidean metric on ${}^{\phi}\!TM$ which is polyhomogeneous on $M$.
\label{mwfc.2g}\end{remark}

\begin{example}
In the local basis \eqref{mwfc.2e}, an example of a $\QFB$-metric is given by
\begin{equation}
     \sum_{i=1}^{\ell} \frac{dv_i^2}{v_i^4} + \sum_{i=1}^{\ell} \sum_{j=1}^{k_i} \frac{(dy_i^j)^2}{v_i^2} + \sum_{k=1}^q  dz_k^2.
\label{mwfc.2f}\end{equation}
\label{mwfc.4}\end{example}
\begin{example}
If $M$ is a manifold with fibred boundary $\phi:\pa M\to S$, then a $\QFB$-metric is a fibred boundary metric (or $\phi$-metric) in the sense of \cite{Mazzeo-MelrosePhi}.  If moreover $S=\pa M$ and $\phi=\Id$, then a $\QFB$-metric is a scattering metric (also called an asymptotically conical metric) in the sense of \cite{MelroseGST}.
\label{mwfc.5}\end{example}

Notice that given a smooth $\QFB$-metric $g_{\QFB}$, there is an  alternative description of the Lie algebra $\cV_{\QFB}(M)$, namely, it is given by the smooth vector fields on $M$ which are uniformly bounded with respect to $g_{\QFB}$:
\begin{equation}
  \cV_{\QFB}(M)=\{  \xi\in \CI(M;TM) \; | \; \sup_{M\setminus \pa M} g_{\QFB}(\xi,\xi)<\infty \}.
\label{mwfc.5a}\end{equation}

In this paper, we are interested in the following particular example of a $\QFB$-metric first considered in \cite{DM2014}.
\begin{definition}
A \textbf{quasi-asymptotically conical metric} ($\QAC$ for short) is a $\QFB$-metric on a manifold with fibred corners such that $S_i=H_i$ and $\phi_i=\Id$ for each maximal boundary hypersurface $H_i$ with respect to the partial order.
\label{mwfc.3}\end{definition}

\begin{example}(\cf \cite[\S2.3.5]{DM2014})
Let $M_1$ and $M_2$ be two smooth manifolds with boundary and consider the manifold with corners $M=[M_1\times M_2; \pa M_1\times \pa M_2 ]$ obtained by blowing up the corner of $M_1\times M_2$ in the sense of Melrose \cite{MelroseAPS}.  Let $\beta: M\to M_1\times M_2$ denote the blow-down map.  As illustrated in Figure~\ref{pac.1} below, $M$ has three boundary hypersurfaces, $H_1$ and $H_2$  coming respectively from the old faces $\pa M_1\times M_2$ and $M_1\times \pa M_2$, and $H_3$ coming from the blown up corner. 
\begin{figure}[h]
\begin{tikzpicture}
\draw (0,1)--(0,3);
\draw (1,0)--(5,0);
\draw (1,4)--(5,4);
\draw (6,1)--(6,3);

\draw (8.5,0)--(8.5,4);
\draw (8.5,0)--(14.5,0);
\draw (8.5,4)--(14.5,4);
\draw (14.5,0)--(14.5,4);

\draw (1,0) arc [radius=1, start angle=0, end angle=90];
\draw (0,3) arc [radius=1, start angle=-90, end angle=0];
\draw (6,1) arc [radius=1, start angle=90, end angle=180];
\draw (5,4) arc [radius=1, start angle=180, end angle=270];

\node at (0.5,0.5) {$H_3$};
\node at (0.5,3.5) {$H_3$};
\node at (5.5,0.5) {$H_3$};
\node at (5.5,3.5) {$H_3$};
\node at (3,-0.5) {$H_2$};
\node at (3,4.5) {$H_2$};
\node at (0.5,2) {$H_1$};
\node at (5.5,2) {$H_1$};
\node at (3,2) {$M$};
\draw[->] (6.5,2)--(8,2);
\node[above] at (7.25,2) {$\beta$};
\node at (11.5,-0.5) {$M_1\times \pa M_2$};
\node at (11.5,4.5) {$M_1\times \pa M_2$};
\node at (9.5,2) {$\pa M_1\times M_2$};
\node at (13.5,2) {$\pa M_1\times M_2$};
\end{tikzpicture}
\caption{The blow-down map  $\beta:M\to M_1\times M_2$}
\label{pac.1}\end{figure}
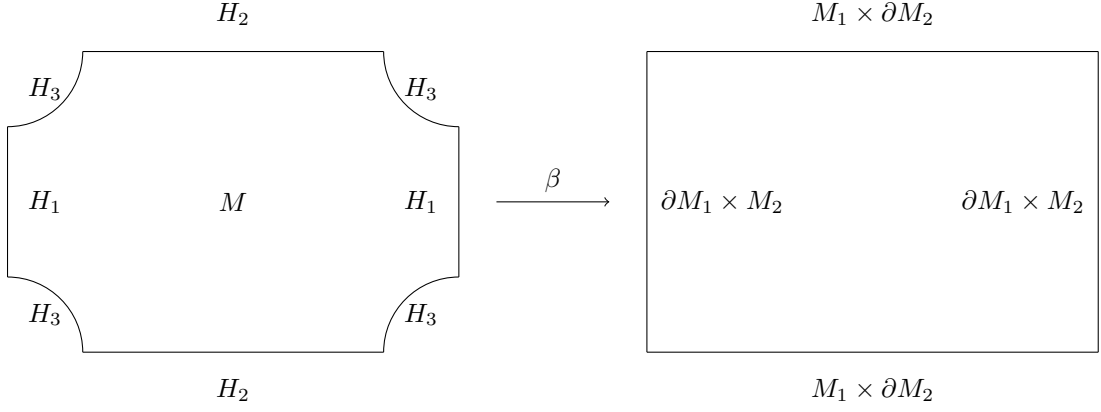

\noindent These faces come naturally with fibre bundle structures
$$
           \phi_1: \pa M_1\times M_2\to \pa M_1, \quad \phi_2: M_1\times \pa M_2\to \pa M_2,  \quad \phi_3=\Id: H_3\to H_3
$$
given respectively by the projections on the left and right factors for $H_1=\pa M_1\times M_2$ and $H_2=M_1\times \pa M_2$, and by the identity map on $H_3$.  These fibre bundles $\phi=(\phi_1,\phi_2,\phi_3)$ endow $(M,\phi)$ with a manifold with fibred corners structure with partial order given by $H_1<H_3$ and $H_2<H_3$.  Let $u_1$ and $u_2$ be boundary defining functions on $M_1$ and $M_2$ respectively and denote also by $u_1$ and $u_2$ their pull-back to $M_1\times M_2$ via the projections on the left and right factors.  On $M$, consider the polar coordinates
$u_1=r\cos \theta$, $u_2=r\sin\theta$, so that $x_1=\cos\theta$, $x_2= \sin\theta$ and $x_3=r$ are boundary defining functions for $H_1$, $H_2$ and $H_3$ respectively. In terms of these choices, the functions $v_i$ of Definition~\ref{mwfc.2} are given by
$$
      v_1= x_1x_3=u_1 , \quad v_2=x_2x_3= u_2, \quad v_3=x_3=r.
$$
In terms of the local basis of sections \eqref{mwfc.2e} of the $\QAC$-cotangent bundle, we thus see that any pair of asymptotically conical metrics $g_i\in \CI(M_i; {}^{\sc}TM_i)$, $i=1,2$ naturally induces a $\QAC$-metric $g_{\QAC}$ on $M$ by taking their Cartesian product:
$$
      g_{\QAC}= \beta^*(g_1\times g_2).
$$
\label{mwfc.6}\end{example}

\begin{remark}
On a manifold with boundary $M$ with fibre bundle structure on $\pa M$ given by the identity map, all choices of boundary defining functions lead to the same Lie algebra of scattering vector fields.  In Example~\ref{mwfc.6}, this is reflected by the fact that the Lie algebra of $QAC$-vector fields does not depend on the choice of $u_1$ and $u_2$.  However, choosing boundary defining functions $x_i$ not induced by the choices of $u_1$ and $u_2$ as in the previous example may change the Lie algebra of $QAC$-vector fields.
\label{mwfc.7}\end{remark}

\begin{proposition}
Definition~\ref{mwfc.3} is a particular case of the notion of a $\QAC$-metric introduced in \cite{DM2014}.
\label{mwfc.8}\end{proposition}
\begin{proof}
Let $(M,\phi)$ be a manifold with fibred corners such that $S_i=H_i$ and $\phi_i=\Id$ for each maximal hypersurface. The simplest way to see that Definition~\ref{mwfc.3} is a special case of \cite{DM2014} is to proceed by recurrence on the depth of $M$.  If the depth of $M$ is one, that is, if $M$ is a manifold with boundary, then Definition~\ref{mwfc.3} just corresponds to the notion of a scattering metric in the sense of \cite{MelroseGST}, so that the result is obvious in this case.

Suppose then that the result holds for all $\QAC$-manifolds with fibred corners of depth $k\ge 1$ and smaller and let $(M,\phi)$ be a $\QAC$-manifold with fibred corners of depth $k+1$.  Let $H_1$ be a boundary hypersurface of relative depth $k+1$ in $M$.
  Using a tubular neighborhood $\mathcal{N}_1\cong H_1\times [0,\epsilon)_{x_1}$ of $H_1$ as in Lemma~\ref{mwfc.1c}, we see from Example~\ref{mwfc.4} that in $\mathcal{N}_1$, an example of a smooth $\QAC$-metric is given by
\begin{equation}
   \frac{dv_1^2}{v_1^4} + \frac{\phi_1^{*}g_{S_1}}{v_1^2} + h,
\label{mwfc.9}\end{equation}
where $h$ is a symmetric 2-tensor which restricts to a smooth $\QAC$-metric on each fibre of $\phi_1:H_1\to S_1$ and $g_{S_1}$ is a Riemannian metric on the closed manifold $S_1$.  Here, we think of $h$ as being constant in $x_1$, that is, defined on $H_1$ and pulled back to the tubular neighborhood of $H_1$ given by Lemma~\ref{mwfc.1c}.

Alternatively, we can look at the restriction of $h$ to each level set of $v_1$.  For $c>0$, the region of the level set $v_1=c$ contained in the tubular neighborhood $\mathcal{N}_1$ corresponds to the open set $\cU_c= \{ p\in H_1\; | \; v_2(p)>\frac{c}{\epsilon}\}$ in $H_1$ under the retraction $r_1:\mathcal{N}_1\to H_1$. In fact, we have that
$$
      H_1\setminus \pa H_1= \bigcup_{0<c<\epsilon} \cU_c.
$$
Thus, in terms of the level sets of $v_1$, and restricting to a region $\mathcal{W}$ of $S_1$ where the fibre bundle $\phi_1:H_1\to S_1$ is trivial, we see that the metric \eqref{mwfc.9} can be assumed to be the Cartesian product of the cone metric $\frac{dv_1^2}{v_1^4} + \frac{g_{S_1}}{v_1^2}$ with a $\QAC$-metric $h$ on $Z_1$, but restricted to the region
$$
         \{ (c,v,z) \in [0,\epsilon)_{v_1}\times \mathcal{W}\times Z_1\; | \;  v_2(z)> \frac{c}{\epsilon} \}\subset [0,\epsilon)_{v_1}\times \mathcal{W}\times Z_1,
$$
where $Z_1$ is a typical fibre of $\phi_1$ so that $\phi_1^{-1}(\mathcal{W})\cong \mathcal{W}\times Z_1$.  This is precisely the local inductive model given in \cite[Lemma~2.9]{DM2014}. This completes the proof.
\end{proof}

\begin{remark}
Using directly \cite[Lemma~2.9]{DM2014}, we did not have to deal with the notion of resolution blow-ups introduced in \cite{DM2014}.  Since this is intuitively useful, let us nevertheless quickly explain how resolution blow-ups arise in terms of Definition~\ref{mwfc.3}. Thus, let $(M,\phi)$ be a $\QAC$-manifold with fibred corners and let $D$ be the (disjoint) union of all maximal boundary hypersurfaces of $M$. Then $D$ naturally inherits a manifold with fibred corners structure. As such, there is an associated smoothly stratified space $Y_0$ with singular strata in one-to-one correspondence with the boundary hypersurfaces of $D$.  Let $x=\prod_{H_i\subset \pa M} x_i$ be the product of all boundary defining functions of $M$. Then for $\epsilon>0$ small, the level set
$$
Y_{\epsilon}:=  \{ p\in M \; | \; x(p)=\epsilon\}
$$
of $x$ corresponds to a smooth desingularization of $Y_0$ in the sense of \cite[\S~2.3]{DM2014}.  Let $\iota_{\epsilon}: Y_{\epsilon}\to M$ be the natural inclusion and consider the metric $g_{\epsilon}:= \iota_{\epsilon}^*(x^2g_{\QAC})$ for  $g_{\QAC}$  a choice of $\QAC$-metric on $(M,\phi)$.   Then for $\epsilon>0$, the family $(Y_{\epsilon}, g_{\epsilon})$ is a resolution blow-up of $(Y_0,g_0)$ for some  incomplete iterated edge metric $g_0$   on $Y_0$, the stratified space associated to $D$.  A subtle point however is that $g_0$ is not equal to $\iota_0^* x^2 g_{\QAC}$ for  $\iota_0: D\to M$ the natural inclusion.  For instance, if $x_{\max}$ denotes the boundary defining function for $D$, then terms of the form $\iota_{\epsilon}^*x^2\frac{dx^2_{\max}}{x_{\max}^4}$ lead to  non-zero contributions in the limit $\epsilon\searrow 0$, but these contributions are clearly not captured by $\iota_0^* x^2 g_{\QAC}$, which in fact is not an incomplete iterated edge metric.  See Remark~\ref{kqac.16} below for a related phenomenon.          

\label{mwfc.10}\end{remark}

\begin{remark}
One advantage of using $\QAC$-manifolds with fibred corners is that the weight functions of \cite[\S~2.4]{DM2014} admit a simple description in terms of boundary defining functions.  If $x=\prod_{H_i\subset \pa M} x_i$ is the product of all of the boundary defining functions of $M$, then the radial function of \cite{DM2014} can  be taken to be $\rho=\frac1{x}$.  Moreover, if $(M,\phi)$ is of depth $\ell$, then for $1\le j\le \ell$, the weight function $w_j$ of \cite[\S~2.4]{DM2014} is simply the product of all boundary defining functions associated to boundary hypersurfaces of relative depth at least $j$.  In particular, in the situation where the boundary hypersurfaces of $M$ are totally ordered, say $H_1<\cdots <H_{\ell}$, the weight function $w_j$ is simply given by
$$
    w_j= \prod_{i\le j} x_i= \frac{x}{v_{j+1}}.
$$
 \label{mwfc.11}\end{remark}

Using a similar approach as in the proof of Proposition~\ref{mwfc.8}, we can obtain the following general result about  $\QFB$-metrics.
\begin{proposition}
Each $\QFB$-metric is a complete metric of infinite volume with bounded geometry.
\label{bg.1}\end{proposition}
\begin{proof}
A $\QFB$-metric is a particular example of a metric with Lie structure at infinity in the sense of \cite{ALN04}.  As such, it is complete and of infinite volume by \cite[Proposition~4.1 and Corollary~4.9]{ALN04}.  By \cite[Corollary~4.3]{ALN04}, its curvature is bounded, as well as all of its covariant derivatives.  By \cite[Corollary~4.20]{ALN04} and Remark~\ref{mwfc.2g}, for a fixed Lie algebra of $\QFB$-vector fields, it suffices to find one example of a $\QFB$-metric with positive injectivity radius to conclude that all $\QFB$-metrics have positive injectivity radius.  Now, it is well-known that closed Riemannian manifolds have a positive injectivity radius.  Thus, proceeding by induction on the depth, we can assume that $\QFB$-metrics on a manifold with fibred corners of depth $k$ have a positive injectivity radius.  On a manifold with fibred corners $(M,\phi)$ of depth $k+1$, let $H_1$ be one of the boundary hypersurfaces of relative depth $k+1$.  Using a tubular neighborhood $\mathcal{N}_1\cong H_1\times [0,\epsilon)_{x_1}$ of $H_1$ as in Lemma~\ref{mwfc.1c}, we can consider a $\QFB$-metric as in \eqref{mwfc.9}, this time however with $h$ restricting to a $\QFB$-metric on each fibre.  In particular, each fibre has positive injectivity radius by our induction hypothesis.    Clearly then, by the discussion at the end of the proof of Proposition~\ref{mwfc.8},  the points contained in the smaller tubular neighborhood $\mathcal{N}_1\cong H_1\times [0,\frac{\epsilon}{2})_{x_1}$ have positive injectivity radii uniformly bounded from below by a positive constant.  Considering similar metrics near each boundary hypersurface of relative depth $k+1$, we can then extend in an arbitrary way to obtain a global $\QFB$ metric on $(M,\phi)$, which, thanks to our induction hypothesis, will have positive injectivity radius.  Hence, by our earlier observation, all $\QFB$-metrics have positive injectivity radius.  Combined with the control on the curvature and on its derivatives mentioned previously, this shows that all $\QFB$-metrics have bounded geometry.
\end{proof}

For $\QAC$-metrics, we also have another important property, namely the Sobolev inequality.

\begin{lemma}[Sobolev inequality]
Given a $\QAC$-metric $g$ on a manifold with fibred corners $M$ of dimension $m$, there exists a constant $C>0$ such that
\begin{equation}
   \left( \int_{M\setminus \pa M} |u|^{\frac{2m}{m-2}} d\vol(g)\right)^{\frac{m-2}{m}}\le C \int |\nabla u|^2_{g} d\vol(g)    \quad \forall u\in \CI_c(M\setminus \pa M).
\label{ma.4b}\end{equation}
\label{si.1}\end{lemma}

\begin{proof}
By \cite[(3.20) and (4.16)]{DM2014}, the heat kernel of $g$ satisfies a Gaussian bound, which is well-known to be equivalent to the existence of a constant $C>0$ such that \eqref{ma.4b} holds; see for instance \cite{Grigoryan}.
\end{proof}

Given an asymptotically conical metric, one can define corresponding H\"older spaces.  However, in some situations, see for instance \cite{CH2013,CMR2015}, it is more convenient to work with the weighted H\"older spaces associated to a conformally related metric, namely a b-metric in the sense of Melrose \cite{MelroseAPS}.  The same phenomenon arises for $\QAC$-metrics, since the H\"older spaces introduced and used in \cite{Joyce,DM2014} are really those associated to a metric conformal to a $\QAC$-metric.

\begin{definition}
Let $(M,\phi)$ be a $\QAC$-manifold with fibred corners and let $x_{\max}$ be the product of the boundary defining functions associated to all maximal boundary hypersurfaces of $M$.  A smooth quasi b-metric on $M$ ($\QCyl$-metric for short) is a metric $g_{\QCyl}$ of the form
$$
        g_{\QCyl}= x_{\max}^2 g_{\QAC}
$$
for some smooth $\QAC$-metric $g_{\QAC}$.
\label{mwfc.12}\end{definition}
As for a $\QAC$-metric, the space
\begin{equation}
  \cV_{\QCyl}(M)=\{ \xi \in \CI(M;TM) \; | \; \sup_{M\setminus \pa M} g_{\QCyl}(\xi,\xi)< \infty\}
\label{mwfc.13}\end{equation}
of smooth vector fields on $M$ uniformly bounded  with respect to $g_{\QCyl}$ is in fact a Lie algebra.  Indeed, we see from Definition~\ref{mwfc.2} and \eqref{mwfc.5a} that this space can alternatively be defined   as the space of  $b$-vector fields $\xi$ such that for each $i$,
\begin{itemize}
\item $\left. \xi \right|_{H_i}$ is tangent to the fibres of $\phi_i$ if $H_i$ is not maximal;
\item $\xi v_i\in \frac{v_i^2}{x_{\max}}\CI(M)$,
\end{itemize}
which is clearly closed under the Lie bracket.  In particular, the Lie algebra $\cV_{\QCyl}(M)$ does not depend on the choice of the $\QCyl$-metric for a fixed $\QAC$-manifold with fibred corners and for a fixed choice of boundary defining functions.  Looking at the corresponding universal enveloping algebra over $\CI(M)$, one can then define the space $\Diff^k_{\QCyl}(M)$ of $\QCyl$-differential operators of order $k$ to be the space of differential operators generated by $\CI(M)$ and products of up to $k$ elements of $\cV_{\QCyl}(M)$.  In particular, we obtain the following analog of Proposition~\ref{bg.1}.
\begin{proposition}
Each $\QCyl$-metric is a complete metric of infinite volume with bounded geometry.  
\label{bg.2}\end{proposition}
\begin{proof}
A $\QCyl$-metric is a particular example of a metric with Lie structure at infinity in the sense of \cite{ALN04}, so we can proceed as in the proof of Proposition~\ref{bg.1} to conclude that it  is complete of infinite volume with bounded curvature, as well as all of its covariant derivatives.  To show that it has positive injectivity radius, we can proceed by induction on the depth as in the proof of Proposition~\ref{bg.1}.  
\end{proof}

There are various functional spaces that we can associate to $\QAC$-metrics and $\QCyl$-metrics.  We will be particularly interested in Hölder spaces.  Recall that to a given complete metric $g$ on $M\setminus \pa M$, a Euclidean vector bundle $E\to M\setminus \pa M$ with a compatible choice of connection  and $k\in \bbN_0$, one can associate the space $\cC^k_g(M\setminus \pa M;E)$ comprising  continuous sections $f: M\setminus \pa M\to E$ such that
\begin{equation}
   \nabla^j f\in \cC^0(M\setminus \pa M; T_j^0(M\setminus \pa M)\otimes E)\quad \mbox{and}\quad \sup_{p\in M\setminus \pa M} |\nabla^j f(p)|_g<\infty, \quad \forall j\in \{0,\ldots,k\},
\label{mwfc.14}\end{equation}
where $\nabla$ denotes the covariant derivative induced by the Levi-Civita connection of $g$ and the connection on $E$, $|\cdot|_g$ is the norm induced by the metric $g$ and the Euclidean structure on $E$, and
$$
T_j^0(M\setminus \pa M)= \underset{\mbox{$j$-times}}{\underbrace{T^*(M\setminus \pa M)\otimes \cdots \otimes T^*(M\setminus \pa M)}}.
$$
The space $\cC^k_g(M\setminus \pa M;E)$ is in fact a Banach space with norm given by
\begin{equation}
 \| f\|_{g,k} := \sum_{j=0}^k \sup_{p\in M\setminus \pa M} |\nabla^jf(p)|_g.
\label{mwfc.15}\end{equation}
Taking the intersection over all $k$, we also get the Fréchet space
$$
    \CI_g(M\setminus \pa M;E)= \bigcap_{k\in \bbN_0} \cC^k_g(M\setminus \pa M;E).
$$
For $\alpha\in (0,1]$ and $k\in \bbN_0$, we can also consider the H\"older space $\cC^{k,\alpha}_g(M\setminus \pa M;E)$ of functions $f\in \cC^k_g(M\setminus \pa M;E)$ such that
$$
  [\nabla^k f]_{g,\alpha}:= \sup \left\{  \frac{\left| P_{\gamma}(\nabla^k f (\gamma(0)))- \nabla^kf(\gamma(1)) \right|}{\ell(\gamma)^{\alpha}} \quad  | \quad \gamma\in \CI([0,1]; M\setminus \pa M), \; \gamma(0)\ne \gamma(1)\right\}<\infty,
$$
where $P_{\gamma}: \left. T^0_k(M\setminus \pa M)\otimes E\right|_{\gamma(0)}\to \left. T^0_k(M\setminus \pa M)\otimes E\right|_{\gamma(1)}$ is the parallel transport along $\gamma$ and $\ell(\gamma)$ is the length of $\gamma$ with respect to the metric $g$.  This is also a Banach space with norm given by
\begin{equation}
   \| f\|_{g,k,\alpha}:= \| f\|_{g,k} + [\nabla^k f]_{g,\alpha}.
\label{mwfc.16}\end{equation}
For $\rho\in \CI(M\setminus \pa M)$ a positive function, we can also consider the weighted version
\begin{equation}
  \rho\cC^{k,\alpha}_{g}(M\setminus \pa M;E) := \{   f \; | \; \frac{f}{\rho} \in \cC^{k,\alpha}_g(M\setminus \pa M;E)\} \quad \mbox{with norm} \quad \|f\|_{\rho\cC^{k,\alpha}_g}:= \left\| \frac{f}{\rho} \right\|_{g,k,\alpha}.
   \label{mwfc.17}\end{equation}

By choosing $g=g_{\QAC}$ to be a smooth $\QAC$-metric, we get in particular the Banach spaces $$
\cC^k_{\QAC}(M\setminus \pa M;E) \quad  \mbox{and} \quad \cC^{k,\alpha}_{\QAC}(M\setminus\pa M;E).
$$  Notice however that they do not  correspond to the Hölder spaces of \cite[\S~9]{Joyce} and \cite{DM2014}.  Indeed, \cite[\S~9]{Joyce} and \cite{DM2014} consider instead weighted versions of  the Banach spaces  $\cC^k_{\QCyl}(M\setminus \pa M;E)$ and $\cC^{k,\alpha}_{\QCyl}(M\setminus\pa M;E)$  obtained by choosing $g=g_{\QCyl}= x_{\max}^2g_{\QAC}$ to be a smooth $\QCyl$-metric on $M$.  The reason for this choice is that one can obtain nicer mapping properties for elliptic $\QAC$-operators when acting on weighted $\QCyl$-Hölder spaces.  Results stated in terms of $\QCyl$-Hölder spaces are also more precise, since, as one can easily check, there are continuous strict inclusions
\begin{equation}
       \cC^k_{\QCyl}(M\setminus \pa M;E) \subset \cC^k_{\QAC}(M\setminus \pa M;E), \quad \cC^{k,\alpha}_{\QCyl}(M\setminus \pa M;E) \subset \cC^{k,\alpha}_{\QAC}(M\setminus \pa M;E).
\label{mwfc.18}\end{equation}
As in \cite[\S~9]{Joyce} and \cite{DM2014}, we will mostly deal with $\QCyl$-Hölder spaces and their weighted counterparts.  Still, we will also have to deal with $\QAC$-Hölder spaces. The following lemma will be helpful in these cases in bringing the discussion back to $\QCyl$-Hölder spaces.

\begin{lemma}
For $0<\delta<1$, there is a continuous inclusion $x_{\max}^{\delta}\cC^{0,1}_{\QAC}(M\setminus \pa M;E)\subset \cC^{0,\alpha}_{\QCyl}(M\setminus \pa M;E)$ for $\alpha\le \delta$.
\label{mwfc.19}\end{lemma}
\begin{proof}
Let $g_{\QAC}$ be a choice of smooth $\QAC$-metric  and consider the conformally related $\QCyl$-metric $g_{\QCyl}:=x^2_{\max}g_{\QAC}$.  Given $\gamma\in \CI([0,1];M\setminus \pa M)$, let $\ell_{\QAC}(\gamma)$ and $\ell_{\QCyl}(\gamma)$ denote the length of $\gamma$ with respect to the metrics $g_{\QAC}$ and $g_{\QCyl}$.  Then using the Hölder inequality with $p=\frac1{\delta}$ and $q=\frac1{1-\delta}$, observe that 
\begin{equation}
\begin{aligned}
    \left| \frac1{x_{\max}(\gamma(0))^{\delta}}-\frac{1}{x_{\max}(\gamma(1))^{\delta}} \right| & = \delta\left| \int_{\gamma}\frac{ dx_{\max}}{x_{\max}^{1+\delta}} \right|= \delta\left| \int_{\gamma}\frac{ dx_{\max}}{x_{\max}^{2\delta}x_{\max}^{1-\delta}} \right|  \le \delta \left| \int_{\gamma} \frac{dx_{\max}}{x_{\max}^2} \right|^{\delta} \left| \int_{\gamma} \frac{dx_{\max}}{x_{\max}} \right|^{1-\delta} \\
    &\le\delta (K \ell_{\QAC}(\gamma))^{\delta}(K\ell_{\QCyl}(\gamma))^{1-\delta} \\
    &=  \delta K \ell_{\QAC}(\gamma)^{\delta}\ell_{\QCyl}(\gamma)^{1-\delta} 
    \end{aligned}   
  \label{ineq.1}\end{equation}
for some constant $K>0$ depending on $g_{\QAC}$ and $g_{\QCyl}$, but not on the path $\gamma$.  
 
Now, for $f\in x_{\max}^{\delta}\cC^{0,1}_{\QAC}(M\setminus \pa M;E)$,  consider the positive constant $C:= \left\| \frac{f}{x_{\max}^{\delta}}\right\|_{g_{\QAC},0,1}$.  Then, for any path $\gamma\in \CI([0,1];M\setminus \pa M)$, we have that
\begin{equation}
\begin{aligned}
     2 C\min \{\ell_{\QAC}(\gamma),1\} &\ge \left| \frac{P_{\gamma}(f(\gamma(0)))}{x_{\max}(\gamma(0))^{\delta}}- \frac{f(\gamma(1))}{x_{\max}(\gamma(1))^{\delta}} \right| \\
     &= \left| \frac{P_{\gamma}(f(\gamma(0)))}{x_{\max}(\gamma(0))^{\delta}}-\frac{f(\gamma(1))}{x_{\max}(\gamma(0))^{\delta}}+ \frac{f(\gamma(1))}{x_{\max}(\gamma(0))^{\delta}} -  \frac{f(\gamma(1))}{x_{\max}(\gamma(1))^{\delta}} \right| \\
     &\ge  \frac{|P_{\gamma}(f(\gamma(0)))-f(\gamma(1))|}{x_{\max}(\gamma(0))^{\delta}} - \|f\|_{g_{\QAC},0} \left| \frac1{x_{\max}(\gamma(0))^{\delta}}-\frac{1}{x_{\max}(\gamma(1))^{\delta}} \right| \\
     &\ge \frac{|P_{\gamma}(f(\gamma(0)))-f(\gamma(1))|}{x_{\max}(\gamma(0))^{\delta}} -C K\delta \ell_{\QAC}(\gamma)^{\delta}\ell_{\QCyl}(\gamma)^{1-\delta},
 \end{aligned}     
 \label{ineq.2}\end{equation} 
where we have used \eqref{ineq.1} and   the inequality $\|f\|_{g_{\QAC},0}\le  \left\| \frac{f}{x_{\max}^{\delta}}\right\|_{g_{\QAC},0,1}=C$ in the last step.  Thus, we infer from \eqref{ineq.2} that for all $\gamma\in \CI([0,1];M\setminus \pa M)$, 
\begin{equation}
    |P_{\gamma}(f(\gamma(0)))-f(\gamma(1))|\le 2C x_{\max}(\gamma(0))^{\delta}\min \{\ell_{\QAC}(\gamma),1\}  +C K\delta x_{\max}(\gamma(0))^{\delta} \ell_{\QAC}(\gamma)^{\delta}\ell_{\QCyl}(\gamma)^{1-\delta}.
\label{ineq.4}\end{equation}
Thus, if $\ell_{\QCyl}(\gamma)\ge 1$, then
\begin{equation}
\frac{|P_{\gamma}(f(\gamma(0)))-f(\gamma(1))|}{\ell_{\QCyl}(\gamma)^{\alpha}}\le |P_{\gamma}(f(\gamma(0)))-f(\gamma(1))| \le 2\|f\|_{g_{\QAC},0}\le 2\left\|\frac{f}{x_{\max}^{\delta}}\right\|_{g_{\QAC},0,1}=2C\le 2C +CK\delta.
\label{ineq.3}\end{equation}
If instead $\ell_{\QCyl}(\gamma)\le 1$,  let $t_{\min}\in[0,1]$ be such that 
$$
     x_{\max}(\gamma(t_{\min}))= \min\{x_{\max}(\gamma(t))\; ; \; t\in [0,1]\},
$$ 
so that 
$$
\left| \log \left( \frac{x_{\max}(\gamma(t_{\min}))}{x_{\max}(\gamma(0))} \right) \right| \le \left| \int_0^{t_{\min}} \frac{dx_{\max}}{x_{\max}}\circ \gamma \right| \le K \ell_{\QCyl}(\gamma)\le K \quad \Longrightarrow \quad \frac{x_{max}(\gamma(0))}{x_{\max}(\gamma(t_{\min}))}\le e^{\ell_{K\QCyl}(\gamma)}\le e^K,
$$
where $K>0$ is the constant occurring in \eqref{ineq.1}.
Then \eqref{ineq.4} yields
\begin{equation}
\begin{aligned}
\frac{|P_{\gamma}(f(\gamma(0)))-f(\gamma(1))|}{\ell_{\QCyl}(\gamma)^{\alpha}} &\le \frac{2Cx_{\max}(\gamma(0))^{\delta} \min\{\ell_{\QAC}(\gamma),1\}}{\ell_{\QCyl}(\gamma)^{\alpha}} +
   \frac{CK\delta x_{\max}(\gamma(0))^{\delta} \ell_{\QAC}(\gamma)^{\delta}\ell_{\QCyl}(\gamma)^{1-\delta}}{\ell_{\QCyl}(\gamma)^{\alpha}} \\
     & \le  \frac{2Cx_{\max}(\gamma(0))^{\delta} \min\{\ell_{\QAC}(\gamma),1\}}{(x_{\max}(\gamma(t_{\min})))^{\alpha}\ell_{\QAC}(\gamma)^{\alpha}} 
     +\frac{CK\delta x_{\max}(\gamma(0))^{\delta} \ell_{\QAC}(\gamma)^{\delta}\ell_{\QCyl}(\gamma)^{1-\alpha}}{(x_{\max}(\gamma(t_{\min})))^{\delta}\ell_{\QAC}(\gamma)^{\delta}}  \\
     &\le 2Ce^{K\alpha} +CK\delta \ell_{\QCyl}(\gamma)^{1-\alpha}e^{K\delta} \le (2 +K\delta)Ce^{K\delta},
\end{aligned}
\label{mwfc.20}\end{equation}
since $\alpha\le \delta$.   
Hence, combining  \eqref{ineq.3} with \eqref{mwfc.20} and  taking the supremum over $\gamma$ yields
$$
     [f]_{g_{\QCyl},0,\alpha} \le (2+K\delta)C e^{K\delta}= (2+K\delta)e^{K\delta} \left\| \frac{f}{x_{\max}^{\delta}}\right\|_{g_{\QAC},0,1} ,
$$
from which the result follows.
\end{proof}

\section{Manifolds with fibred corners coming from Sasaki-Einstein orbifolds}  \label{cr.0}

Let $Z$ be a closed orbifold of real dimension $2n+1$.  A \textbf{Ricci-flat Kähler cone metric} on the Cartesian product $C:=\bbR^+ \times Z$ is a cone metric 
$$
      g_C= dr^2+ r^2 g_Z,
$$
where $g_Z$ is some Riemannian metric on $Z$, together with a complex structure $J_C$ on $C$ such that $g_C$ is Ricci-flat and Kähler with respect to $J_C$ with Kähler form given by
$$
     \omega_C= \frac{\sqrt{-1}}{2}\pa \db r^2.
$$
In particular, when such a metric $g_C$ and complex structure $J_C$ exist, the corresponding canonical line bundle $K_C$ of $C$ is flat with respect to the Chern connection induced from $g_C$ and $J_C$.

Notice then that the metric $g_Z$ is  Sasakian and that there is a corresponding \textbf{Reeb vector field} on $Z$ given by
$$
     \xi= J_C \frac{\pa}{\pa r}.
$$ 
The orbits of the flow of this vector field induce a foliation called the \textbf{Reeb foliation}.  In the present paper, we will always assume that the Ricci-flat Kähler cone $(C,g_C,J_C)$ is \textbf{quasi-regular} in the following sense.
\begin{definition}
The Ricci-flat Kähler cone $(C,g_C,J_C)$ is \textbf{quasi-regular} if there exists a Kähler-Einstein Fano orbifold $(D,g_D)$ and a holomorphic line bundle $L$ over $D$ with Hermitian metric $h_L$ such that $L\setminus D$ is biholomorphic to $(C,J_C)$  with the properties that on $L\setminus D,$
\begin{enumerate}
\item The radial function $r$ is given by 
$$
         r=\|\cdot\|^{\frac1q}_{h_L} \quad \mbox{ on } \;L\setminus D
 $$
for some $q\in \bbN$ and $Z= r^{-1}(1)$;
 \item  The Reeb foliation corresponds to the fibres of the unit circle bundle of $(L,h_L)$, that is,
\begin{equation}
\xymatrix{
\bbS^1 \ar[r]  & Z \ar[d]^{\nu_L} \\
                       & D;
}
\label{cr.4}\end{equation}
\item The map $\nu_L$ in \eqref{cr.4} is an orbifold  Riemannian submersion from $(Z,g_Z)$ to $(D,g_D)$;
\item For some $p\in \bbN$, the line bundle $K_C^p$ has a nowhere vanishing holomorphic section $\Omega_C^p\in H^0(C;K_C^p)$ which is parallel with respect to the induced Chern connection from $K_C$ and $J_C$.  In particular, there is a non-zero constant $c_p$ such that   
\begin{equation}
    (\omega_C^{n+1})^p= c_p\Omega_C^p\wedge \overline{\Omega^p_C},
\label{cr.3}\end{equation}
where $(\omega_C^{n+1})^p$ is the tensorial product of $\omega_C^{n+1}$ with itself $p$ times seen as a section of $K_C\otimes \overline{K}_C$.
\end{enumerate}
Furthermore, when $p=1$ in the previous assertion, we say that $(C,g_C,J_C)$ is a \textbf{quasi-regular Calabi-Yau cone}.  
\label{qr.1}\end{definition}
 As a simple computation shows, the Kähler-Einstein constant of the metric $g_D$ is completely determined by the previous definition, namely
$$
    \Ric(g_D)= 2(n+1)g_D.
$$
Furthermore, the Kähler form of $g_D$ is such that 
\begin{equation}
     \sqrt{-1} \Theta_{h_L}= -2q \omega_D,
\label{qr.2}\end{equation}
where $\Theta_{h_L}$ is the curvature of $(L, h_L)$ with respect to the Chern connection.
\begin{remark}
Let us clarify that we assume that  $L$ is locally trivial in the sense that in a neighborhood of any point $p$ on $D$, there are orbifold charts $\psi: \cU\to  \bbC^n/\Gamma$ and   $\psi_L: p_{L}^{-1}(\cU)\to (\bbC^n\times \bbC)/\Gamma$ inducing a commutative diagram
\begin{equation}
\xymatrix{
   p_L^{-1}(\cU)\ar[d]^{p_L} \ar[r]^-{\psi_L} & (\bbC^n\times \bbC)/\Gamma \ar[d]^{\pr_1} \\
   \cU \ar[r]^{\psi}  &  \bbC^n/\Gamma,
}
\label{sf.1b}\end{equation}
where $\pr_1$ is the projection on the first factor.  In particular, although we do assume that $C$ and $Z$ may have orbifold singularities, not all singularities of $D$ correspond to singularities of $Z$ and $C$.  Indeed, even if $Z$ were smooth, as a space of leaves, $D$ could still be singular.  Also, notice that a point $p\in Z$ is singular if and only if all points of $\nu_L^{-1}(\nu_L(p))$ are singular.       
\label{sf.1}\end{remark} 

In fact, a simple way to construct examples of quasi-regular Calabi-Yau cone metrics is to start off with a K\"ahler-Einstein Fano orbifold $(D,g_D)$ and to apply the Calabi ansatz \cite{Calabi, LeBrun}, which yields  a Calabi-Yau cone
$(K_D\setminus D, g_C)$ with the zero section $D \subset K_D$ corresponding to the apex of the cone, where $Z$ is then the total space of the unit circle bundle of $K_D$ over $D$.  The Kähler form $\omega_C$ of $g_C$ can then be written explicitly in terms of the Kähler-Einstein metric $g_D$ as
\begin{equation}
   \omega_C= \frac{\sqrt{-1}}2 \pa \db \| \cdot \|^{\frac{2}{n+1}}_{g_D},
\label{cr.2}\end{equation}
where $\|\cdot \|_{g_D}$ is the Hermitian metric on $K_D$ induced by $g_D$ seen as a function on $K_D\setminus D$.  

To describe the natural holomorphic volume form associated to this Calabi-Yau cone, recall that the space of sections $H^0(K_D; \pi^*(K_D))$ has a tautological element $\varpi$ given by
$\varpi_p= \pi^*p $ for $p\in K_D$, where $\pi: K_D\to D$ is the canonical projection.   For any choice of local coordinates $(z_1,\ldots, z_n)$ on $D$, we have a local section $dz_1\wedge\cdots\wedge dz_n$ of $K_D$, and hence local coordinates $(z_1,\ldots,z_n,v)$ on $K_D$ with $(z_1,\ldots,z_n,v)$ corresponding to $vdz_1\wedge\cdots \wedge dz_n\in \left. K_D\right|_{(z_1,\ldots,z_n)}$.  In these local coordinates, $\varpi$ is simply given by
$$
   \varpi = v dz_1\wedge \cdots \wedge dz_n.
$$
Taking the exterior derivative of the tautological element $\varpi$ yields the canonical holomorphic volume form $\Omega_C$ of $K_D$, namely,
$$
     \Omega_C := d\varpi= \pa \varpi \in H^0(K_D; K_{K_D}).
$$
The fact that the cone metric $g_C$ is Calabi-Yau amounts to the fact  that
$$
     \omega^{n+1}_C= c\Omega_C\wedge \overline{\Omega}_C
$$
for some fixed non-zero constant $c$.

More generally, if $p_L:L\to D$ is a $q$th root of $K_D$, namely if  $L^{\otimes q}= K_D$, then the natural map
\begin{equation}
 \begin{array}{lccc} Q : &L\setminus D &\to & K_D\setminus D \\
                                               &    \sigma   & \mapsto & \sigma^{\otimes q}
                                                   \end{array}
\label{cr.3b}\end{equation}
 is a $\bbZ_q$-cover of $K_D\setminus D$, so that  pulling back $\Omega_C$ and $\omega_C$ to $L\setminus D$ endows $L\setminus D$ with an orbifold Calabi-Yau cone structure.  Conversely, given such a line bundle $L$,  we have that $\bbZ_p$, seen as the group of $p$th-roots of unity, naturally acts by isometry  on $(L\setminus D, Q^*g_C)$ via the natural $\bbS^1$-action, and the quotient, which is naturally identified with $L^p\setminus D$, naturally inherits the structure of a quasi-regular Ricci-flat Kähler cone since the $\bbZ_p$-action is trivial on $(Q^*\Omega_C)^{ p}\in H^0(L\setminus D; K_{L\setminus D}^p)$.

Now, suppose that $(C,g_C)$ is a quasi-regular Ricci-flat Kähler cone with holomorphic parallel section $\Omega_C^p\in H^0(C;K_C^p)$ as above and suppose that $X$ is a Kähler orbifold with a nowhere vanishing holomorphic section $\Omega_X^{p}\in H^0(X;K_X^p)$  such that for some compact subset $\cK\subset X$, there is a biholomorphism
\begin{equation}
X\setminus \cK\cong (\kappa,\infty)\times Z \subset C
\label{bih.1}\end{equation}
for some $\kappa\ge 0$ identifying $\Omega_X^p$ with $\Omega_C^p$.  Assume further that the singular set of $X$ has complex codimension at least $2$ and that $X$ admits a local product Kähler crepant resolution $\hX\to X$ in the sense of \cite{Joyce} such that $\Omega_X^p$ lifts to a nowhere vanishing holomorphic section $\Omega_{\hX}^p\in H^0(\hX; K^p_{\hX})$. 
In the remainder of this section, we will explain how the resolution $\hX$ can then be naturally compactified as a manifold with fibred corners.  However, before doing that, let us give some examples of such a space $X$.  
\begin{example}
We can take $X= \bbC^{n+1}/\Gamma$ for some finite subgroup $\Gamma\subset \SU(n+1)$.  It is not automatic that $X$ will admit a local product Kähler crepant resolution, but for many choices of $\Gamma$ it will  \cite[\S 6.4-6.6]{Joyce}.  For instance, if $\Gamma\subset \SU(2)\subset \SU(n+1)$, there is automatically a crepant resolution by \cite{Kronheimer1989}.  On the other hand, we can take $(C,g_C)=((\bbC^{n+1}/\Gamma)\setminus \{0\},g_{E}),$ where $g_E$ is the Euclidean metric on $\bbC^{n+1}/\Gamma$ and $\cK=\{0\}$, to obtain trivially that $X\setminus \{0\}$ is biholomorphic to $C$.   
\end{example} 

\begin{example}
Let $(D,g_D)$ be a Kähler-Einstein Fano orbifold of complex dimension $n$ and assume that $D$ admits a local product (Kähler) crepant resolution $\hD$.  This automatically implies that the total space $K_{\hD}$ of the canonical line bundle of $\hD$ is a local product (Kähler) crepant resolution of $K_D$ and that the following diagram commutes,
\begin{equation}
\xymatrix{
      K_{\hD} \ar[r]^{\beta_K} \ar[d] & K_D \ar[d] \\
      \hD \ar[r]^{\beta} & D,
}
\label{cr.1}\end{equation}
where the vertical maps are the canonical projections and the horizontal maps are the blow-down maps of the local product resolutions.  Thus, in this setting, we can take $X=K_D$ and the Calabi-Yau cone $C=K_D\setminus D$ with Kähler metric given by \eqref{cr.2}. 
\label{exa.1}\end{example}
\begin{remark}
In the previous example, we could for instance take $D=\bbC\bbP^n/\Gamma$ with $\Gamma\subset \SU(n+1)$ a finite subgroup.  In this case, the Fubini-Study metric descends to a Kähler-Einstein metric $g_D$ on $D$. On the other hand, $D$ does not always admit a local product Kähler crepant resolution, but in some cases it does, for instance in Example~\ref{ap.3} in the appendix or when $n=2$ and $\Gamma\subset \SU(3)$ is the subgroup generated by the diagonal matrix with diagonal entries given by $e^{\frac{2\pi i}3}, e^{-\frac{2\pi i}3}$ and $1$.     We could also take $D$ to be a Kähler-Einstein  log del Pezzo surface with canonical singularities of degree at most $4$; see \cite{OSS2016} for examples.  According to \cite[p.165]{OSS2016}, $D$ then has only singularities of type $D_4$ and $A_k$ for $k\le 7$, so automatically admits a local product Kähler crepant resolution by proceeding as in \cite{Kronheimer1989}.  Finally, notice that if $(D_1,g_{D_1})$ and $(D_2,g_{D_2})$ are two examples of Kähler-Einstein Fano orbifolds admitting local product Kähler crepant resolutions, then after scaling the metrics, we can assume without loss of generality that $\Ric(g_{D_i})= g_{D_i}$ for $i=1,\,2$, so that the Cartesian product $(D_1\times D_2, g_{D_1}\times g_{D_2})$ is another example of a Kähler-Einstein Fano orbifold admitting a local Kähler crepant resolution.  
\label{exa.2} \end{remark}

To describe the natural compactification of $\hX$, observe first that we can compactify $X$ into an orbifold with boundary $X_{\sc}$ such that the biholomorphism \eqref{bih.1} extends to a diffeomorphism
$$
          X_{\sc}\setminus {\cK}\cong  (\kappa,\infty]\times Z.
$$
Here, we mean that $X_{\sc}$ is an \textbf{orbifold with boundary} in the sense that it is locally modeled by charts of the form $\bbR^{2n+2}/\Gamma_1$ or $(\bbR^{2n+1}/\Gamma_2)\times [0,\infty)$ with $\Gamma_1\subset \GL(2n+2,\bbR)$ and $\Gamma_2\subset \GL(2n+1,\bbR)$ finite subgroups.  Notice in particular that with this definition, the orbifold singularities are always transversal to the boundary.   Clearly, we have that $\pa X_{\sc}\cong Z$ and that the function $x:=\frac1{r}$ defined near $\pa X_{\sc}$ is a natural choice of boundary defining function.

As an orbifold, $\pa X_{\sc}$ has the structure of a smoothly stratified space with strata corresponding to the various types of orbifold singularities.  In particular, the strata are partially ordered by the relation
$$
        s_1\le s_2  \Longleftrightarrow s_1\subset \overline{s}_2.
$$
  In fact, by Remark~\ref{sf.1}, each stratum $\overline{s}_i$ of $\pa X_{\sc}$ is naturally equipped with an $\bbS^1$-orbibundle structure
\begin{equation}
\xymatrix{
\bbS^1 \ar[r]  & \overline{s}_i \ar[d]^{\nu_i} \\
                       & \overline{\sigma}_i,
}
\label{cr.5}\end{equation}
where $\overline{\sigma}_i:= \nu_i(\overline{s}_i)$.  
The partial order on the strata of $\pa X_{\sc}$ gives us a systematic way of blowing up the strata of $\pa X_{\sc}$ in $X_{\sc}$ in the sense of \cite{ALMP2012}.  Namely, let $\{s_1,\ldots, s_k\}$ be an exhaustive list of all the strata of $\pa X_{\sc}$ compatible with the partial order in the sense that
$$
      s_i\le s_j \Longrightarrow i\le j.
$$
In particular, $s_k$ will be the regular stratum, so that $\pa X_{\sc}= \overline{s}_k$.  One constructs a natural space $\tX_{\sc}$  out of $X_{\sc}$ by blowing up all of the strata of $\pa X_{\sc}$ in $X_{\sc}$ except the regular one; that is,
\begin{equation}
   \tX_{\sc} = [X_{\sc}; \overline{s}_1, \overline{s}_2, \ldots, \overline{s}_{k-1}] \quad \mbox{with blow-down map} \quad \beta_{\sc}: \tX_{\sc}\to X_{\sc}.
\label{cr.6b}\end{equation}
Each of these blow-ups creates a new boundary hypersurface, so that $\tX_{\sc}$ is not an orbifold with boundary, but rather an \textbf{orbifold with corners}, by which we mean that $\tX_{\sc}$ is locally modeled on charts of the form $\bbR^{2n+2-q}/\Gamma\times [0,\infty)^q$ for $q\in\{0,1,\ldots, 2n+2\}$ and $\Gamma\subset \GL(2n+2-q,\bbR)$ a finite subgroup.
By construction, $\tX_{\sc}$ has a boundary hypersurface $H_i$ for each stratum $s_i$ of $\pa X_{\sc}$.  Since we blow up each $\overline{s}_i$ in $X_{\sc}$ instead of in $\pa X_{\sc}$, notice that each boundary hypersurface
$H_i$ remains an orbifold for each $i<k$, whereas $H_k$ is the manifold with fibred corners that resolves the stratified space $\pa X_{\sc}$, hence contains no orbifold singularities.  As we are about to show, there is a natural fibre bundle structure on each boundary hypersurface, making $\tX_{\sc}$ an orbifold with fibred corners in the following sense.

\begin{definition}
Let $M$ be an orbifold with corners and suppose that each of its boundary hypersurfaces $H_i$ is a fibre bundle $\phi_i: H_i\to S_i$ whose base $S_i$ is a manifold with corners (hence has no orbifold singularities) and whose fibres are all orbifolds with fibred corners.  If $\phi$ denotes the collection of fibre bundle maps $\phi_i$, then we say that $(M,\phi)$ is an \textbf{orbifold with fibred corners} if there exists a partial order on the boundary hypersurfaces such that the same conditions as in Definition~\ref{mwfc.1} are satisfied, namely:
\begin{itemize}
\item  Any subset $I$ of boundary hypersurfaces such that $\bigcap_{i\in I}H_i\ne \emptyset$ is totally ordered;

\item If $H_i<H_j$, then $H_i\cap H_j\ne \emptyset$, $\left. \phi_i\right|_{H_i\cap H_j}: H_i\cap H_j\to S_i$ is a surjective submersion and $S_{ji}:= \phi_j(H_i\cap H_j)$ is one of the boundary hypersurfaces of the manifold with corners $S_j$.  Moreover, there is a surjective submersion $\phi_{ji}: S_{ji}\to S_i$ such that $\phi_{ji}\circ \phi_j=\phi_i$ on $H_i\cap H_j$.

\item The boundary hypersurfaces of $S_j$ are given by the $S_{ji}$ for $H_i<H_j$.
\end{itemize}
\label{owfc.1}\end{definition}

\begin{proposition}
The orbifold with corners $\tX_{\sc}$ itself has a natural orbifold with fibred corners structure.
\label{cr.6a}\end{proposition}
\begin{proof}
Before we perform the blow-ups for larger strata, the closure of the stratum $s_i$ lifts to a submanifold $S_i$ of $[X_{\sc}; \overline{s}_1,\ldots, \overline{s}_{i-1}]$, so that the blow-up face associated to $s_i$ is just the radial compactification $\overline{NS_i}$ of the normal bundle of $S_i$ (as a suborbifold of a boundary hypersurface of  $[X_{\sc}; \overline{s}_1,\ldots, \overline{s}_{i-1}]$), whose fibres are orbifolds of the form $\overline{V}_i= \overline{\bbC^{m_i}/\Gamma_i}$ with $m_i=n-\dim_{\bbC}\sigma_i$ and $\Gamma_i\subset \SU(n_i)$ a finite subgroup, that is,
\begin{equation}
\xymatrix{
\overline{V}_i \ar[r]  & \overline{NS_i} \ar[d]^{\phi_i}  \\
                       & S_i.
}
\label{cr.6}\end{equation}
The subsequent blow-ups of strata modify the face associated to $s_i$, but only in the fibres of \eqref{cr.6}, so that ultimately $H_i$ comes equipped with a fibre bundle structure
\begin{equation}
\xymatrix{
\tV_{i} \ar[r]  & H_i \ar[d]^{\phi_i}  \\
                       & S_i
}
\label{cr.7}\end{equation}
with $\tV_{i}$ obtained from $\overline{V}_i$ by blowing up the singular strata of $\pa\overline{V}_i$ in $\overline{V}_i$ in an order compatible with the partial order of the strata of $\pa \overline{V}_i$.
In other words, $\overline{V}_i$ is naturally an orbifold with boundary and $\tV_{i}$ is obtained from $\overline{V}_i$ in the same way that $\tX_{\sc}$ is obtained from $X_{\sc}$.

Notice that thanks to the order in which we do the blow-ups, $S_i$ is the natural manifold with fibred corners that resolves the closure $\overline{s}_i$ seen as a stratified space.  In particular, it has no orbifold singularities and the singularities of the face $H_i$ are all in the fibres of \eqref{cr.7}.  Now, clearly, since each boundary hypersurface of $\tX_{\sc}$ is associated to a stratum of $\pa X_{\sc}$, the partial order of the strata of $\pa X_{\sc}$ induces a partial order on the boundary hypersurfaces of $\tX_{\sc}$.  Moreover, we clearly have that
$$
     H_i\cap H_j\ne \emptyset \Longleftrightarrow s_i\le s_j \; \mbox{or} \; s_j\le s_i,
$$
and we see that the partial order on the boundary hypersurfaces of $\tX_{\sc}$ totally orders any subset $I$ of the boundary hypersurfaces with $\bigcap_{i\in I} H_i\ne \emptyset$. Finally, if $H_i< H_j$, then it follows from the order in which we blew up the strata that $\phi_j$ restricts to a surjective submersion on $H_i\cap H_j$ onto the boundary hypersurface $S_{ji}$ of $S_j$, and that the fibre bundle structure $\phi_{ji}: S_{ji}\to S_i$ coming from the orbifold with fibred corners of $S_j$ is such that $\phi_{ji}\circ \phi_j= \phi_i$.
\end{proof}

The space $\tX_{\sc}$ provides a compactification of $X$, but it still has orbifold singularities.  On the other hand, since we only blew up strata of $\pa X_{\sc}$, we still have a natural identification
$\tX_{\sc}\setminus \pa \tX_{\sc}= X$, so we can still use the local product Kähler crepant resolution of $X$ to remove the orbifold singularities.

\begin{theorem}
 The local product Kähler crepant resolution $\beta_X: \hX\to X$ naturally extends to give a resolution $\beta_{\QAC}: \hX_{\QAC}\to \tX_{\sc}$ of the orbifold with fibred corners $\tX_{\sc}$ by a manifold with fibred corners $\hX_{\QAC}$, inducing the following commutative diagram
$$
\xymatrix{ \hX \ar[d]^-{\beta_X}\ar @{^{(}->}[rr] & & \hX_{\QAC}\ar[d]^-{\beta_{\QAC}} \\
     X \ar @{{(}->}[r] & X_{\sc} & \ar[l]^-{\beta_{sc}} \tX_{\sc}.
}    
$$  
The space $\hX_{\QAC}$ is called the \textbf{$\QAC$-compactification of $\hX$} and we say that it is a \textbf{$\QAC$-resolution of $\tX_{\sc}$.}
\label{cr.8}\end{theorem}
\begin{proof}
If $\sigma_1=\nu_1(s_1)$ is a point, then $S_1=s_1=\bbS^1$ and  $H_1=\bV_1\times\bbS^1$ with $\phi_1: \bV_1\times \bbS^1\to \bbS^1$ given by the projection on the right factor and $\overline{V}_1= \overline{\bbC^{n}/\Gamma_1}$ for a certain finite subgroup $\Gamma_1\subset \SU(n)$ acting freely on $\bbC^n\setminus \{0\}$.    Here, $\bbC^n/\Gamma_1$ can be seen as an orbifold chart for the corresponding stratum $\sigma_1$ in $D$.  Clearly then, the crepant resolution of $X$ extends to one for $\tX_{\sc}$ and $H_1$.      If $\hH_1$ denotes the induced resolution of $H_1$, then $\hH_1=\hY_1\times \bbS^1$, where $\hY_1$ is the radial compactification of  the crepant resolution $Y_1$ of $\bbC^n/\Gamma_1$, so that we still have a natural fibre bundle $\hphi_1: \hH_1\to S_1$ which is just the projection $\hY_1\times \bbS^1\to \bbS^1$ on the right factor.
If $\sigma_1$ is not a point but $\overline{s}_1$ is still a singularity of relative depth one,  then we can proceed in the same way, that is, the crepant resolution of $X$ clearly extends to give a resolution of $\tX_{\sc}$ near $H_1$ with resolution $\hH_1$ of $H_1$ obtained from \eqref{cr.6} by replacing each fibre $\overline{V}_1$ by its crepant resolution $\hY_1$, that is,
 \begin{equation}
\xymatrix{
 \hY_1 \ar[r]  & \hH_1 \ar[d]^{\hphi_1}  \\
                       & S_1.
}
\label{cr.10}\end{equation}
More generally, using the fact that the resolution $\hX\to X$ is a local product resolution, we see by induction on the depth of $X_{\sc}$ that $\hH_i$ will be the total space of a fibre bundle
\begin{equation}
\xymatrix{
 \hY_i \ar[r]  & \hH_i \ar[d]^{\hphi_i}  \\
                       & S_i
}
\label{cr.10}\end{equation}
with $\hY_i$ the $\QAC$-resolution of $\tV_i$.  Indeed, the local product Kähler crepant resolution of $X$ naturally induces one on $V_i$, and since $\tV_i$ is an orbifold with  fibred corners of smaller depth than $X_{\sc}$, we can assume by induction that the theorem already holds true for $\tV_i$.  

Since at each step the local product resolution $\hX\to X$ is only used fibrewise in the fibre orbibundle $\phi_i:H_i\to S_i$, we see that $\hX_{\QAC}$ is naturally a manifold with fibred corners with fibre bundle structure on the boundary hypersurface $\hH_i$ given by \eqref{cr.10} and with partial order on the boundary hypersurfaces of $\hX_{\QAC}$ induced by the one on the boundary hypersurface of $\tX_{\sc}$.
\end{proof}

Notice that for the face $\hH_k$ associated to the regular stratum $s_k$,  we have that $S_k=H_k=\hH_k$ and that the fibre bundle $\hphi_k:\hH_k\to S_k$ is just the identity map.  Since the only maximal stratum with respect to the partial order is the regular stratum, we see that $\hH_k$ is the only maximal boundary hypersurface of $\hX_{\QAC}$.

Using Lemma~\ref{mwfc.2ee}, we can also specify a natural Lie algebra of $\QAC$-vector fields on $\tX_{\sc}$, that is, a natural choice of $\QAC$-equivalence class of boundary defining functions. Indeed, let $\beta_0:=\beta_{\sc}:\tX_{\sc}\to X_{\sc}$ be the blow-down map and for each $i\ge 1$, consider the partial blow-down maps
\begin{equation}
  \beta_i: \tX_{\sc}\to [X_{\sc};\overline{s}_1,\ldots,\overline{s}_{i}] \quad \mbox{and} \quad \beta_i': [X_{\sc};\overline{s}_1,\ldots,\overline{s}_{i}]\to X_{\sc},
 \label{cr.12}\end{equation}
so that $\beta_0=\beta_i'\circ \beta_i$ for each $i$.  For each $i\ge 1$, consider the lift of $\pa X_{\sc}$ to $[X_{\sc};\overline{s}_1,\ldots,\overline{s}_{i}]$, namely
$$
            B_i:= \overline{(\beta_i')^{-1}(\pa X_{\sc}\setminus(\overline{s}_1\cup \cdots \cup \overline{s}_i))}.
$$
Let also $W_i$ be the boundary hypersurface of $[X_{\sc};\overline{s}_1,\ldots,\overline{s}_{i}]$ coming from the blow-up of $\overline{s}_i$ and let $\rho_0$ be any choice of boundary defining function for $B_0:=\pa X_{\sc}$ in $X_{\sc}$.  More generally,  proceeding recursively on $i$, choose a boundary defining function $\rho_i\in\CI([X_{\sc};\overline{s}_1,\ldots,\overline{s}_{i}])$ of $B_i$ such that $\rho_i$ is identically equal to the pull-back of $\beta_{i,i-1}^*\rho_{i-1}$ outside a small neighborhood $\cU_i$ of $W_i$ not intersecting the boundary hypersurfaces of $[X_{\sc};\overline{s}_1,\ldots,\overline{s}_{i}]$ disjoint from $W_i$, where
$$
\beta_{i,i-1}:[X_{\sc};\overline{s}_1,\ldots,\overline{s}_{i}]\to [X_{\sc};\overline{s}_1,\ldots,\overline{s}_{i-1}]
$$
is the blow-down map.  Then on $\tX_{\sc}$, the functions
\begin{equation}
      x_{i}= \frac{\beta_{i-1}^*\rho_{i-1}}{\beta_i^*\rho_i}, \; i<k \quad \mbox{and}\quad  x_k:= \rho_k
\label{cr.13}\end{equation}
are such that $x_i\in \CI(\tX_{\sc})$ is a boundary defining function for $H_i$ for each $i$.
\begin{lemma}
On $(\tX_{\sc},\phi)$, the Lie algebra of $\QAC$-vector fields specified by the choice of the $\QAC$-equivalence class of the boundary defining functions \eqref{cr.13} does not depend on the choice of the functions $\rho_i$, hence yields  a natural Lie algebra of $\QAC$-vector fields.
\label{cr.14}\end{lemma}
\begin{proof}
Suppose that for each $i$, $\rho_i$ and $\rho_i'$ are two different choices of boundary defining functions and suppose recursively that $\rho_i$, respectively $\rho_i'$, is identically equal to $\beta_{i,i-1}^*\rho_{i-1}$, respectively $\beta_{i,i-1}^* \rho_{i-1}'$,
outside a small neighborhood $\cU_i$ of $W_i$ not intersecting the boundary hypersurfaces of $[X_{\sc};\overline{s}_1,\ldots,\overline{s}_{i}]$ disjoint from $W_i$.  In this case, essentially by definition of the blow-down map, we have that for $j>i$,
\begin{equation}
\left.  \frac{\beta_i^*\rho_i}{\beta^{*}_i \rho_i'} \right|_{H_j}= \phi_j^*h_{ij}  \quad \mbox{for some}  \; h_{ij}\in \CI(S_j).
\label{cr.15}\end{equation}
On the other hand, thanks to the identification of $\rho_i$ with $\beta_{i,i-1}^*\rho_{i-1}$ outside $\cU_i$, we also have that for $H_j\ge H_i$,
\begin{equation}
         \left.  \frac{\beta_{i-1}^*\rho_{i-1}}{v_i} \right|_{H_j}= \phi_j^*f_{ij}  \quad \mbox{for some}  \; f_{ij}\in \CI(S_j), \; \mbox{where} \; v_i=\prod_{H_j\ge H_i} x_i.
\label{cr.16}\end{equation}
Of course, there is a similar statement for $\rho_i'$, so the combination of \eqref{cr.15} and \eqref{cr.16} allows us to apply Lemma~\ref{mwfc.2ee}, from which the result follows.
\end{proof}

Provided that we can choose the functions $\rho_i$ in such a way that their lifts $\beta_{\QAC}^*\beta_i^*\rho_i$ to $\hX_{\QAC}$ are smooth, we obtain  corresponding boundary defining functions on $\hX_{\QAC}$,
\begin{equation}
      \widehat{x}_{i}= \beta_{\QAC}^* x_i=\frac{ \beta_{\QAC}^*\beta_{i-1}^*\rho_{i-1}}{ \beta_{\QAC}^*\beta_i^*\rho_i}, \; i<k \quad \mbox{and}\quad   \widehat{x}_k:=\beta_{\QAC}^*x_k=  \beta_{\QAC}^*\rho_k.
\label{cr.17}\end{equation}

\begin{lemma}
The $\QAC$-equivalence class of the boundary defining functions \eqref{cr.17} does not depend on the choice of the functions $\rho_i$ that lift to be smooth on $\hX_{\QAC}$. 
\label{cr.18}\end{lemma}
\begin{proof}
Given two different choices $\rho_i$ and $\rho_i'$, we can proceed as in the proof of Lemma~\ref{cr.14} with \eqref{cr.15} and \eqref{cr.16} replaced respectively by
\begin{equation}
\left.  \frac{\beta^*_{\QAC}\beta_i^*\rho_i}{\beta_{\QAC}^*\beta^{*}_i \rho_i'} \right|_{\hH_j}= \hphi_j^*h_{ij}  \quad \mbox{for some}  \; h_{ij}\in \CI(S_j)
\label{cr.15b}\end{equation}
and
\begin{equation}
         \left.  \frac{\beta_{\QAC}^*\beta_{i-1}^*\rho_{i-1}}{\widehat{v}_i} \right|_{\hH_j}= \hphi_j^*f_{ij}  \quad \mbox{for some}  \; f_{ij}\in \CI(S_j), \; \mbox{where} \; \widehat{v}_i=\prod_{\hH_j\ge \hH_i} \widehat{x}_i.
\label{cr.16b}\end{equation}
\end{proof}

Thus, to see that $\hX_{\QAC}$ comes endowed with a natural Lie algebra of $\QAC$-vector fields, we need to show that the functions $\rho_i$ can be chosen in such way that they lift to be smooth on $\hX_{\QAC}$, a discussion that we postpone until Lemma~\ref{bdfcr.1} in the next section.

\section{The Ricci-flat Kähler cone metric seen as a $\QAC$-metric}\label{gi.0}

Continuing with the setup of the previous section, we will show in this section how an orbifold Ricci-flat Kähler cone metric can be seen as a $\QAC$-metric on $\tX_{\sc}$ in a neighborhood of $\pa \tX_{\sc}$. Let $p\in C$ be a singular point. Then by Remark~\ref{sf.1}, $C=L\setminus D$ where $L$ is an orbifold holomorphic line bundle over a Kähler-Einstein Fano orbifold $D$, and we can find orbifold charts as in \eqref{sf.1b}.    However, since $p$ and the entire fibre of $L$ containing $p$ is singular, we know by averaging that the holomorphic line bundle $\pr_1:\bbC^n\times \bbC\to \bbC^n$ in \eqref{sf.1b} has a $\Gamma$-invariant holomorphic section that does not vanish near $0\in \bbC^n$.  Thus, this means that without loss of generality, we can assume that there is  a finite subgroup $\Gamma_1\subset \SU(n)$ acting linearly on $\bbC^n$   together with  orbifold charts $\psi: \cU\to \bbC^n/\Gamma_1$ and $\psi_L:p_L^{-1}(\cU)\to \bbC^n/\Gamma_1\times \bbC$ inducing the commutative diagram
\begin{equation}
\xymatrix{
   p_L^{-1}(\cU)\ar[d]^{p_L} \ar[r]^-{\psi_L} & \bbC^n/\Gamma_1 \times \bbC \ar[d]^{\pr_1} \\
   \cU \ar[r]^{\psi}  &  \bbC^n/\Gamma_1,
}
\label{gi.1a}\end{equation}
and such that $\psi_L(p)=(0,1)$.
Let $z=(z_1,\ldots,z_n)$ be the coordinates on $\bbC^n\setminus \Gamma_1$ and $v$ the coordinate on $L$.  Then the Kähler form of the Ricci-flat Kähler cone metric $g_C$ takes the form
\begin{equation}
  \omega_C= \frac{\sqrt{-1}}{2} \pa\db |v|^{\frac{2}{q}}f(z,\bz) 
\label{gi.2}\end{equation}
for some positive smooth function $f$ and some positive $q\in \bbQ$.  
Let $W_1\subset \bbC^n$ be the subspace of points fixed by each element of $\Gamma_1$ and suppose without loss of generality that it corresponds to the subspace $z_1=\cdots=z_{m_1}=0$, so that we have the decomposition $\bbC^n/\Gamma_1= \bbC^{m_1}/\Gamma_1\times W_1$ with $W_1=\bbC^{n-m_1}$.  In terms of the orbifold chart \eqref{gi.1a}, $\psi^{-1}(W_1)$ corresponds to the singular stratum in which $p_L(p)$ lies.

\begin{lemma}
The differential of $f$ is such that $\displaystyle \left. \frac{\pa f}{\pa z^i}  \right|_{W_1}=\left. \frac{\pa f}{\pa \overline{z}^i}  \right|_{W_1}=0$ for $i\in\{1,\ldots, m_1\}$.
\label{gi.3}\end{lemma}
\begin{proof}
The function $f$ is invariant under the action of $\Gamma_1$.  In particular, $\left. df\right|_{W_1}$ is invariant under the action of $\Gamma_1$, which by definition of $W_1$, holds if and only if the statement of the lemma holds.
\end{proof}

To distinguish between the factors $\bbC^{m_1}/\Gamma_1$ and $W_1$, let us set $u^i=z^i$ for $i\le m_1$ and $\eta_1^j=z^j$ for $j>m_1$.  Then by the previous lemma, the Taylor expansion of $f$ at $W_1$ is of the form
\begin{equation}
  f(u,\overline{u},\eta_1,\overline{\eta}_1)= f_0(\eta_1,\overline{\eta}_1)+ \Hess(f)_{W_1}(u,\overline{u}) + \mathcal{O}(|u|^3),
\label{gi.3b}\end{equation}
where $\Hess(f)_{W_1}$ is the Hessian of $f$ restricted to $W_1$ and only applied to the normal bundle of $W_1$.

Instead of the coordinates $(z,v)$, one can then consider the holomorphic coordinates $(\zeta,\lambda)$ related to $(z,v)$ by
\begin{equation}
 v=\lambda_1^{q}, \; z^i=\frac{\zeta_1^i}{\lambda_1},\quad i\le m_1, \; \eta_1^j=z^j, \; j>m_1.
\label{gi.4}\end{equation}
Away from $\lambda_1=0$, this is a valid change of coordinates for $\theta< \arg \zeta_1^i< \theta+ \frac{2\pi}{q}$ for all $i$ for some fixed choice of $\theta$.

In these new coordinates, the Kähler form of the Ricci-flat Kähler cone metric takes the form
\begin{equation}
\begin{aligned}
    \omega_C &= \frac{\sqrt{-1}}{2} \pa \db \left( |\lambda_1|^2 f(\frac{\zeta_1}{\lambda_1}, \frac{\overline{\zeta_1}}{\overline{\lambda_1}}, \eta_1, \overline{\eta_1})\right) \\
                        &= \frac{\sqrt{-1}}{2}\left( \pa\db |\lambda_1|^2f_E(\frac{\zeta_1}{\lambda_1}, \frac{\overline{\zeta_1}}{\overline{\lambda_1}}, \eta_1, \overline{\eta_1})+    \pa\db \mathcal{P}\right),
\end{aligned}
\label{gi.4c}\end{equation}
where
$$
f_E(u,\overline{u},\eta_1,\overline{\eta_1}):=  f_0(\eta_1,\overline{\eta_1})+ \Hess(f)_{W_1}(u,\overline{u})
$$
and
 \begin{equation}
\begin{aligned}
   \mathcal{P} &= |\lambda_1|^2   f(\frac{\zeta_1}{\lambda_1}, \frac{\overline{\zeta_1}}{\overline{\lambda_1}},\eta_1,\overline{\eta_1}) -|\lambda_1|^2f_0(\eta_1,\overline{\eta_1})- \Hess(f)_{W_1}(\zeta_1,\overline{\zeta_1})  \\
   &= \mathcal{O}(\frac{|\zeta|^3}{|\lambda_1|}) \; \mbox{as} \; |\lambda_1|\to \infty, \; \frac{|\zeta|}{|\lambda_1|}+ |\eta_1|\le C.
\end{aligned}
 \label{gi.7}\end{equation}

\subsection{The Ricci-flat Kähler cone metric seen as a $\QAC$-metric when $\tX_{\sc}$ is of depth one}

We now suppose that $\tX_{\sc}$ is an orbifold with fibred corners of depth one.  In this case,
if $p$ lies in a singularity of relative depth 1, then the action of $\Gamma_1$ on $\bbC^{m_1}\setminus\{0\}$ is free.
Let $H_1$ denote the boundary hypersurface of  $\tX_{\sc}$ corresponding to $p$ and $H_{\max}$ the boundary hypersurface corresponding to the maximal stratum.   In terms of the coordinates \eqref{gi.4}, notice that
$x_{\max}=\frac{1}{\sqrt{1+|\zeta_1|^2}}$ is a boundary defining function for $H_{\max}$ and $x_{1}=\frac{\sqrt{1+|\zeta_1|^2}}{|\lambda_1|}$ is a boundary defining function for $H_1$.  Moreover, $\zeta_1^1,\ldots,\zeta_1^{m_1}$ are holomorphic coordinates in the interior of the fibres of $\phi_1: H_1\to S_1$ and $\arg \lambda_1, \eta_1, \overline{\eta_1}$ are coordinates on the interior of $S_1$.

With this interpretation, we can replace \eqref{gi.7} with the more precise estimate
\begin{equation}
   \mathcal{P}\in x_{\max}^{-2}x_1\CI(\tX_{\sc}) \quad \Longrightarrow \quad \pa\db \mathcal{P} \in x_1\CI(\tX_{\sc}; {}^{\phi}T^*\tX_{\sc}\wedge {}^{\phi}T^*\tX_{\sc}).
\label{gi.7pr}\end{equation}
In particular, we deduce that $\omega_E:= \frac{\sqrt{-1}}{2}\pa\db |\lambda_1|^2f_E(\frac{\zeta_1}{\lambda_1}, \frac{\overline{\zeta_1}}{\overline{\lambda_1}}, \eta_1, \overline{\eta_1})$ is also a K\"ahler form near $H_1$.

\begin{proposition} The K\"ahler metric $g_E$ associated to the K\"ahler form $\omega_E$ is a $QAC$-metric with respect to the Lie algebra of $\QAC$-vector fields induced by the choice of $x_1$ and $x_{\max}$.
\label{gi.7b}\end{proposition}
\begin{proof}
Notice first that in terms of the boundary defining functions $x_1$ and $x_{\max}$,
\begin{equation}
         d\lambda_1, d\overline{\lambda}_1, \lambda_1d\eta_1, \overline{\lambda}_1 d\overline{\eta}_1, d\zeta_{1}, d\overline{\zeta}_{1}
\label{gi.7b2}\end{equation}
is a local basis of $\QAC$-forms.   Now, the K\"ahler form $\omega_E$ is of the form
\begin{multline}
                      \omega_E= \frac{\sqrt{-1}}{2}\left( f_E d\lambda_1\wedge d\overline{\lambda_1} +  \sum_{i=1}^{m_1}\sum_{j=1}^{m_1} \left.\frac{\pa^2f}{\pa u^i\pa\overline{u}^j }\right|_{W_1} d\zeta_1^i \wedge d\overline{\zeta}_1^j  + |\lambda_1|^2 \sum_{i=m_1+1}^n\sum_{j=m_1+1}^n \frac{\pa^2 f_0}{\pa \eta_1^i \pa \overline{\eta}_1^j}d\eta_1^i\wedge d\overline{\eta}_1^j  \right. \\
                      \left. + \sum_{i=1}^{m_1} \left(  \lambda_1\frac{\pa f_0}{\pa \eta_1^i}d\eta_1^i\wedge d\overline{\lambda}_1  + \overline{\lambda}_1\frac{\pa f_0}{\pa \overline{\eta}_1^i}d\lambda_1\wedge d\overline{\eta}_1^i \right)  +   \nu  \right),
\label{gi.7bb}\end{multline} where
$$
\nu= \sum_{i=1}^{m_1} \sum_{j=m_1+1}^n \left( a_{ij} d\zeta_1^i\wedge \overline{\lambda_1}d\overline{\eta}_1^j +b_{ij} \lambda_1 d\eta_1^j\wedge d\overline{\zeta}_1^i\right) + \sum_{i=m_1}^n \sum_{j=m_1+1}^n c_{ij} |\lambda_1|^2d\eta_1^i\wedge d\overline{\eta}^j_1
$$
with
$$
     a_{ij},  b_{ij}, c_{ij}\in x_1\CI(\tX_{\sc})
$$
and with the convention that $d\eta_1^{m_1}:= d\log \lambda_1$.
By Example~\ref{mwfc.4} and the local basis of $\QAC$-forms \eqref{gi.7b2}, we therefore see that the metric $g_E$ is a $\QAC$-metric with respect to the boundary defining functions $x_1$ and $x_{\max}$.
\end{proof}
In particular, we see from \eqref{gi.7bb} that $(\Hess(f)_{W_1})_{ij}=\left.\frac{\pa^2f}{\pa u^i\pa\overline{u}^j }\right|_{W_1}$ is positive-definite.  In fact,
$$
       g_{\phi_1}:=\Hess(f)_{W_1}
$$
is a family of Euclidean metrics on the interior of the fibres of the fibre bundle $\phi_1: H_1\to S_1$, that is, a Euclidean metric for the vector bundle $\phi_1: H_1\setminus\pa H_1\to S_1$.

\begin{corollary}
The metric $g_C$ is a $\QAC$-metric which has the same restriction as $g_E$ on $H_1$, namely
\begin{multline}
f_E d\lambda_1\otimes d\overline{\lambda_1} +  \sum_{i=1}^{m_1}\sum_{j=1}^{m_1} \left.\frac{\pa^2f}{\pa u^i\pa\overline{u}^j }\right|_{W_1} d\zeta_1^i \otimes d\overline{\zeta}_1^j   + |\lambda_1|^2 \sum_{i=m_1+1}^n\sum_{j=m_1+1}^n \frac{\pa^2 f_0}{\pa \eta_1^i \pa \overline{\eta}^j_1}d\eta_1^i\otimes d\overline{\eta}^j_1 \\
+ \sum_{i=1}^{m_1} \left(  \lambda_1\frac{\pa f_0}{\pa \eta_1^i}d\eta_1^i\otimes d\overline{\lambda}_1  + \overline{\lambda}_1\frac{\pa f_0}{\pa \overline{\eta}_1^i}d\lambda_1\otimes d\overline{\eta}_1^i \right).
\end{multline}
\label{gi.7c}\end{corollary}
\begin{proof}
This is a direct consequence of  \eqref{gi.7pr} and Proposition~\ref{gi.7b}.
\end{proof}

\subsection{The Ricci-flat Kähler cone metric seen as a $\QAC$-metric when $\tX_{\sc}$ is of arbitrary depth}

More generally, if $\tX_{\sc}$ is of arbitrary depth, then given a singular point $p$, we can still use the coordinates \eqref{gi.4}, but with the difference that this time the action of $\Gamma_1$ on $\bbC^{m_1}\setminus \{0\}$ is not necessarily free.    However, to see that this model agrees with models at boundary hypersurfaces of lower relative depth in $\tX_{\sc}$, we will need to introduce more refined coordinates.

Let us first describe these more refined coordinate systems. Suppose that the action of $\Gamma_1$ on $\bbC^{m_1}\setminus\{0\}$ is not free.  Let $p_2\in (\bbC^{m_1}/\Gamma_1)\setminus\{0\}$ be given and let $\ell_2$ be the corresponding complex line passing through $p_2$ and the origin.  If $p_2$ is not singular, then it suffices to use the coordinates \eqref{gi.4} to describe the behavior at infinity of the metric $g_C$ in the direction of $\ell_2$. In this case, the discussion is essentially as before.  If instead $p_2$ is singular, then the coordinates \eqref{gi.4} are no longer appropriate to describe the metric $g_C$ at infinity in the direction of $\ell_2$.  Near $p_2\in \bbC^{m_1}/\Gamma_1$, where it is understood that we are using the coordinates $\zeta^i_1=\lambda_1 z^i$, we can introduce an orbifold chart of the form
$$
  \varphi_2: \bbC^{m_2}/\Gamma_2\times \bbC\times \bbC^{\nu_2-1} \to \cU_2\subset \bbC^{m_1}/\Gamma_1,  \quad \nu_2= m_1-m_2,
$$
with $p_2= \varphi_2(0,1,0)$ and with $\Gamma_2\subset \SU(m_2)$ a finite subgroup such that $\Fix(\Gamma_2)=\{0\}$, where
$$
      \Fix(\Gamma_2)= \{ q \in \bbC^{m_2} \; | \; \gamma\cdot q=q \;\forall \gamma\in \Gamma_2\}.
$$
In particular, notice that $\Gamma_2$ can be identified with a subgroup of $\Gamma_1$, namely, with the stabilizer of a lift $\widetilde{p}_2$ of $p_2$ to $\bbC^{m_1}$ under the action of $\Gamma_1$.  We denote by $(\zeta_2,\lambda_2, z_2)$ the complex linear coordinates on each factor of $\bbC^{m_2}\times \bbC\times \bbC^{\nu_2-1}$. However, in terms of $\QAC$-geometry, the coordinates that will actually be useful are instead given by
$$
             \zeta_2,  \lambda_2, \eta_2:= \frac{z_2}{\lambda_2}.
$$
If the action of $\Gamma_2$ on $\bbC^{m_2}\setminus \{0\}$ is free, then these coordinates combined with $\lambda_1,\eta_1$ in \eqref{gi.4} will be what we need.  If the action is not free, then these coordinates will still suffice away from the singularities of $(\bbC^{m_2}/\Gamma_2)\setminus \{0\}$, but for $p_3\in (\bbC^{m_2}/\Gamma_2)\setminus \{0\}$ a singular point, the coordinates must be refined again in the direction of the complex line $\ell_3$ passing through $p_3$ and the origin. Even once these coordinates have been introduced, we may still have to repeat this step finitely many times before we have an adequate description of $g_C$ in all directions at infinity. More precisely, if $\ell$ is the relative depth of the point $p$, then we might have to apply this procedure up to $\ell-1$ times to the coordinates \eqref{gi.4} depending on the direction at infinity in which we wish to look.

Thus, in general, we will have a finite sequence of subgroups
$$
        \Gamma_\ell\subset \Gamma_{\ell-1}\subset \cdots \subset \Gamma_2\subset \Gamma_1
$$
with $\Gamma_i\subset \SU(m_i)$ acting linearly on $\bbC^{m_i}$ in such a way that $\Fix(\Gamma_i)=\{0\}$, and orbifold charts
$$
     \varphi_i: \bbC^{m_i}/\Gamma_i\times \bbC\times \bbC^{\nu_i-1}\to \cU_i\subset \bbC^{m_{i-1}}/\Gamma_{i-1}.
$$
For such a chart $(\zeta_i,\lambda_i, z_i)\in \bbC^{m_i}/\Gamma_i\times \bbC\times \bbC^{\nu_i-1}$, we consider instead the projectivised coordinates
$$
         \zeta_i, \lambda_i, \eta_i:= \frac{z_i}{\lambda_i}
$$
with the recursive relation $\zeta_{i-1}= \varphi_i(\zeta_i,\lambda_i, \lambda_i\eta_i)$.  Combining these yields the holomorphic coordinates
\begin{equation}
 (\zeta_{\ell}, \lambda_{\ell}, \eta_{\ell}, \lambda_{\ell-1},\eta_{\ell-1},\ldots,\lambda_1, \eta_1)\in \bbC^{m_{\ell}}/\Gamma_{\ell}\times(\bbC\times \bbC^{\nu_{\ell-1}})\times \cdots \times (\bbC\times \bbC^{\nu_1-1}).
\label{gid.1}\end{equation}
Suppressing the reference to the map $\varphi_i$ of the orbifold chart to lighten notation, we relate these new coordinates to the coordinates $(\zeta_1,\lambda_1,\eta_1)$ by
$$
      \zeta_1=(\zeta_{\ell},\lambda_{\ell}, \lambda_{\ell}\eta_{\ell},\lambda_{\ell-1}, \lambda_{\ell-1}\eta_{\ell-1},\ldots, ,\lambda_{2}, \lambda_{2}\eta_{2}), \; \lambda_1=\lambda_1, \; \eta_1=\eta_1.
$$
This should be compared with the real coordinate system of \eqref{mwfc.2b} with $\frac{1}{|\lambda_i|}$ playing the role of $v_i$, $\arg(\lambda_i),\eta_i$ playing the role of $y_i$, and since the function $\frac{1}{|\zeta_{\ell}|}$ can be used as a boundary defining function $x_{\ell+1}$ of the maximal boundary hypersurface $H_{\ell+1}:= H_{\max}$, $\frac{\zeta_{\ell}}{|\zeta_{\ell}|}$ plays the role of $y_{\ell+1}$.  Notice that this labeling of the boundary hypersurfaces is compatible with the partial order given by the orbifold with fibred corners structure in that
$$
         H_{i}< H_{j} \; \Longrightarrow \; i<j.
$$

\noindent We can thus pick
$$
    x_{\ell}= \frac{v_{\ell}}{x_{\ell+1}}=\frac{|\zeta_{\ell}|}{|\lambda_{\ell}|}, \quad x_i= \frac{v_i}{v_{i+1}}=\frac{|\lambda_{i+1}|}{|\lambda_{i}|}, \; i<\ell,
$$
as boundary defining functions for the other boundary hypersurfaces $H_{\ell}, \ldots, H_{1}$ involved in this coordinate system.  

Moreover, the coordinate system \eqref{gid.1} induces the coordinates 
$$
(\zeta_{\ell}, \lambda_{\ell},\eta_{\ell}, \ldots, \lambda_{i+1},\eta_{i+1}, \arg{\lambda_i},\eta_i,\frac{\lambda_{i-1}}{|\lambda_i|},\eta_{i-1},\ldots, \frac{\lambda_1}{|\lambda_i|}, \eta_1)
$$ on $H_i$, and in terms of these coordinates
 the projection $\phi_i: H_i\to S_i$ is given by
$$
  \phi_i(\zeta_{\ell}, \lambda_{\ell},\eta_{\ell}, \ldots, \lambda_{i+1},\eta_{i+1}, \arg{\lambda_i},\eta_i,\frac{\lambda_{i-1}}{|\lambda_i|},\eta_{i-1},\ldots, \frac{\lambda_1}{|\lambda_i|}, \eta_1)= (\arg{\lambda_i},\eta_i,\frac{\lambda_{i-1}}{|\lambda_i|},\eta_{i-1},\ldots, \frac{\lambda_1}{|\lambda_i|}, \eta_1)\in S_i.
$$

Notice moreover that in the coordinates \eqref{gid.1}, the crepant resolution $\beta_X: \hX\to X$ is local product in the sense that it is described entirely in terms of the coordinates $\zeta_{\ell}$, so that the other coordinates of \eqref{gid.1} naturally lift to holomorphic coordinates on the crepant resolution $\hX$.  This allows us to complete the discussion of \S~\ref{cr.0} about the existence of a natural Lie algebra of $\QAC$-vector fields on $\hX_{\QAC}$.  
\begin{lemma}
The functions $\rho_i$ introduced in \eqref{cr.13} can be chosen in such a way that they lift to be smooth on $\hX_{\QAC}$.  In particular, by Lemma~\ref{cr.18}, the boundary defining functions \eqref{cr.17} induce a well-defined Lie algebra of $\QAC$-vector fields on $\hX_{\QAC}$.  
\label{bdfcr.1}\end{lemma}
\begin{proof}
Using a suitable partition of unity on $\hX_{\QAC}$, the problem becomes local on $\tX_{\sc}$, so we can work with the coordinate chart \eqref{gid.1}.  But in this chart, the choice of $\rho_i$ is equivalent to the choice of $v_i$, so we can set
$$
     \rho_i= v_i= \frac{1}{|\lambda_i|},
$$
which clearly lifts to be smooth on $\hX_{\QAC}$.
\end{proof}

Now, to describe the metric $g_C$ near $H_i$ in the coordinates \eqref{gid.1}, we can, as in Lemma~\ref{gi.3}, use the local invariance of $f$ under the action of $\Gamma_i$ to deduce that, at $W_i= \Fix(\Gamma_i)$, $f$ has a Taylor expansion of the form
\begin{equation}
 f\left(\frac{\zeta}{\lambda_1},\frac{\overline{\zeta}}{\overline{\lambda_1}}, \eta_1,\overline{\eta}_1\right)= f_{S_i}\left(\frac{\lambda_i}{\lambda_1},\frac{\overline{\lambda}_i}{\overline{\lambda}_1}, \frac{\lambda_i\eta_i}{\lambda_1},\frac{\overline{\lambda}_i\overline{\eta}_i}{\overline{\lambda}_1},\ldots, \frac{\lambda_2}{\lambda_1},\frac{\overline{\lambda}_2}{\overline{\lambda}_1}, \frac{\lambda_2\eta_2}{\lambda_1},\frac{\overline{\lambda}_2\overline{\eta}_2}{\overline{\lambda}_1}, \eta_1\right) 
   + \Hess(f)_{W_i}\left(\frac{u_i}{\lambda_1}, \frac{\overline{u}_i}{\overline{\lambda_1}}\right)+ \mathcal{O}\left(\left| \frac{u_i}{\lambda_1} \right|^3\right)
\label{gid.2}\end{equation}
with  $u_i=(\zeta_{\ell}, \lambda_{\ell}, \lambda_{\ell}\eta_{\ell}, \ldots, \lambda_{i+1}, \lambda_{i+1}\eta_{i+1})$.  In particular,
\begin{equation}
|\lambda_1|^2  f\left(\frac{\zeta}{\lambda_1},\frac{\overline{\zeta}}{\overline{\lambda_1}}, \eta_1,\overline{\eta}_1\right)= f_{E_i} + \mathcal{O}\left(\frac{|u_i|^3}{|\lambda_1|}\right)
\label{gid.3}\end{equation}
with
\begin{equation}
f_{E_i}= |\lambda_1|^2f_{S_i}\left(\frac{\lambda_i}{\lambda_1},\frac{\overline{\lambda}_i}{\overline{\lambda}_1}, \frac{\lambda_i\eta_i}{\lambda_1},\frac{\overline{\lambda}_i\overline{\eta}_i}{\overline{\lambda}_1},\ldots, \frac{\lambda_2}{\lambda_1},\frac{\overline{\lambda}_2}{\overline{\lambda}_1}, \frac{\lambda_2\eta_2}{\lambda_1},\frac{\overline{\lambda}_2\overline{\eta}_2}{\overline{\lambda}_1}, \eta_1,\overline{\eta}_1\right)+ \Hess(f)_{W_i}(u_i, \overline{u}_i).
\label{gid.4}\end{equation}
Since $\displaystyle \left|\frac{u_i}{\lambda_1}\right|= \mathcal{O}(w_i)$ with $\displaystyle w_i= \prod_{j\le i} x_i$, we see that
\begin{equation}
     \omega_{E_i}-\omega_C \in w_i\CI(\tX_{\sc}; \Lambda^2({}^{\phi}\!T\tX_{\sc})),
\label{gid.5}\end{equation}
where
\begin{equation}
\begin{aligned}
\omega_{E_i}:= & \ \frac{\sqrt{-1}}{2} \pa\db |\lambda_1|^2 f_{E_i} \\
                         = & \ \omega_{\phi_i} + \phi_i^*\omega_{S_i} + \nu_i
\end{aligned}
\label{gid.6}\end{equation}
with $ \nu_i\in w_i\CI(\tX_{\sc}; \Lambda^2({}^{\phi}\!T\tX_{\sc}))$ and
\begin{gather}
\omega_{\phi_i}:= \frac{\sqrt{-1}}2 \pa\db \Hess(f)_{W_i}(u_i,\overline{u}_i), \\
\omega_{S_i}:= \frac{\sqrt{-1}}2 \pa\db \left( |\lambda_1|^2f_{S_i}\left(\frac{\lambda_i}{\lambda_1},\frac{\overline{\lambda}_i}{\overline{\lambda}_1}, \frac{\lambda_i\eta_i}{\lambda_1},\frac{\overline{\lambda}_i\overline{\eta}_i}{\overline{\lambda}_1},\ldots, \frac{\lambda_2}{\lambda_1},\frac{\overline{\lambda}_2}{\overline{\lambda}_1}, \frac{\lambda_2\eta_2}{\lambda_1},\frac{\overline{\lambda}_2\overline{\eta}_2}{\overline{\lambda}_1}, \eta_1,\overline{\eta}_1\right) \right) .
\label{gid.6c}\end{gather}
In particular, since $\omega_C$ is a Kähler form, this implies that near $H_i$,  $\omega_{E_i}$ is the  Kähler form of a Kähler metric $g_{E_i}$ and  that $\omega_{\phi_i}$ is a Kähler form in each fibre of $\phi_i: H_i\to S_i$ of a corresponding family of Kähler metrics $g_{\phi_i}$ which are in fact Euclidean.   This also implies   that
\begin{equation}
  \left.\omega_C\right|_{H_i}= \left.  \omega_{E_i}\right|_{H_i}.
\label{gid.6b}\end{equation}
Moreover, since $\frac{w_j}{w_i}\in \CI(\tX_{\sc})$ for $i<j$, we also have that in this case
$$
  \omega_{E_i}-\omega_{E_j}= \omega_{E_i}-\omega_C  -(\omega_{E_j}-\omega_C) \in w_i\CI(\tX_{\sc}; \Lambda^2({}^{\phi}\!T\tX_{\sc})),
$$
which implies that
$$
          \left. \omega_{E_i} \right|_{H_i\cap H_j} =  \left. \omega_{E_j} \right|_{H_i\cap H_j}.
$$
\begin{lemma}
The metrics $g_{E_i}$ and the Ricci-flat Kähler cone metric $g_C$ are $\QAC$-metrics on $\tX_{\sc}$.
\label{gid.7}\end{lemma}
\begin{proof}
In terms of the holomorphic coordinates \eqref{gid.1}, a local basis of the complexification of ${}^{\phi}\!T^*\tX_{\sc}$, \cf \eqref{mwfc.2e}, is given by
$$
         d\lambda_1, d\overline{\lambda}_1, \lambda_1d\eta_1, \overline{\lambda}_1 d\overline{\eta}_1, \ldots, d\lambda_{\ell}, d\overline{\lambda}_{\ell}, \lambda_{\ell}d\eta_{\ell}, \overline{\lambda}_{\ell}d\overline{\eta}_{\ell}, d\zeta_{\ell}, d\overline{\zeta}_{\ell}.
$$
Thus, we see from \eqref{gid.5} and the local description \eqref{gid.6} that $g_{E_i}$ and $g_C$ are $\QAC$-metrics.
\label{gid.8}\end{proof}

Recall that the interior of the fibres of $\phi_i: H_i\to S_i$ are modeled on $\bbC^{m_i}/\Gamma_i$ for some finite subgroup $\Gamma_i\subset \SU(m_i)$, while the fibres of $\hphi_i: \hH_i\to S_i$ are modeled on a Kähler crepant resolution $Y_i$ of $\bbC^{m_i}/\Gamma_i$.  This suggests the following definition.
\begin{definition}
A K\"ahler $\QAC$-metric $g\in \CI_{\QCyl}(\hX; {}^{\hphi}\!T^*\hX_{\QAC}\otimes {}^{\hphi}\!T^*\hX_{\QAC})$ on $\hX$ with K\"ahler form $\omega$ is said to be \textbf{asymptotic with rate $\delta$} to the Ricci-flat Kähler cone metric $g_C$ if:
\begin{enumerate}
\item Near $\hH_{\max}$, $g- g_C\in \widehat{x}_{\max}^{\delta}\CI_{\QCyl}(\hX; {}^{\hphi}\!T^*\hX_{\QAC}\otimes {}^{\hphi}\!T^*\hX_{\QAC})$;
\item Near $\hH_i$, $\omega- (\omega_i + \hphi_i^*\omega_{S_i})\in \widehat{x}_{\max}^{\delta}\widehat{x}_{i}\CI_{\QCyl}(\hX; {}^{\hphi}\!T^*\hX_{\QAC}\otimes {}^{\hphi}\!T^*\hX_{\QAC})$
with $\omega_{S_i}$ as in \eqref{gid.6c} and with $\omega_i$ a closed $(1,1)$-form on $\hH_{i}$ which restricts on each fibre of $\hphi_i:\hH_i\to S_i$ to the K\"ahler form of a K\"ahler $\QAC$-metric asymptotic to $g_{\phi_i}$ with rate $\delta$.  Moreover, as a family of $(1,1)$-forms parametrized by $S_i$, $\omega_i$ is smooth up to $\pa S_i$.
\end{enumerate}
\label{gid.8a}\end{definition}
\begin{remark}
This definition is not circular.  Since the fibres of $\hphi_i:\hH_i\to S_i$ are of depth lower than those of $\hX_{\QAC}$, we can assume, proceeding by induction on the depth of $\hX_{\QAC}$, that the notion of a Kähler $\QAC$-metric asymptotic to $g_{\phi_i}$ with rate $\delta$ has already been defined.
\end{remark}

\section{Existence of K\"ahler $\QAC$-metrics asymptotic to the Ricci-flat Kähler cone metric}\label{kqac.0}

Before discussing examples of Kähler $\QAC$-metrics, we need to provide examples of orbifolds equipped with asymptotically conical Kähler metrics.
Let $g_C$ be a Ricci-flat Kähler cone metric defined on $C=L\setminus D$ as in \S~\ref{cr.0},  where $L$ is some holomorphic line bundle over a Kähler-Einstein Fano orbifold.  For $N>0$, set
$$
      D_N= \{\ell \in L \;| \; \|\ell \|_{h_L}\le N \}.
$$
Suppose that $X$ is a complex orbifold such that for some compact set $\cK\subset X$, $X\setminus \cK$ is biholomorphic to $L\setminus D_N$ for some $N>0$.

\begin{lemma}
Suppose that $\omega$ is a compactly supported closed $(1,1)$-form on $X$ which is positive on $\cK$.  Then there exists a Kähler form $\tomega$ on $X$ such that $(X\setminus \cK', \tomega)$ is isometric and biholomorphic to $(L\setminus D_{N'}, \omega_C)$ for some compact set $\cK'\subset X$ and some $N'>0$.
\label{kqac.1}\end{lemma}
\begin{proof}
By continuity, $\omega$ is positive in a small neighborhood $\cU$ of $\cK$.  Using the identification $X\setminus \cK\cong L\setminus D_N$, we will work on $L\setminus D_N$, so that $\omega$ will be positive on $\cU\cap(L\setminus D_N)$.  Now, let $\eta\in \CI(\bbR)$ be a non-decreasing convex function such that
$$
      \eta(t)=  \left\{ \begin{array}{ll}  t, \quad \mbox{if} \; t\ge2, \\  \frac32, \quad \mbox{if} \; t\le1,  \end{array} \right.
$$
and set $\eta_{\delta,a}(t)= \eta(\frac{t-a}{\delta})$ for $\delta>0$ and $a\in \bbR$, so that $\eta_{\delta,a}$ is also a non-decreasing convex function.  On $L\setminus D_N$, one computes that
\begin{equation}
\begin{aligned}
\frac{\sqrt{-1}}2\pa\db\eta_{\delta,a}(\|\cdot\|^{\frac{2}{q}}_{h_L})  
    &=  \frac{\sqrt{-1}}2 \eta_{\delta,a}''(\|\cdot\|^{\frac{2}{q}}_{h_L}) \pa\|\cdot\|^{\frac{2}{q}}_{h_L}\wedge \db\|\cdot\|^{\frac{2}{q}}_{h_L}  + \frac{\sqrt{-1}}2\eta_{\delta,a}' (\|\cdot\|^{\frac{2}{q}}_{h_L})\pa\db \|\cdot\|^{\frac{2}{q}}_{h_L} \\
      &\ge \frac{\sqrt{-1}}2\eta_{\delta,a}'(\|\cdot\|^{\frac{2}{q}}_{h_L}) \pa\db \|\cdot\|^{\frac{2}{q}}_{h_L} \ge 0,
 \end{aligned}
\label{kqac.2}\end{equation}
so that this $(1,1)$-form is non-negative.  Moreover, in the region where $\|\cdot\|^{\frac{2}{q}}_{h_L}\ge 2\delta+a$, it is equal to
$$
         \frac{\sqrt{-1}}2\pa\db \|\cdot\|^{\frac{2}{q}}_{h_L}= \omega_C,
$$
the Kähler form of the Ricci-flat Kähler cone metric $g_C$, whereas in the region where $\|\cdot\|^{\frac{2}{q}}_{h_L}\le \delta+a$, it vanishes.  In particular, choosing $\delta>0$ sufficiently small and $a=N^{\frac{2}{q}}$, we have that
$$
   \nu= \frac{\sqrt{-1}}2\pa\db \eta_{\delta,a}(\|\cdot\|_{h_L}^{\frac{2}{q}})
$$
is a non-negative closed $(1,1)$-form which vanishes on $D_N$, is strictly positive on $L\setminus (\cU\cap(L\setminus D_N))$, and is equal to $\omega_C$ outside a compact set.  Hence, it suffices to take
$$
   \tomega= \nu + \epsilon\omega
$$
with $\epsilon>0$ sufficiently small.
\end{proof}

We are in particular interested in the case $X=L$.  Notice then that $g_C$ is not smooth in the orbifold sense since it has a singularity along $D$, so we cannot simply take $g_C$ itself to obtain the conclusion of Lemma~\ref{kqac.1}.

\begin{proposition}
On $L$, there exists a Kähler metric smooth in the orbifold sense which is equal to $g_C$ outside a compact set.
\label{kqac.3}\end{proposition}
\begin{proof}
Let $p_{L}: L\to D$ denote the bundle projection, let $\varphi\in H^0(L;p_{L}^*L)$ denote the tautological section and equip $p_{L}^*L$ with the Hermitian metric
$$
     h_{p_{L}}:= e^{ \| \varphi\|^2_{h_L}}p_{L}^*h_L.
$$
The curvature of $(p_{L}^*L,h_{p_{L}})$ is then given by
$$
   \sqrt{-1}\Theta_{h_{p_{L}}}= -\sqrt{-1}\pa\db \| \varphi\|^{2}_{h_L} -2qp_{L}^*\omega_D,
$$
where we use formula \eqref{qr.2} for the curvature of $L$.    Clearly, the restriction of the $(1,1)$-form $\sqrt{-1}\pa\db \| \varphi\|^{2}_{h_L}$
is positive on each fibre of $p_{L}: L\to D$.  Moreover,  $p_{L}^*\omega_D$ is positive in the directions transverse to the fibres of $p_{L}:L\to D$.  Since the non-vertical part of $\sqrt{-1}\pa\db \| \varphi\|^{2}_{h_L}$ is $\mathcal{O}(\|\varphi\|_{h_L})$, we see that $-\sqrt{-1}\Theta_{h_{p_{L}}}$ is positive in $D_{N}$ for $N>0$ sufficiently small.  Replacing $h_{p_{L}}$ with $\hat{h}_{p_{L}}$ such that $\hat{h}_{p_{L}}= h_{p_{L}}$ on $D_N$ and $h_{p_L}(\varphi,\varphi)\equiv 1$ outside a compact set, we see that
$$
      \omega:= -\sqrt{-1}\Theta_{\hat{h}_{p_{L}}}
$$
is a compactly supported closed $(1,1)$-form which is positive on $D_N$.  It then suffices to apply the previous lemma with $\omega$ to obtain the desired Kähler form.
\end{proof}

Suppose now that $X$ is a complex orbifold and $g$ is a complete Kähler metric on $X$ with Kähler form $\omega$, such that there is a biholomorphism $X\setminus \cK\cong L\setminus D_N$ for some $N>0$ and some compact set $\cK\subset X$ inducing at the same time an isometry between $g$ and $g_C$. Let $\tX_{\sc}$ be the corresponding orbifold with fibred corners given by \eqref{cr.6b} and let $H_1,\ldots, H_k$ be an exhaustive list of the boundary hypersurfaces of $\tX_{\sc}$ compatible with the partial order in the sense that $H_i<H_j\Longrightarrow i<j$.  The goal of this section is to construct examples of Kähler $\QAC$-metrics on the $\QAC$-resolution $\hX_{\QAC}$ of $\tX_{\sc}$.  Our strategy will consist of gluing local models of Kähler metrics to the K\"ahler metric $g$ at places where the singularities of $\tX_{\sc}$ are resolved by a local product Kähler crepant resolution. To do this in a systematic way, we introduce a natural space $\widehat{\cX}$ on which this gluing can be performed.  This is in fact where most of the effort will be put, for once this space is defined, the gluing construction becomes very simple; see the proof of Theorem~\ref{kqac.24} below.

To introduce this space $\widehat{\cX}$, we need to work with the orbifold with fibred corners $\tX_{\sc}$.  As an orbifold, $\tX_{\sc}$ is automatically a stratified space.  Let $\Sigma_{k+1}, \ldots, \Sigma_{\ell}$ be an exhaustive list of its strata compatible with the inclusion in the sense that
$$
      \Sigma_i\subset \overline{\Sigma}_j \; \Longrightarrow \; i<j.
$$
In particular, $\Sigma_{\ell}$ is the regular stratum.  Consider then the orbifold with corners
\begin{equation}
    \tX_{\sc}\times [0,1)
\label{kqac.4}\end{equation}
and set
\begin{equation}
  \cX = [\tX_{\sc}\times [0,1); \overline{\Sigma}_{k+1}\times \{0\}, \ldots , \overline{\Sigma}_{\ell-1}\times \{0\}].
\label{kqac.5}\end{equation}
Clearly, $\cX$ is an orbifold with corners.  Some of the boundary hypersurfaces come from the lift of old hypersurfaces $H_i\times [0,1)$ to $\cX$, namely
$$
\cH_i= \overline{\beta^{-1}(\overset{\circ}{H}_i\times (0,1))} \quad \mbox{for} \; i\in \{1,\ldots, k\},
$$
where $\beta:\cX\to \tX_{\sc}\times [0,1)$ is the blow-down map.
As is clear from the definition, $\cH_i$ is also naturally equipped with a fibre bundle structure
\begin{equation}
\xymatrix{
      \cV_i \ar[r] & \cH_i \ar[d]^{\varphi_i} \\
          & S_i,
} \quad  \begin{array}{c} \\ \\ \\ i \in \{1,\ldots, k \}, \\ \end{array}
\label{kqac.7}\end{equation}
where $\cV_i$ is obtained from $\tV_i\times [0,1)$ in the same way that $\cX$ was obtained from $\tX_{\sc}\times [0,1)$, namely by blowing up the strata of $\tV_i\times \{0\}$ in $\tV_i\times [0,1)$ in order of decreasing relative depth.

The other boundary hypersurfaces of $\cX$ come from the blow-up of  $\overline{\Sigma}_i\times \{0\}$  in $\tX_{\sc}\times [0,1)$, as well as the lift of the hypersurface $\tX_{\sc}\times \{0\}$.  Thus, for each $i\in\{k+1,\ldots,\ell\}$, $\cX$ has a boundary hypersurface $\cH_i$ associated to the stratum $\Sigma_i$.   As in the proof of Proposition~\ref{cr.6a}, this boundary hypersurface $\cH_i$ is naturally equipped with a fibre bundle structure
\begin{equation}
\xymatrix{
      \tV_i\ar[r] & \cH_i \ar[d]^{\varphi_i} \\
          & \tSigma_i,
} \quad \begin{array}{c} \\ \\ \\ i \in \{k+1,\ldots, \ell \}, \\ \end{array}
\label{kqac.8}\end{equation}
where $\tSigma_i$ is the manifold with fibred corners that resolved the stratified space $\overline{\Sigma}_i$ and $\tV_i$ is obtained from $\overline{V}_i= \overline{\bbC^{n+1-\dim_{\bbC}\Sigma_i}/\Delta_i}$ by blowing up the singular strata of $\pa \overline{V}_i$ in $\overline{V}_i$ in an order compatible with the partial order of the strata of $\pa\overline{V}_i$, where $\Delta_i\subset \GL(n+1-\dim_{\bbC}\Sigma_i,\bbC)$ is some finite subgroup.    In particular, the base $\tSigma_i$ is smooth as a manifold with fibred corners and only the fibres $\tV_i$ remain with orbifold singularities, \cf~the fibre bundle  \eqref{cr.7}.  Notice also that in the case that $i=\ell$, we have that $\tSigma_{\ell}=\cH_{\ell}$ and $\varphi_i$ is just the identity map.

\begin{lemma}
The orbifold with corners $\cX$ is in fact an orbifold with fibred corners $(\cX,\varphi)$, where  $\varphi=(\varphi_1,\ldots, \varphi_{\ell})$ is the collection of fibre bundle maps given by
\eqref{kqac.7} and \eqref{kqac.8}.
\label{kqac.9}\end{lemma}
\begin{proof}
It suffices to observe that for the partial order on the boundary hypersurfaces given by
\begin{equation}
  \cH_i< \cH_j \; \Longleftrightarrow \; i<j \; \mbox{and} \; \cH_i\cap \cH_j\ne \emptyset,
\label{kqac.10}\end{equation}
all the compatibility conditions of Definition~\ref{mwfc.1} are satisfied for the collection of bundle maps $\varphi=(\varphi_1,\ldots, \varphi_{\ell})$.
\end{proof}
As for $\tX_{\sc}$, there are natural choices of boundary defining functions for the boundary hypersurfaces of $\cX$. Indeed, for $\cH_i$ with $i\le k$, we simply take
$$
      r_i:= \beta^*\pr_1^* x_i,
$$
where $\pr_1: \tX_{\sc}\times [0,1)\to \tX_{\sc}$ is the projection on the first factor and $x_i\in \CI(\tX_{\sc})$ is a boundary defining function for $H_i$ as specified in \S~\ref{cr.0}.  For $i>k$, we proceed as follows.  Consider the partial blow-down map
$$
      \beta_i: \cX\to [\tX_{\sc}\times [0,1); \overline{\Sigma}_{k+1}\times \{0\},\ldots, \overline{\Sigma}_i\times \{0\}]
$$
and denote by $B_i$ the boundary hypersurface of $[\tX_{\sc}\times [0,1); \overline{\Sigma}_{k+1},\ldots, \overline{\Sigma}_i]$ given by the lift of $\tX_{\sc}\times \{0\}$ with the convention that
$B_k:=\tX_{\sc}\times \{0\}$ in $\tX_{\sc}\times [0,1)$.  Let also $W_i$ denote the boundary hypersurface of $[\tX_{\sc}\times [0,1); \overline{\Sigma}_{k+1}\times \{0\},\ldots, \overline{\Sigma}_i\times \{0\}]$ corresponding to the blow-up of $\overline{\Sigma}_i\times \{0\}$.
Set 
\begin{equation}
\rho_k :=\pr_2 \in \CI(\tX_{\sc}\times [0,1)),
\label{kqac.10c}\end{equation} 
where $\pr_2: \tX_{\sc}\times [0,1)\to [0,1)$ is the projection on the second factor.  Starting with $i=k+1$, choose then recursively a boundary defining function $\rho_i$ for $B_i$ such that, outside a small neighborhood $\cU_i$ of $W_i$ disjoint from the boundary hypersurfaces not intersecting $W_i$, $\rho_i$ is identified with the lift of $\rho_{i-1}$.  Then for $i>k$, we can take for $\cH_i$ the boundary defining function
\begin{equation}
    r_i := \left\{ \begin{array}{ll}\frac{\beta_{i-1}^*\rho_{i-1}}{\beta_i^*\rho_i}, & k<i<\ell, \\  \rho_i, & i=\ell,  \end{array}  \right.
\label{kqac.10b}\end{equation}
with the convention that $\beta_k:=\beta$.  
\begin{lemma}
The $\QAC$-equivalence class of the boundary defining functions $r_1,\ldots,r_{\ell}$ does not depend on the choice of the functions $\rho_i$ for $i>k$.  Hence the canonical choice \eqref{kqac.10c} yields  a natural Lie algebra of $\QAC$-vector fields $\cV_{\QAC}(\cX)$ on $\cX$.
\label{kqac.11}\end{lemma}
\begin{proof}
One proceeds as in the proof of Lemma~\ref{cr.14}.  We leave the details to the reader.
\end{proof}
Using the function $\varepsilon:= \beta^*\rho_k= \beta^*\pr_2\in \CI(\cX)$, we consider the following Lie subalgebra of $\cV_{\QAC}(\cX)$,
\begin{equation}
   \cV_{\QAC,\varepsilon}(\cX):= \{ \xi\in \cV_{\QAC}(\cX)\; | \; \xi \varepsilon\equiv 0\},
\label{kqac.12}\end{equation}
which corresponds to the Lie algebra of $\QAC$-vector fields tangent to the level sets of $\varepsilon$.  As for $\cV_{\QAC}(\cX)$, there exists a natural vector bundle $\cE\to \cX$ and a natural map $\iota_{\varepsilon}: \cE\to T\cX$ such that there is a canonical identification
$$
         \cV_{\QAC,\varepsilon}(\cX)= (\iota_{\varepsilon})_* \CI(\cX;\cE).
$$
 In fact, $\cE$ is naturally a vector subbundle of ${}^{\varphi}\! T\cX$, which induces a natural map
\begin{equation}
          {}^{\varphi}\! T^*\cX \to \cE^*.
\label{kqac.13}\end{equation}
This means in particular that a smooth $\QAC$-metric naturally restricts to define an element of $\CI(\cX;\cE^*\otimes \cE^*)$.

We are interested in the pull-back $g_{\varepsilon}:= \beta^*\pr_1^* g$ to $\cX$ of the smooth K\"ahler $\QAC$-metric $g$ on $\tX_{\sc}$.

\begin{lemma}
The pull-back $g_{\varepsilon}:= \iota_{\varepsilon}^*\beta^*\pr_1^* g$ is such that $\displaystyle \frac{g_{\varepsilon}}{\varepsilon^2}\in \CI(\cX;\cE^*\otimes \cE^*)$.
\label{kqac.14}\end{lemma}
\begin{proof}
It suffices to check that given a section $s\in \CI(\tX_{\sc};{}^{\phi}\!T^*\tX_{\sc})$, its pull-back $s_{\varepsilon}:=\iota^*_{\varepsilon}\beta^*\pr_1^*s$ is such that
$$
        \frac{s_{\varepsilon}}{\varepsilon}\in \CI(\cX;\cE^*).
$$
This can be seen by using the local basis of sections \eqref{mwfc.2e} of ${}^{\phi}\!T^*\tX_{\sc}$ on $\tX_{\sc}$.  Indeed, let us denote by
$$
\hat{v}_i:= \beta^*\pr_1^* v_i, \quad \hat{y}_{i}^{n_i}= \beta^*\pr_1^* y_i^{n_i}, \quad \hat{z}_q= \beta^*\pr_1^* z_q
$$
the pull-backs of the functions appearing in \eqref{mwfc.2e}.  Now, the function $\hat{v}_i$ should be compared with its analog on $\cX$, namely
$$
       w_i= \prod_{\cH_j\ge \cH_i} r_j = \hat{v_i} t_i \quad \mbox{with} \; t_i= \prod_{\cH_j\ge H_i, j>k} r_j.
$$
By the choice of the functions $\rho_i$ above, notice that we can assume that $t_i=\varepsilon$.  In this case, we compute that
\begin{equation}
       \frac{dw_i}{w_i^2}= \frac{d \hat{v_i}}{\varepsilon \hat{v}^2_i}+ \frac{d\varepsilon}{\hat{v_i}\varepsilon^2}\; \Longrightarrow \; \iota_{\varepsilon}^* \left( \frac{d \hat{v_i}}{\varepsilon \hat{v}_i^2}\right)=
       \iota_{\varepsilon}^*\left( \frac{dw_i}{w_i^2} \right)\in \CI(\cX;\cE^*).
\label{kqac.15}\end{equation}
Similarly,
$$
       \varepsilon^{-1} \iota_{\varepsilon}^* \beta^* \frac{dy_i^{n_i}}{v_i}= \iota_{\varepsilon}^* \frac{d\hat{y}^{n_i}_i}{w_i} \in \CI(\cX;\cE^*)
$$
 and
 $$
          \varepsilon^{-1} \iota_{\varepsilon}^* \beta^* dz_q= \iota^*_{\varepsilon}  \frac{d\hat{z}_q}{\varepsilon}\in \CI(\cX;\cE^*).
 $$
 Thus, this local computation shows that $\varepsilon^{-1}s_{\varepsilon}\in \CI(\cX;\cE^*)$, as desired.
\end{proof}
\begin{remark}
Notice that, because of the term $\frac{d\varepsilon}{\hat{v_i}\varepsilon^2}$ in \eqref{kqac.15}, $\varepsilon^{-2}g$  is not an element of 
$$\CI(\cX;{}^{\varphi}\!T^*\cX\otimes{}^{\varphi}\!T^*\cX),
$$ 
namely, it is singular as a section of ${}^{\varphi}\!T^*\cX\otimes{}^{\varphi}\!T^*\cX$ near $\beta^{-1}(\pa \tX_{\sc}\times\{0\})\subset \cX$.
\label{kqac.16}\end{remark}

To give a description of the restriction of $\varepsilon^{-2}g_{\varepsilon}$ to $\cH_i$ for $i>k$, we need first to give a description of the restriction of $\cE$ to $\cH_i$.  First, observe that the orbifold with fibred corners structure of $\cX$ naturally induces an orbifold with fibred corners structure on $\cH_i$ by considering the fibre bundle
$
      \varphi_j: \cH_j\cap \cH_i\to \tSigma_{ji}
$
on $\cH_j\cap \cH_i$ for $\cH_j>\cH_i$,
$
      \varphi_j: \cH_j\cap \cH_i\to \tSigma_j
$
on $\cH_j\cap\cH_i$ with $j>k$ and $\cH_j<\cH_i$, and
$
\varphi_j: \cH_j\cap \cH_i\to S_j
$
on $\cH_j\cap\cH_i$ with $j\le k$ and $\cH_j<\cH_i$.

Restricting the boundary defining functions of $\cX$ to $\cH_i$, we also obtain a Lie algebra of $\QFB$-vector fields and a corresponding $\QFB$-tangent bundle that we will denote by ${}^{\varphi}T\cH_i$.
Clearly, there is a natural map $\left.\cE\right|_{\cH_i}\to {}^{\varphi}T\cH_i$ and its kernel
$$
      N_i\cE:= \ker \left(\left. \cE\right|_{\cH_i}\to {}^{\varphi}T\cH_i\right)
$$
is a vector bundle on $\cH_i$.  On the other hand, $\Im \left( \left.\cE\right|_{\cH_i}\to {}^{\varphi}T\cH_i\right)$ is also a vector bundle, namely the vertical tangent bundle ${}^{\varphi}\!T(\cH_i/\tSigma_i)$ which, on each fibre $\varphi_i^{-1}(p)$ of \eqref{kqac.8}, restricts to define the $\QFB$-tangent bundle of that fibre. Consequently, there is a natural short exact sequence of vector bundles
\begin{equation}
\xymatrix{
  0\ar[r] & N_i\cE \ar[r] & \left. \cE\right|_{\cH_i} \ar[r] & {}^{\varphi}\!T(\cH_i/\tSigma_i) \ar[r] & 0.
}
\label{kqac.17}\end{equation}
Since there is a natural inclusion ${}^{\varphi}\!T(\cH_i/\tSigma_i)\subset \left.\cE\right|_{\cH_i}$, this short exact sequence splits and yields a natural decomposition
\begin{equation}
        \left.\cE\right|_{\cH_i}= N_i\cE\oplus {}^{\varphi}\!T(\cH_i/\tSigma_i).
\label{kqac.17b}\end{equation}
The vertical tangent bundle ${}^{\varphi}\!T(\cH_i/\tSigma_i)$ also naturally fits into another natural short exact sequence, namely
\begin{equation}
\xymatrix{
  0 \ar[r] & {}^{\varphi}\!T(\cH_i/\tSigma_i) \ar[r] & {}^{\varphi}\!T\cH_i \ar[r] & \varphi_i^* ({}^{\varphi}T\tSigma_i) \ar[r] & 0,
    }
\label{kqac.18}\end{equation}
where ${}^{\varphi}T\tSigma_i$ is the $\QFB$-tangent bundle of $\tSigma_i$.  Notice that this tangent bundle is well-defined thanks to the fact that the boundary defining functions of $\cX$ are compatible with the collection of fibre bundle maps $\varphi$; see the discussion just below Definition~\ref{mwfc.1b}. In particular, this yields the natural identification
$$
    \varphi_i^* ({}^{\varphi}T\tSigma_i) ={}^{\varphi}\!T\cH_i/{}^{\varphi}\!T(\cH_i/\tSigma_i).
$$
On the other hand, multiplication by the boundary defining function $r_i$ induces the identification
$$
          {}^{\varphi}\!T\cH_i/{}^{\varphi}\!T(\cH_i/\tSigma_i)\cong  \left. N_i\cE \right|_{\cH_i},
 $$
 so that  there is a canonical identification 
\begin{equation}
    \left. N_i\cE \right|_{\cH_i}\cong \varphi_i^* ({}^{\varphi}T\tSigma_i).
\label{kqac.19}\end{equation}
Summing up, we have a canonical decomposition
\begin{equation}
\left.\cE\right|_{\cH_i}=  \varphi_i^*({}^{\varphi}T\tSigma_i)  \oplus {}^{\varphi}\!T(\cH_i/\tSigma_i).
\label{kqac.19b}\end{equation}
Now, proceeding as in \S~\ref{gi.0} but replacing $|\lambda_1|$ with $\varepsilon$, we can show that in terms of this decomposition, the restriction of $\varepsilon^{-2}g_{\varepsilon}$ to $\cH_i$  takes the form
\begin{equation}
    \left.  \frac{g_{\varepsilon}}{\varepsilon^2}\right|_{\cH_i}= g_{\varphi_i} + \varphi_i^* g_{\tSigma_i},
\label{kqac.20}\end{equation}
where $g_{\tSigma_i}\in \CI(\tSigma_i; N^*\tSigma_i \otimes N^*\tSigma_i)$ and where $g_{\varphi_i}$ is on each fibre $\tV_i$ of \eqref{kqac.8} a $\QAC$-metric induced by a corresponding Euclidean metric on
$\overline{V}_i$.  In terms of the form $\omega_{\varepsilon}=\beta^*\pr_1^*\omega$, we have a corresponding decomposition
\begin{equation}
    \left.  \frac{\omega_{\varepsilon}}{\varepsilon^2}\right|_{\cH_i}= \omega_{\varphi_i} + \varphi_i^* \omega_{\tSigma_i},
\label{kqac.20}\end{equation}
where $\omega_{\varphi_i}$ and $\omega_{\tSigma_i}$ are closed $(1,1)$-forms.

We can finally introduce the space of deformations that will allow us to obtain Kähler $\QAC$-metrics on $\hX_{\QAC}$.

\begin{lemma}
The local product Kähler crepant resolution $\hX_{\QAC}$ of $\tX_{\sc}$ extends to give a resolution of $\cX$ by a manifold with fibred corners $\widehat{\cX}$.
\label{kqac.21}\end{lemma}
\begin{proof}
The proof is similar to the proof of Theorem~\ref{cr.8} and is left as an exercise.
\end{proof}
On the resolution $\widehat{\cX}$, the boundary hypersurface $\cH_i$ is replaced by a boundary hypersurface $\widehat{\cH}_i$ that is a resolution of $\cH_i$.  Moreover, the fibre bundles \eqref{kqac.7} and \eqref{kqac.8} are replaced by
\begin{equation}
\xymatrix{
      \mathcal{Y}_i \ar[r] & \widehat{\cH}_i \ar[d]^{\hvarphi_i} \\
          & S_i,
} \quad  \begin{array}{c} \\ \\ \\ i \in \{1,\ldots, k \}, \\ \end{array}
\label{kqac.22}\end{equation}
where $\mathcal{Y}_i$ is a local product crepant resolution of $\cV_i$ and
\begin{equation}
\xymatrix{
      \mathcal{Y}_i\ar[r] & \widehat{\cH}_i \ar[d]^{\hvarphi_i} \\
          & \tSigma_i,
} \quad \begin{array}{c} \\ \\ \\ i \in \{k+1,\ldots, \ell \}, \\ \end{array}
\label{kqac.23}\end{equation}
with $\mathcal{Y}_i$ a local product crepant resolution of $\tV_i$.  Notice also that the function $\varepsilon$ on $\cX$ naturally extends to a smooth function on $\widehat{\cX}$, which we also denote by $\varepsilon$.  Similarly, as in Lemma~\ref{bdfcr.1}, the boundary defining functions $r_i$ defined in \eqref{kqac.10b} can be chosen to lift to smooth boundary defining functions on $\widehat{\cX}$, yielding a natural Lie algebra $\cV_{\QAC}(\widehat{\cX})$ of $\QAC$-vector fields.    Hence, we can introduce a Lie subalgebra of $\cV_{\QAC}(\widehat{\cX})$, namely,
\begin{equation}
   \cV_{\QAC,\varepsilon}(\widehat{\cX}):= \{ \xi\in \cV_{\QAC}(\widehat{\cX})\; | \; \xi \varepsilon\equiv 0\},
\label{kqac.23b}\end{equation}
and a corresponding vector bundle $\widehat{\cE}\to \widehat{\cX}$ with a map $\iota: \widehat{\cE} \to T\widehat{\cX}$ inducing a canonical identification
$$
         \iota_*\CI(\widehat{\cX};\widehat{\cE})=  \cV_{\QAC,\varepsilon}(\widehat{\cX}).
$$
As the next theorem shows, the deformation space allows us to formulate a criterion for the existence of Kähler $\QAC$-metrics on $\hX_{\QAC}$.
\begin{theorem}
Suppose that for each $i\in \{k+1,\ldots, \ell\}$, we can find a smooth closed $(1,1)$-form $\omega_{\hvarphi_i}$ on $\widehat{\cH}_i$ that  restricts on each fibre of $\hvarphi_i: \widehat{\cH}_i\to \tSigma_i$ to the Kähler form of a $\QAC$-metric asymptotic to $\omega_{\varphi_i}$ with rate $\delta>0$ in the sense of Definition~\ref{gid.8a}. (Notice that this is trivial for $i=\ell$ since the fibres of $\hvarphi_i=\varphi_i$ are points.) Suppose moreover that the forms
$$
      \omega_i:= \omega_{\hvarphi_i} + \hvarphi_i^*\omega_{\tSigma_i}
$$
on $\widehat{\cH}_i$ are compatible in the sense that   $\left.  \omega_i\right|_{\widehat{\cH}_i\cap \widehat{\cH}_j}= \left.  \omega_j\right|_{\widehat{\cH}_i\cap \widehat{\cH}_j}$ for all $i,j\in \{k+1,\ldots, \ell\}$.  Then $\hX_{\QAC}$ admits a smooth Kähler $\QAC$-metric asymptotic to $g_C$ with rate $\delta$.
\label{kqac.24}\end{theorem}
\begin{proof}
Thanks to the compatibility condition, we can find $\homega\in \CI(\widehat{\cX}; \widehat{\cE}^* \wedge \widehat{\cE}^*)$  such that $\homega$ is a closed $(1,1)$-form on each level set of the function $\varepsilon$ whose restriction to $\cH_i$ is $\omega_i$ for each $i\in\{k+1,\ldots,\ell\}$.  By continuity, this means that for $c>0$ sufficiently small, the restriction of $\homega$ to the level set
$\{\varepsilon=c\}\cong \hX_{\QAC}$ is positive-definite as a section of $\widehat{\cE}^* \wedge \widehat{\cE}^*$.   Since $\left. \widehat{\cE}\right|_{\{\varepsilon=c\}}\cong {}^{\hphi}\!T\hX_{\QAC}$, we see that $\left. \homega\right|_{\{\varepsilon=c\}}$ is the desired Kähler form.
\end{proof}

\begin{corollary}
If the compact K\"ahler-Einstein Fano orbifold $(D,g_D)$ has only isolated singularities of complex codimension at least two with each locally admitting a K\"ahler crepant resolution, then $D$ admits a K\"ahler crepant resolution $\hD$ and the $\QAC$-compactification $\hX_{\QAC}$ of $K_{\hD}$ admits a K\"ahler $\QAC$-metric equal to $g_C$ in a neighborhood of the maximal boundary hypersurface of $\hX_{\QAC}$, in particular, asymptotic to $g_C$ with rate $\delta$ for any $\delta>0$.
\label{kqac.25}\end{corollary}
\begin{proof}
By Proposition~\ref{kqac.3}, $K_{D}$ admits a smooth K\"ahler metric $g$ which is equal to $g_C$ outside a compact set, so we have a corresponding metric $g_{\varepsilon}=\iota_{\varepsilon}^*\beta^*\pr_1^*g$ on $\cX$ to which we can hope to apply Theorem~\ref{kqac.24}.    Now, since the singularities of $D$ are isolated, the fibre bundles $\hvarphi_i: \widehat{\cH}_i\to \tSigma_i$ are trivial and $\tSigma_i\cong \overline{\bbC}$.    Thus, using Lemma~\ref{kqac.1}, we can construct for each $i$ the forms $\omega_{\hvarphi_i}$ in the statement of Theorem~\ref{kqac.24} which are equal to $\omega_{\varphi_i}$ in a neighborhood of $\widehat{\cH}_i\cap \widehat{\cH}_{\ell}$, where $\widehat{\cH}_{\ell}$ is the maximal boundary hypersurface of $\widehat{\cX}$.  The result then follows by applying Theorem~\ref{kqac.24}.  Notice in particular that  without loss of generality, the deformation $\left. \homega\right|_{\{\varepsilon=c\}}$ can be chosen to be equal to $\omega_C$ in a neighborhood of the maximal hypersurface of $\hX_{\QAC}$.  Moreover, by restricting to the zero section of $K_{\hD}$, we also obtain a K\"ahler metric on $\hD$, as claimed.
\end{proof}

\begin{corollary}
If the compact K\"ahler-Einstein Fano orbifold $(D,g_D)$ is of the form
$$
            (D,g_D)= (D_1\times \cdots \times D_q, g_1\times \cdots \times g_q)
$$
with each $D_i$ an orbifold as in the previous corollary, then the $\QAC$-compactification $\hX_{\QAC}$ of $K_{\hD}$ admits a K\"ahler $\QAC$-metric equal to $g_C$ in a neighborhood of the maximal boundary hypersurface of $\hX_{\QAC}$, in particular, asymptotic to $g_C$ with rate $\delta$ for any $\delta>0$.
\label{kqac.26}\end{corollary}
\begin{proof}
Again, we can use Proposition~\ref{kqac.3} to construct a smooth K\"ahler metric $g$ on the orbifold $K_D$ equal to $g_C$ outside a compact set, so we have a corresponding metric $g_{\varepsilon}$ on $\cX$ to which we can hope to apply Theorem~\ref{kqac.24}.  We still also know that the fibre bundles $\hvarphi_i: \widehat{\cH}_i\to \tSigma_i$ are all trivial.  For those for which the fibres are manifolds with boundary, we can proceed as before using Lemma~\ref{kqac.1} to construct the form $\omega_{\hvarphi_i}$.  For the other fibre bundles, the fibres are $\QAC$-resolutions of spaces of the form
$$
    \bbC^{n_1}/\Gamma_1\times \cdots \times \bbC^{n_r}/\Gamma_r
$$
with $\Gamma_j\subset \SU(n_j)$ a finite subgroup acting freely on $\bbC^{n_j}\setminus \{0\},$ hence we can apply Lemma~\ref{kqac.1} on the crepant resolution of each factor to obtain the form $\omega_{\hvarphi_i}$.  It suffices to make consistent choices to ensure that the compatibility conditions of Theorem~\ref{kqac.24} are satisfied.  Hence the result follows by applying Theorem~\ref{kqac.24}.  Again, we can do this in such a way that the resulting K\"ahler form is equal to $\omega_C$ in a neighborhood of the maximal boundary hypersurface of $\hX_{\QAC}$.
\end{proof}

\section{Solving the complex Monge-Ampère equation}\label{ma.0}

Let $X$ be a Kähler orbifold as in \S~\ref{cr.0} and let $\mu$ be the complex codimension of the singular set of $X_{\sc}$.  In other words,
\begin{equation}
   \mu= \min_{H_i<H_{\max}}  {m_i}
\label{ma.1f}\end{equation}
with $m_i$ the complex dimension of the fibres of the fibre bundle $\hphi_i: \hH_i\to S_i$.
Notice that by the hypotheses of \S~\ref{cr.0}, we are assuming in particular that
\begin{equation}
  2\le \mu\le n.
\label{ma.1c}\end{equation}

Now let $\omega_0$ be the Kähler form of a Kähler $\QAC$-metric $g_0$ on $\hX_{\QAC}$ asymptotic to the Ricci-flat Kähler cone metric $g_C$ with rate $\epsilon>0$ for some $\epsilon>0$.
Notice that the examples of Kähler $\QAC$-metrics of Corollary~\ref{kqac.26} are in fact asymptotic to $g_C$ with a rate of $\alpha$ as large as we want, but their Ricci potentials do not necessarily decay near  $\pa\hX_{\QAC}\setminus \hH_{\max}$.  In fact, knowing that $g_0$ is asymptotic to $g_C$ at rate $\epsilon$ implies that its Ricci potential
\begin{equation}
     r_0:= \log\left( \frac{(\omega_0^{n+1})^p}{c_p\Omega_{\hX}^p\wedge \overline{\Omega_{\hX}^p}} \right)
\label{ma.1d}\end{equation}
is in $x_{\max}^{\epsilon}\CI_{\QCyl}(\hX)$, where $\Omega_{\hX}^p$ is the lift of $\Omega_X^p$ to the local product Kähler crepant resolution $\hX$ of $X$.  Hence, the Ricci potential decays near $H_{\max}$, but not necessarily near the other boundary hypersurfaces of $\hX_{\QAC}$.  However, by Definition~\ref{gid.8a}, we see that $r_0$ has well-defined restrictions at the other boundary hypersurfaces in the  following sense.
\begin{definition}
For $\hH_i<\hH_{\max}$, let $\CI_{\QCyl}(\hH_i/S_i)$ be the space of smooth functions on 
$$
     \hH_i\setminus \left(  \bigcup_{\hH_j>\hH_i} \hH_j\cap \hH_i\right)
$$
which restrict on each fibre $\phi_i^{-1}(s)$ of $\hphi_i: \hH_i\setminus \left(  \bigcup_{\hH_j>H_i} \hH_j\cap \hH_i\right)\to S_i$ to a function in $\CI_{\QCyl}(\hphi_i^{-1}(s))$.  Then a function $f\in x^{\alpha}_{\max}\CI_{\QCyl}(\hX)$ is said to \textbf{restrict to $\pa\hX_{\QAC}$} if for each $\hH_i<\hH_{\max}$, there exists $f_i\in x_{\max}^{\alpha}\CI_{\QCyl}(\hH_i/S_i)$ such that 
$$
f-f_i\in x_{\max}^{\alpha}x_i\CI_{\QCyl}(\hX).
$$
We denote by $x^{\alpha}_{\max}\CI_{\QCyl,r}(\hX)$ the space of functions in $x^{\alpha}_{\max}\CI_{\QCyl}(\hX)$ that restrict to $\pa\hX_{\QAC}$.
\label{rest.1}\end{definition}

To obtain a Ricci-flat $\QAC$-metric on $\hX_{\QAC}$, we then need to solve the complex Monge-Amp\`ere equation
 \begin{equation}
    \log\left( \frac{(\omega_0+ \sqrt{-1}\pa\db u)^{n+1}}{\omega_0^{n+1}} \right)= -r_0.
\label{ma.1rf}\end{equation}
 In fact, this is a particular case of the more general complex Monge-Ampère equation
  
\begin{equation}
    \log\left( \frac{(\omega_0+ \sqrt{-1}\pa\db u)^{n+1}}{\omega_0^{n+1}} \right)= f, \quad f\in x_{\max}^{\alpha}\CI_{\QCyl,r}(\hX).
\label{ma.1}\end{equation}
 Since it does not require any further work, we will solve equation \eqref{ma.1} with $f\in x_{\max}^{\alpha}\CI_{\QCyl,r}(\hX)$ not necessarily equal to $-r_0$.    

To begin, we make the simplifying assumption that 
\begin{equation}
       4\le \alpha \le 2\mu
\label{ma.1cc}\end{equation}
and that the restriction of $f$ to $\pa \hX_{\QAC}$ is zero, so that in fact $f\in x^{\alpha}_{\max}x_{\sing}\CI_{\QCyl}(\hX)$, where 
$$
x_{\sing}= \prod_{H_i < H_{\max}} x_i= \frac{x}{x_{\max}}
$$
is the product of all of the boundary defining functions except the one of the maximal boundary hypersurface.  As we shall soon see, it is always possible to reduce the problem to this simpler setting.  Now, to solve \eqref{ma.1} for $f\in x^{\alpha}_{\max}x_{\sing}\CI_{\QCyl}(\hX)$, our strategy is to modify the metric so that \eqref{ma.1} is replaced with a complex Monge-Ampère equation with a new $f$ decaying faster at infinity.  In order to do this, we follow the strategy of \cite[Lemma~2.12]{CH2013} using the following Fredholm theory result.
\begin{theorem}[\cite{DM2014}]
For all $s\in \bbN_0$ and $\gamma\in(0,1)$, the Laplacian $\Delta$ of a $\QAC$-metric on $\hX_{\QAC}$  induces an isomorphism
\begin{equation}
    \Delta:  x^{-\delta} x_{\sing}^{\tau} \cC^{s+2,\gamma}_{\QCyl}(\hX)\to  x^{2-\delta} x_{\sing}^{\tau-2} \cC^{s,\gamma}_{\QCyl}(\hX)
\label{ma.3}\end{equation}
 provided $-2n<\delta<0$ and $2-2\mu< \tau < 0$, where we recall that $\hX= \hX_{\QAC}\setminus \pa \hX_{\QAC}$.
 \label{ma.1b}\end{theorem}
 \begin{proof}
 Recall that in terms of the notation of \cite{DM2014}, $x=\rho^{-1}$ and $x_{\sing}=w_1$ is the function defined in Remark~\ref{mwfc.11}.
 Thus, by \cite[Theorem~{7.6}]{DM2014}, we know that the map \eqref{ma.3} is Fredholm.  Since $\pa\hX_{\QAC}$ is connected, we can deduce from \cite[Theorem~{6.10}]{DM2014} and the proof of \cite[Theorem~{7.6}]{DM2014} that the map \eqref{ma.3} is in fact an isomorphism.
 \end{proof}

\begin{lemma}
If $f\in x_{\max}^{\alpha} x_{\sing}\CI_{\QCyl}(\hX)$ with $\alpha$ as in \eqref{ma.1cc}, then  there exists $v\in x_{\max}^{\alpha-2}x_{\sing}\CI_{\QCyl}(\hX)$ such that $\tomega_0=\omega_0 + \sqrt{-1}\pa\db v$ is the K\"ahler form of a $\QAC$-metric $\widetilde{g}_0$ asymptotic to the Ricci-flat Kähler cone metric $g_C$ with rate $\alpha$ and 
$$
\widetilde{f}:=f-\log\left(\frac{\tomega_0^{n+1}}{\omega_0^{n+1}}\right) \in x_{\max}^{2\alpha}x_{\sing}^{3}\CI_{\QCyl}(\hX)\subset x_{\max}^{\alpha+1}x_{\sing}^{3}\CI_{\QCyl}(\hX).
$$
\label{ma.2}\end{lemma}
\begin{proof}
We follow the strategy of \cite[Lemma~2.12]{CH2013} using Theorem~\ref{ma.1b}.  As the reader will see, the inequality \eqref{ma.1c} will be used in an essential way in the proof.
Since
$$
f\in x_{\sing}x_{\max}^{\alpha}\CI_{\QCyl}(\hX)= x^{\alpha} x_{\sing}^{1-\alpha} \CI_{\QCyl}(\hX),
$$
we see, by taking $\tau= 3-\alpha$ and $\delta= 2-\alpha\le-2$ in Theorem~\ref{ma.1b}, that there exists a unique $u\in x^{-\delta} x_{\sing}^{\tau} \CI_{\QCyl}(\hX)$ such that
$$
   \Delta_{g_0}u= 2f,
$$
where $\Delta_{g_0}= g_0^{ij}\nabla_i\nabla_j$ is the Laplacian associated to $g_0$.
Since $\delta<0$ and $\tau-\delta=1>0$, $u$ is decaying at infinity, hence $\omega_0+ \sqrt{-1}\pa\db u$ is still positive-definite outside a compact set.  Moreover, we can truncate $u$ to obtain a new function $v_1$ equal to $u$ outside a compact set such that $\omega_1:=\omega_0+ \sqrt{-1}\pa\db v_1$ is positive-definite everywhere.  Now, one computes that
\begin{equation}
\begin{aligned}
   (\omega_0+ \sqrt{-1}\pa\db v_1)^{n+1} &= (1+ \frac12 \Delta_{g_0} v_1)\omega_0^{n+1} + \frac{(n+1)!}{2!(n+1-2)!} \omega_0^{n+1-2}(\sqrt{-1}\pa\db v_1)^2 + \cdots + (\sqrt{-1}\pa\db v_1)^{n+1} \\
    &= (1+f ) \omega_0^{n+1}  + x^{4-2\delta}x_{\sing}^{2\tau-4}\CI_{\QCyl}(\hX; \Lambda^{2n+2}({}^{\hphi\!}T^*\hX_{\QAC})),
 \end{aligned}
\end{equation}
which implies  that
$$
   f_1:= f- \log\left(  \frac{\omega_1^{n+1}}{\omega_0^{n+1}}\right)\in  x^{4-2\delta}x_{\sing}^{2\tau-4} \CI_{\QCyl}(\hX)= x^{2\alpha}x_{\sing}^{2-2\alpha} \CI_{\QCyl}(\hX)\subset  x^{\alpha}x_{\sing}^{2-\alpha} \CI_{\QCyl}(\hX).
 $$
In particular, we see that $f_1$ decays faster than $f$ at infinity.  Repeating the above argument  with $\omega_1$ and $f_1$ in place of $\omega_0$ and $f$, this time using the isomorphism \eqref{ma.3} with $g_1$ instead of $g_0$ and with $\delta=2-\alpha$ as before, but with $\tau= \frac72-\alpha$, we can find $v_2\in  x^{-\delta}x_{\sing}^{\tau}\CI_{\QCyl}(\hX)$ such that
$\omega_2:= \omega_1+ \sqrt{-1}\pa\db v_2$ is positive-definite with 
$$
f_2:= f- \log\left(  \frac{\omega_2^{n+1}}{\omega_0^{n+1}}\right)= f_1- \log\left(  \frac{\omega_2^{n+1}}{\omega_1^{n+1}}\right)  
  \in x^{4-2\delta}x_{\sing}^{2\tau-4}\CI_{\QCyl}(\hX)= x^{2\alpha}x_{\sing}^{3-2\alpha}\CI_{\QCyl}(\hX)=x_{\max}^{2\alpha}x_{\sing}^{3}\CI_{\QCyl}(\hX).
$$
Thus, it suffices again to take $v=v_1+v_2$ to obtain the result.

\end{proof}

For the Kähler metric $\tomega_0$ and the function $\widetilde{f}$, we can now appeal to the result of Tian-Yau \cite{Tian-Yau1991} or its parabolic version \cite{Chau-Tam} to solve the complex Monge-Ampère equation.
\label{ma.4}

\begin{theorem}
For the Kähler form $\tomega_0$ and the function $\widetilde{f}$ given by Lemma~\ref{ma.2}, the complex Monge-Amp\`ere equation
\begin{equation}
\log\left( \frac{(\tomega_0+ \sqrt{-1}\pa\db u)^{n+1}}{\tomega_0^{n+1}} \right)= -\widetilde{f}\label{ma.4a}\end{equation}
has a unique solution $u$ in $x^{\alpha-1}_{\max}x_{\sing}^2\CI_{\QCyl}(\hX)$.
\label{ma.5}\end{theorem}

\begin{proof}
This is very similar to what has been done for asymptotically conical metrics in \cite{Joyce, Goto}. We will therefore go over the argument putting emphasis on the new features.  The idea is to  apply the continuity method to
\begin{equation}
\log\left( \frac{(\tomega_0+ \sqrt{-1}\pa\db u_t)^{n+1}}{\tomega_0^{n+1}} \right)= t\widetilde{f}
\label{ma.6}\end{equation}
for $t\in[0,1]$.
That is, we will show that the set 
$$
    S=\{ s\in [0,1] \; |\; \exists u_s\in x^{\alpha-1}_{\max}x_{\sing}^3\CI_{\QCyl}(\hX) \mbox{ a solution of \eqref{ma.6} for} \; t=s\}
$$
is in fact all of $[0,1]$ by showing that it is non-empty, open and closed.  Clearly, $u_0=0$ is a solution of \eqref{ma.4a} for $t=0$, so that $S$ is non-empty.  The openness of $S$ follows from Theorem~\ref{ma.1b}.

For closedness, suppose that $[0,\tau)\subset S$ for some $0<\tau\le 1$.  We need to show that \eqref{ma.6} has a solution for $t=\tau$. For this, we need to derive a priori estimates for solutions of \eqref{ma.6}. We do this as follows.  First of all, thanks to the Sobolev inequality \eqref{ma.4b}, we can apply a Moser iteration to obtain an a priori $\cC^0$-bound on a solution $u_t$ of \eqref{ma.6}.  Yau's method then provides a uniform bound on $\sqrt{-1}\pa\db u_t$.  By the result of Evans-Krylov, this yields an a priori $\cC^{2,\gamma}$-bound on solutions, where the H\"older norm is defined in term of $\tg_0$. If $\{t_i\}$ is a strictly increasing sequence with $t_i\nearrow \tau$  and $\{u_{t_i}\}$ is a corresponding sequence of solutions of \eqref{ma.6} for $t=t_1,t_2,\ldots,$ then using the Arzela-Ascoli theorem, one can extract a subsequence that converges in $\cC^{2}_{\QAC}(\hX)$ to some function $u$.      Clearly then, $u$ is solution of \eqref{ma.6} for $t=\tau$.   Bootstrapping, we thus see that $u\in\cC^{\infty}_{\QAC}(\hX)$.

To see that $u$ is in fact in $x^{\alpha-1}_{\max}x_{\sing}^3\CI_{\QCyl}(\hX)$, we need to work slightly harder. First, using Moser iteration with weights as in \cite[\S8.6.2]{Joyce} and the fact that $\widetilde{f}\in x_{\max}^{\alpha+1}x_{\sing}^{3}\CI_{\QCyl}(\hX)\subset x^3\CI_{\QCyl}(\hX)$, we obtain an a priori bound in $x^{\nu}\cC^0(\hX)$ for some $0<\nu<3-2=1$ for the solutions $u_{t_i}$.  Strictly speaking, the argument in \cite[\S8.6.2]{Joyce} is written for ALE-metrics, but as subsequently explained in \cite[\S9.6.2]{Joyce}, since we have the Sobolev inequality \eqref{ma.4b}, the argument also works for the $\QAC$-metric $\tg_0$ and only involves minor notational changes.  This a priori bound thus implies that $u\in x^{\nu}\cC^0(\hX)\cap\CI_{\QAC}(\hX)$.

To improve the statement about the regularity of $u$, we now work directly with equation \eqref{ma.6} for $t=\tau$.  Notice first that the equation can be rewritten as

$$
    \tau \widetilde{f}= \int_0^1 \frac{\pa}{\pa t} \log\left( \frac{(\tomega_0+ t\sqrt{-1}\pa\db u)^{n+1}}{\tomega_0^{n+1}} \right)dt = \int_0^1 \left( \frac{(n+1)\tomega_{u,t}^n\wedge \sqrt{-1}\pa\db u}{\tomega^{n+1}_{u,t}} \right)dt,
 $$
where $\tomega_{u,t}= \tomega_0+ t\sqrt{-1}\pa\db u.$  In other words, the complex Monge-Ampère equation can be rewritten as
\begin{equation}
    \Delta_u u= \tau\widetilde{f},
\label{ma.6bb}\end{equation}
where
$$
       \Delta_u v= \int_0^1 \left( \frac{(n+1)\tomega_{u,t}^n\wedge \sqrt{-1}\pa\db v}{\tomega^{n+1}_{u,t}} \right)dt = \frac12 \int_0^1 \left(\Delta_{\tomega_{u,t}}v\right) dt
       $$
with $\Delta_{\tomega_{u,t}}$ the Laplacian associated to the Kähler form $\tomega_{u,t}$.

Since a $\QAC$-metric has bounded geometry by Proposition~\ref{bg.1} or \cite[Remark~2.20]{DM2014}, applying the Schauder estimate to equation \eqref{ma.6bb}, we find that in fact $u\in x^{\nu}\CI_{\QAC}(\hX)$.  Since
$x^{\nu}\cC^{1}_{\QAC}(\hX) \subset \cC^{0,\gamma}_{\QCyl}(\hX)$ for $\gamma\le \nu$ by Lemma~\ref{mwfc.19}, we see in particular that $\|\pa\db u\|_{g_0}\in \cC^{0,\gamma}_{\QCyl}(\hX)$.  Rewriting \eqref{ma.6bb} in terms of an elliptic $\QCyl$-operator, that is,
\begin{equation}
      (x_{\max}^{-2}\Delta_u) u= x_{\max}^{-2}\tau\widetilde{f},
\label{ma.6b}\end{equation}
we can, thanks to Proposition~\ref{bg.2}, apply the Schauder estimate once again and bootstrap to see that $u\in x^{\nu}\CI_{\QCyl}(\hX)$.  Finally, using the inclusion $ x_{\max}^{\alpha-1}x_{\sing}^2\CI_{\QCyl}(\hX)\subset x^{\nu}x_{\sing}^{-\nu}\CI_{\QCyl}(\hX)$, we can apply the isomorphism \eqref{ma.3} with $(\delta,\tau)$ equal to $(-\nu,-\nu)$ and $(1-\alpha,3-\alpha)$ to conclude that $u\in x^{\alpha-1}_{\max}x_{\sing}^2\CI_{\QCyl}(\hX)$.  This shows that the set $S$ is closed and completes the proof of existence.

For uniqueness, we can proceed as  in \cite[Proposition~7.13]{Aubin}, but using the isomorphism \eqref{ma.3} instead of the maximum principle.
\end{proof}
This means that for $f\in x_{\max}^{\alpha}x_{\sing}\CI_{\QCyl}(\hX)$, we can solve the original equation \eqref{ma.1}.

\begin{corollary}
When $f\in x_{\max}^{\alpha}x_{\sing}\CI_{\QCyl}(\hX)$, the complex Monge-Ampère equation \eqref{ma.1} has a unique solution $u\in x_{\max}^{\alpha-2}x_{\sing}\CI_{\QCyl}(\hX)$.
\label{ma.7}\end{corollary}
\begin{proof}
Applying Lemma~\ref{ma.2}, this amounts to solving the complex Monge-Ampère equation \eqref{ma.4a}, so existence follows from Theorem~\ref{ma.5}.  On the other hand, uniqueness follows again by proceeding as in \cite[Proposition~7.13]{Aubin}, but using the isomorphism \eqref{ma.3} instead of the maximum principle. \end{proof}

The decay of the Ricci potential at infinity, and more generally of the function $f$, is a strong assumption.  Still, since it is satisfied by the examples of the previous section in the case that $\hX_{\QAC}$ is a manifold with boundary, we will be able to relax this assumption by proceeding by induction on the depth of $\hX_{\QAC}$ and using the existence result of Corollary~\ref{ma.7}. 
\begin{proposition}
Given $f\in x_{\max}^{\beta}\CI_{\QCyl,r}(\hX)$ with $\beta\ge 4$ and $\beta\ne 2\mu$, there exists a function $u\in x^{\alpha-2}_{\max}\CI_{\QCyl,r}(\hX)$ with $\alpha=\min\{2\mu,\beta\}$ such that:
\begin{enumerate}
\item $\omega:= \omega_0+ \frac{\sqrt{-1}}2 \pa\db u$ is the K\"ahler form of a K\"ahler $\QAC$-metric $\widetilde{g}$  asymptotic to $g_C$ with rate $\alpha$;
 \item $f-\log\left(  \frac{\omega^{n+1}}{\omega_0^{n+1}}\right)\in x_{\max}^{\alpha}x_{\sing}\CI_{\QCyl}(\hX)$.
 \end{enumerate}
\label{gid.9}\end{proposition}
\begin{proof}
Let $\hH_1,\ldots,\hH_{\ell+1}$ be an exhaustive list of the boundary hypersurfaces of $\hX_{\QAC}$ compatible with the partial order in the sense that
$$
    \hH_i< \hH_j \; \Longrightarrow \; i<j.
$$
To construct $\omega$, we will recursively construct the restrictions $\left. u\right|_{\hH_i}$ proceeding in order of increasing relative depth, that is, in decreasing order with respect to the index $i$.  More precisely, for each $i\in\{1,\ldots,\ell+1\}$, we will show that there exists a function $u_i\in x_{\max}^{\alpha-2}\CI_{\QCyl,r}(\hX)$ such that $\omega+ \sqrt{-1} \pa\db u_i$ is the K\"ahler form of a $\QAC$-metric asymptotic to $g_C$ with rate $\alpha$ and with 
\begin{equation}
  f-\log\left(  \frac{(\omega_0+ \sqrt{-1}\pa\db u_i)^{n+1}}{\omega^{n+1}}\right)\in x_{\max}^{\alpha}\left( \prod_{i\le j<\ell+1} x_j \right)\CI_{\QCyl}(\hX).
\label{gid.9b}\end{equation}

For the maximal boundary hypersurface $\hH_{\ell+1}$, we simply take $u_{\ell+1}=0$.  Suppose now that for some $i$, we can find $u_{i+1} \in x_{\max}^{\alpha-2}\CI_{\QCyl,r}(\hX_{\QAC}\setminus\pa\hX_{\QAC})$ such that $\omega:=\omega_0+ \sqrt{-1} \pa\db u_{i+1}$ is the K\"ahler form of a $\QAC$-metric asymptotic to $g_C$ with rate $\alpha$ such that 
$$
 f-\log\left(  \frac{(\omega+ \sqrt{-1}\pa\db u_{i+1})^{n+1}}{\omega^{n+1}}\right)\in x_{\max}^{\alpha}\left( \prod_{i<j<\ell+1} x_j \right)\CI_{\QCyl}(\hX).
$$
In terms of  $f_i:= \left.f \right|_{\hH_i}$ and the restrictions of $\omega_0$ and $\omega$ to $\hH_i$ respectively given by
 \begin{equation}
\left. \omega_0\right|_{\hH_i}=\frac{\sqrt{-1}}2f_{E_i} d\lambda_1\wedge d\overline{\lambda}_1 + \omega_{0,i} + \phi_i^*\omega_{S_i} \quad \mbox{and} \quad  \left. \omega\right|_{\hH_i}=\frac{\sqrt{-1}}2f_{E_i} d\lambda_1\wedge d\overline{\lambda}_1 + \omega_{i} + \phi_i^*\omega_{S_i},
\label{ma.9}\end{equation}
this means that 
$$
    \widetilde{f}_i:= f_i - \log\left( \frac{\omega_i^{m_i}}{\omega_{0,i}^{m_i}} \right) \in x^{\alpha}_{\max}\left( \prod_{i<j<\ell+1} x_j \right)  \CI_{\QCyl}(\hphi_i^{-1}(s))
$$
in each fibre $\hphi_i^{-1}(s)$ of $\hphi_i: \hH_i\to S_i$, where $m_i$ is the complex dimension of the fibres of $\hphi_i$.      Moreover, $\omega_i$ is a closed $(1,1)$-form which restricts on each fibre of $\hphi_i:\hH_i\to S_i$ to the Kähler form of a Kähler $\QAC$-metric asymptotic to the Euclidean metric $g_{\hphi_i}$ with rate $\alpha$.  We then need to distinguish two cases.

\noindent\textbf{Case 1: The relative depth of $\hH_i$ is strictly bigger than $1$.} In this case, we apply Corollary~\ref{ma.7} to each fibre $\hY_i$ of $\hphi_i: \hH_i\to S_i$ to obtain a unique function
$$
v_i\in \left(  \prod_{i<j\le \ell} x_j\right)x_{\max}^{\alpha-2}\CI_{\QCyl}(\hY_i)
$$
 such that 
\begin{equation}
 \log\left( \frac{(\omega_i+\sqrt{-1}\pa\db v_i)^{n_i}}{\omega_i^{n_i}} \right)= \widetilde{f}_i.
\label{ma.9c}\end{equation} 
On $\hH_i$, this yields a function that we will also denote by $v_i$.  Notice that for all $t\in [0,1]$,
 \begin{equation}
      \omega_i+t \sqrt{-1}\pa\db(v_i)= (1-t)\omega_i + t(\omega_i + \sqrt{-1}\pa\db v_i)
 \label{ma.9b}\end{equation}
 is a convex sum of two Kähler forms, hence is itself a Kähler form.  This fact can be used to extend $v_i$ to a function $v\in  \left(  \prod_{i<j\le \ell} x_j\right) x_{\max}^{\alpha-2}\CI_{\QCyl}(\hX)$ such that
 $\tomega+\sqrt{-1}\pa\db v$ is the Kähler form of a $\QAC$-metric.  To see this, let $c_i: \hH_i\times [0,\epsilon)\to \hX_{\QAC}$ be a collar neighborhood of $\hH_i$ in $\hX_{\QAC}$ as in Lemma~\ref{mwfc.1c}, so that $x_i\circ c_i: \hH_1\times [0,\epsilon)\to [0,\epsilon)$ is the projection onto the second factor. If $\psi\in \CI(\bbR)$ is a cut-off function that takes values in $[0,1]$ with $\psi(t)\equiv 1$ for $t<1$ and $\psi(t)\equiv 0$ for $t>2$, then, given the local descriptions \eqref{ma.9} and \eqref{mwfc.2e}, it suffices to take
$$
        v= (c_i)_*( \psi(N x_i) v_i)
$$
for $N>0$ a constant chosen sufficiently large so that the $(1,1)$-form $\tomega+ \sqrt{-1}\pa\db v$ remains positive-definite.  Indeed, set $\rho_i= \prod_{i<j\le \ell+1}x_j$.  Then using local coordinates as in
\eqref{mwfc.2c}, one computes that
$$
       dx_i\in x_i\rho_i\CI_{\QCyl}(\hX;{}^{\hphi}T^*\hX_{\QAC}), \quad \pa\db x_i\in x_i\rho_i\CI_{\QCyl}(\hX;\Lambda^2({}^{\hphi}T^*\hX_{\QAC})).$$
On the other hand,
$$
\sqrt{-1}\pa\db\left( \psi(Nx_i) v_i\right)= \sqrt{-1}\psi(Nx_i)\pa\db v_i + Q
$$
with
$$
   Q=\sqrt{-1}v_i\left( N^2\psi''(Nx_i)\pa x_i\wedge \db x_i + N \psi'(Nx_i)\pa\db x_i \right) + \sqrt{-1}N\psi'(Nx_i)\left( \pa v_i\wedge \db x_i + \pa x_i \wedge \db v_i \right).
$$
Since $x_i\le 2N^{-1}$ on the support of $\psi(Nx_i)$, for $\nu>0$ to be taken small, we see that in the region $\rho_i<\nu$,
$$
   \| \sqrt{-1}\pa\db\left( \psi(Nx_i) v_i\right)\|_{g_0} \le C\nu
$$
for a constant $C$ independent of $N\ge 1$.  Taking $\nu$ small enough, we can thus ensure that $\omega+ \sqrt{-1}\pa\db v>0$ in the region $\rho_i<\nu$ whatever the choice of $N\ge 1$.  Keeping $\nu>0$ fixed, we now adjust the choice of $N\ge 1$ to ensure that $\omega+\sqrt{-1}\pa\db v$ is also positive-definite in the region $\rho_i\ge \nu$.  In this region, one easily computes that
$$
        dx_i\in x_i^2\CI_{\QCyl}(\hX;{}^{\hphi}T^*\hX_{\QAC}), \quad \pa\db x_i \in x_i^2\CI_{\QCyl}(\hX; \Lambda^2({}^{\hphi}T^*\hX_{\QAC})).
$$
Hence, since again $x_i\le 2N^{-1}$ on the support of $\psi(Nx_i)$, we see that
$$
      \| Q\|_g \le \frac{C_{\nu}}{N}  \quad \mbox{in the region} \; \rho_i\ge \nu
$$
for some constant $C_{\nu}>0$ depending on $\nu$.  Thus, $Q$ can be taken as small as we want by taking $N$ sufficiently large.  On the other hand, by construction, the term $\sqrt{-1}\psi(Nx_i)\pa\db v_i$ is not expected to be small in the region $\rho_i\ge \nu$, but by the convexity property \eqref{ma.9b}, we can still ensure that
$$
        \omega+ \sqrt{-1}\psi(Nx_i)\pa\db v_i >0 \quad \mbox{in the region} \; \rho_i\ge \nu
$$
provided that $N\ge 1$ is taken large enough. Summing up, we can therefore ensure that $\omega+ \sqrt{-1}\pa\db v>0$ provided that $N\ge 1$ is large enough, as claimed.  Notice then that by  \eqref{ma.9c}, it suffices to  take $u_i= u_{i+1}+ v\in x_{\max}^{\alpha}\CI_{\QCyl,r}(\hX)$ to obtain the desired function for which $\omega+ \sqrt{-1}\pa\db u_i$ is the Kähler form of a $\QAC$-metric asymptotic to $g_C$ with rate $\alpha$ and
$$
 f-\log\left(  \frac{(\omega+ \sqrt{-1}\pa\db u_i)^{n+1}}{\omega^{n+1}}\right)\in x_{\max}^{\alpha}\left( \prod_{i\le j<\ell+1} x_j \right)\CI_{\QCyl}(\hX),
$$ 
thereby completing the inductive step.  

\noindent\textbf{Case 2: The relative depth of $\hH_i$ is equal to $1$.}  In this case, the fibres of $\hphi_i: \hH_i\to S_i$ are necessarily manifolds with boundary and the corresponding fibrewise Kähler metrics are $\ALE$.  Instead of applying Corollary~\ref{ma.7}, we  then apply standard results about $\ALE$ metrics to find $v_i$ in \eqref{ma.9c}, namely \cite[\S8.5]{Joyce}, see also \cite[Theorem~{2.1}]{CH2013}.  In particular, when $\beta>2\mu$, it is in this step that $\beta$ is replaced by $\alpha=2\mu$.  We then proceed as in Case 1 to extend $v_i$ to a function on $\hX$  and obtain $u_i$.

\end{proof}

This leads to the main result of this section.

\begin{theorem}
Let $g_0$ be a Kähler $\QAC$-metric on $\hX_{\QAC}$ asymptotic to the Ricci-flat Kähler cone metric $g_C$ with rate $\epsilon>0$.  Let $\omega_0$ be its Kähler form.
Given $f\in x^{\beta}_{\max}\CI_{\QCyl,r}(\hX)$ with $\beta\ge 4$ such that $\beta\ne 2\mu$, there exists a unique $u\in x^{\alpha-2}_{\max}\CI_{\QCyl,r}(\hX)$ with $\alpha=\min\{2\mu,\beta\}$ solving the complex Monge-Ampère equation
$$
 \log\left( \frac{(\omega_0+ \sqrt{-1}\pa\db u)^{n+1}}{\omega_0^{n+1}} \right)= f.
 $$
\label{ma.10}\end{theorem}
\begin{proof}
The existence is obtained by combining Proposition~\ref{gid.9} and Corollary~\ref{ma.7}.  Uniqueness can be seen by proceeding as in \cite[Proposition~{7.13}]{Aubin}, but using the isomorphism \eqref{ma.3} instead of the maximum principle.
\end{proof}

Taking $f$ to be the Ricci potential of $g_0$ gives the following result.  

\begin{corollary}
Let $g_0$ with K\"ahler form $\omega_0$ be a $\QAC$-metric on $\hX_{\QAC}$ asymptotic to the Ricci-flat Kähler cone metric $g_C$ with rate $\epsilon\ge 4$ such that $\epsilon\ne 2\mu$.
Then there exists a unique $u\in x^{\alpha-2}_{\max}\CI_{\QCyl,r}(\hX)$ with $\alpha=\min\{2\mu,\epsilon\}$ solving the complex Monge-Ampère equation
$$
 \log\left( \frac{(\omega_0+ \sqrt{-1}\pa\db u)^{n+1}}{\omega_0^{n+1}} \right)= -r_0,
 $$
where $r_0$ is the Ricci potential defined in \eqref{ma.1d}.  In particular, $\omega_0+ \sqrt{-1}\pa\db u$ is the Kähler form of a Ricci-flat Kähler $\QAC$-metric.
\label{ma.11}\end{corollary}

\begin{remark} If we can take $p=1$ in the definition of $X$ and $\hX$ near equation \eqref{bih.1}, then $\Omega_{\hX}^1$ is a nowhere vanishing parallel holomorphic volume form on $\hX$ and the Ricci-flat Kähler $\QAC$-metric of Corollary~\ref{ma.11} is in fact Calabi-Yau.   
\end{remark}

\appendix

\section{More examples of K\"ahler-Einstein orbifolds admitting a crepant resolution}\label{ap.0}
\begin{center}   
\textit{by Ronan J. Conlon, Frédéric Rochon, and Lars Sektnan }
\end{center}

\medskip

It is possible to slightly widen  the situations where we can apply Theorem~\ref{kqac.24}, yielding in turn more examples to which  Theorem~\ref{ma.10} can be applied.  In this appendix, we will be interested in the case where $(D,g_D)$ is a Kähler-Einstein Fano orbifold with non-isolated singularities of depth 1 having a non-trivial normal bundle, a situation not covered by Corollary~\ref{kqac.26}.  In fact, we will be very specific and only consider orbifold singularities that are locally modelled on 
\begin{equation}
     \bbC^{n-m_i}\times (\bbC^{m_i}/\bbZ_{m_i})
\label{ap.1}\end{equation}
with the generator $e^{\frac{2\pi \sqrt{-1}}{m_i}}$ of $\bbZ_{m_i}$ acting on $\bbC^{m_i}$ via complex multiplication.  For singularities of this kind, it is  well-known that $\pi_i: K_{\bbC\bbP^{m_i-1}}\to \bbC^{m_i}/\bbZ_{m_i}$ is a crepant resolution.
\begin{theorem}
Let $(D,g_D)$ be a compact Kähler-Einstein Fano orbifold with at most depth one singularities. Assume that the isolated singularities can be resolved by a Kähler crepant resolution and that the non-isolated singularities are locally of the form \eqref{ap.1}.  Then $D$ admits a Kähler crepant resolution $\hD$ and the $\QAC$-compactification $\hX_{\QAC}$ of $K_{\hD}$ admits a Calabi-Yau $\QAC$-metric asymptotic to $g_C$ with rate $2\mu$, where $\mu$ is the complex codimension of the singular set of $D$.  
\label{ap.2}\end{theorem}
\begin{proof}
The idea is to first construct a Kähler $\QAC$-metric on $\hX_{\QAC}$ by applying Theorem~\ref{kqac.24}.  By Proposition~\ref{kqac.3}, $K_{D}$ admits a smooth K\"ahler metric $g$ which is equal to $g_C$ outside a compact set, so we have a corresponding metric $g_{\varepsilon}=\iota_{\varepsilon}^*\beta^*\pr_1^*g$ on $\cX$ to which we can hope to apply Theorem~\ref{kqac.24}.  In order to do that, we need to also provide a smooth closed $(1,1)$-form $\omega_{\hvarphi_i}$ on $\widehat{\cH}_i$ as in the statement of Theorem~\ref{kqac.24}.  For a boundary hypersurface $\widehat{\cH}_i$ associated to an isolated singularity of $D$, we can proceed exactly as in the proof of Corollary~\ref{kqac.25} to construct the form $\omega_{\hvarphi_i}$. For a singular stratum of $D$  locally modelled on $\bbC^{n-m_i}\times (\bbC^{m_i}/\bbZ_{m_i})$ as in \eqref{ap.1}, the corresponding crepant resolution is locally of the form $\bbC^{n-m_i}\times K_{\bbC\bbP^{m_i-1}}$, so that the corresponding fibre bundle $\hvarphi_i: \widehat{\cH}_i\to \tSigma_i$ is a $\overline{K}_{\bbC\bbP^{m_i-1}}$-bundle, where $\overline{K}_{\bbC\bbP^{m_i-1}}$ is the radial compactification (or equivalently in this case the $\QAC$-compactification) of $K_{\bbC\bbP^{m_i-1}}$ with respect to the corresponding Calabi-Yau cone metric.  In fact, $\tSigma_i$ itself is a fibre bundle over the corresponding singular stratum of $D$.  For convenience, we will denote this singular stratum by $\sigma_i$ and the corresponding fibre bundle  by
$$
    \nu_i: \tSigma_i\to \sigma_i, \quad \dim_{\bbC}\sigma_i=n-m_i.
$$
Now, what is important for us is that $\hvarphi_i:\widehat{\cH}_i\to \tSigma_i$ is in fact the pull-back of a $\overline{K}_{\bbC\bbP^{m_i-1}}$-bundle over $\sigma_i$,
$$
             \overline{h}_i\to \sigma_i,
$$
for $h_i\to \sigma_i$ a $K_{\bbC\bbP^{m_i-1}}$-bundle  with $h_i= \overline{h_i}\setminus \pa \overline{h_i}$.  

Alternatively, we can regard $h_i$ as a complex line bundle $L_i$ over a $\bbC\bbP^{m_i-1}$-bundle 
$\mathcal{P}_i\to \sigma_i$.  By the proof of Proposition~\ref{kqac.3}, we can therefore find a smooth closed $(1,1)$-form $\omega_i$ on the total space of $h_i$ which has compact support and is positive definite in a neighborhood of $\mathcal{P}_i$ in $h_i$.  Pulling back this form to $\widehat{\cH}_i$, one can then apply Lemma~\ref{kqac.1} fibrewise to $\hvarphi_i: \widehat{\cH}_i\to \tSigma_i$ to obtain the closed $(1,1)$-form required in Theorem~\ref{kqac.24}.  This theorem thus yields a Kähler $\QAC$-metric on $\hX_{\QAC}$ asymptotic to $g_C$ with rate $\alpha$ for any $\alpha>0$.  Restricting this metric to $\hD\subset K_{\hD}$ shows in particular that $\hD$ is a Kähler crepant resolution of $D$.  Finally, applying Theorem~\ref{ma.10}, we obtain the desired Calabi-Yau $\QAC$-metric.
\end{proof}

\begin{example}
In the previous theorem, one can take $D=\bbC\bbP^{2m-1}/\bbZ_{m}$ with the generator $e^{\frac{2\pi\sqrt{-1}}m}$ of $\bbZ_{m}$ acting on $\bbC\bbP^{2m-1}$ by
$$
e^{\frac{2\pi\sqrt{-1}}m}\cdot [z_0:\cdots: z_{2m-1}]= [e^{\frac{2\pi\sqrt{-1}}m}z_0:\cdots:e^{\frac{2\pi\sqrt{-1}}m}z_{m-1}:z_m:\cdots: z_{2m-1}].
$$
The fixed points of this action are given by the disjoint union of 
$$
    \sigma_1:= \{ [z_0:\cdots:z_{2m-1}]\in \bbC\bbP^{2m-1}\; | \; z_0=\cdots=z_{m-1}=0\}
$$
and 
$$
    \sigma_2:= \{ [z_0:\cdots:z_{2m-1}]\in \bbC\bbP^{2m-1}\; | \; z_m=\cdots=z_{2m-1}=0\},
$$
so that $D$ has two disjoint singular strata $\sigma_1$ and $\sigma_2$ each locally modelled on
$\bbC^{m-1}\times \left(\bbC^m/\bbZ_{m}\right)$.  Since $\bbZ_m$ acts isometrically on $\bbC\bbP^{2m-1}$ with respect to the Fubini-Study metric, the orbifold $D$ is naturally a  Kähler-Einstein
Fano orbifold.  
\label{ap.3}\end{example}
More generally, proceeding as in Corollary~\ref{kqac.26}, we can extend Theorem~\ref{ap.2} to apply to a Kähler-Einstein Fano orbifold $(D,g_D)$ of the form
$$
    (D,g_D)= (D_1\times \cdots\times D_q, g_1\times \cdots g_q),
$$
where each $(D_i,g_i)$ a Kähler-Einstein Fano orbifold as in the statement of Theorem~\ref{ap.2}.

\bibliography{QAC-CY}
\bibliographystyle{amsplain}

\end{document}